\documentclass[11pt]{article}

\usepackage{amssymb,amsmath,amsthm,amsfonts}
\usepackage{subfigure, hyperref, multirow, bm, color, stmaryrd, marginnote, graphicx}

\numberwithin{equation}{section}
\allowdisplaybreaks

\newcommand{\R}{\mathbb{R}}
\newcommand{\Z}{\mathbb{Z}}
\newcommand{\N}{\mathbb{N}}

\newcommand{\T}{\mathbb{T}}
\newcommand{\A}{\mathcal{A}}
\newcommand{\E}{\mathcal{E}}
\newcommand{\D}{\mathcal{D}}
\newcommand{\sR}{\mathcal{R}}

\newcommand{\bars}{\overline s}
\newcommand{\bu}{\bm{u}}
\newcommand{\bw}{\bm{w}}
\newcommand{\bx}{\bm{x}}
\newcommand{\X}{\bm{X}}
\newcommand{\be}{\bm{e}}
\newcommand{\bv}{\bm{v}}
\newcommand{\p}{\partial}
\newcommand{\ts}{\thinspace}
\newcommand{\dive}{{\rm{div}}}
\newcommand{\SB}{{\rm SB}}
\newcommand{\abs}[1]{\left\lvert #1 \right\rvert}
\newcommand{\norm}[1]{\left\lVert #1 \right\rVert}
\newcommand{\wh}[1]{\widehat{#1}}

\newcommand{\mc}[1]{\mathcal{#1}}

\newtheorem{theorem}{Theorem}[section]
\newtheorem{lemma}[theorem]{Lemma}
\newtheorem{proposition}[theorem]{Proposition}

\newtheorem{definition}[theorem]{Definition}
\theoremstyle{definition}
\newtheorem{remark}[theorem]{Remark}

\allowdisplaybreaks

\begin{document}
\title{Theoretical justification and error analysis for slender body theory}
\author{Yoichiro Mori, Laurel Ohm, Daniel Spirn
\footnote{This research was supported in part by NSF grant DMS-1620316 and DMS-1516978, awarded to Y.M., by NSF GRF grant 00039202 and a Torske Kubben Fellowship, awarded to L.O., and by NSF grant DMS-1516565, awarded to D.S. The authors thank the IMA where most of this work was performed. The authors also thank the anonymous referees whose detailed comments greatly improved the paper.}\\ \textit{\small School of Mathematics, University of Minnesota, Minneapolis, MN 55455}}
\date{\today}

\maketitle

\begin{abstract}
Slender body theory facilitates computational simulations of thin fibers immersed in a viscous fluid by approximating each fiber using only the geometry of the fiber centerline curve and the line force density along it. However, it has been unclear how well slender body theory actually approximates Stokes flow about a thin but truly three-dimensional fiber, in part due to the fact that simply prescribing data along a one-dimensional curve does not result in a well-posed boundary value problem for the Stokes equations in $\R^3$. Here, we introduce a PDE problem to which slender body theory (SBT) provides an approximation, thereby placing SBT on firm theoretical footing. The slender body PDE is a new type of boundary value problem for Stokes flow where partial Dirichlet and partial Neumann conditions are specified everywhere along the fiber surface. Given only a 1D force density along a closed fiber, we show that the flow field exterior to the thin fiber is uniquely determined by imposing a {\em fiber integrity condition}: the surface velocity field on the fiber must be constant along cross sections orthogonal to the fiber centerline. Furthermore, a careful estimation of the residual, together with stability estimates provided by the PDE well-posedness framework, allow us to establish error estimates between the slender body approximation and the exact solution to the above problem. The error is bounded by an expression proportional to the fiber radius (up to logarithmic corrections) under mild regularity assumptions on the 1D force density and fiber centerline geometry.
\end{abstract}

\tableofcontents

\section{Introduction}
Describing the motion of thin filaments immersed in a viscous fluid presents an important modeling problem in mathematical biology, engineering, and physics. Numerical simulations of slender fibers have been used to help explain the role of cilia in embryonic development \cite{smith2011mathematical} and mucous transport \cite{smith2007discrete}, simulate microtubules forming the mitotic spindle during cell division \cite{shelley2016dynamics}, understand the rheology of fiber suspensions used in creating composite materials \cite{fan1998direct, hamalainen2011papermaking, petrie1999rheology}, and explore the dynamics of swimming microorganisms \cite{gueron1997cilia, lauga2009hydrodynamics, nguyen2011action, rodenborn2013propulsion, saintillan2011emergence, spagnolie2011comparative}. Models describing the interaction between thin structures and a viscous fluid may also aid in the design and optimization of microfluidic devices \cite{avron2008geometric, becker2003self, buchmann2015flow, dreyfus2005microscopic}. \\

To handle the simulation of the large numbers of thin fibers arising in these models, many existing numerical methods rely on a classical approximation known as \emph{slender body theory}. In essence, slender body theory reduces computational costs by exploiting the thin geometry of the objects being modeled. \\

To begin, we assume that the slender fibers are immersed in low Reynolds number flow, typified by any of the following: high viscosity, very slow (creeping) flow, or flow over very small length scales. Such flows are governed by the Stokes equations \eqref{stokes}, where $\bu$ represents the fluid velocity, $p$ is the pressure, and $\mu$ is the viscosity: 
\begin{equation}
\left.
\begin{aligned}
-\mu \Delta \bu +\nabla p &= 0 \\
\dive \ts \bu &=0
\end{aligned}
\right\rbrace
\label{stokes}
 \end{equation}
 accompanied by appropriate boundary conditions. Stokes flow around solid objects in unbounded or semi-bounded domains can be represented succinctly via \emph{boundary integral equations} over the surface of the object \cite{pozrikidis1992boundary}. However, despite this explicit boundary integral representation of a solution to the Stokes system, solving integral equations over moving surfaces remains a computationally intensive task, especially when simulating tens or hundreds of individual objects. Furthermore, from a modeling perspective, specifying the surface traction at each point along the entire surface of a fiber with complicated geometry can quickly become cumbersome.  \\

Instead of treating a filament as a three-dimensional object and solving equations for its surface velocity, slender body theory approximates a thin filament with a one-dimensional force density $\bm{f}(s)$ defined along the filament centerline. The idea of modeling a thin fiber with a line distribution of fundamental singularities originated with Hancock \cite{hancock1953self}, Cox \cite{cox1970motion}, Batchelor \cite{batchelor1970slender}, Lighthill \cite{lighthill1975mathematical}, and Keller and Rubinow \cite{keller1976slender}. Later, Johnson \cite{johnson1980improved} introduced doublet corrections along the fiber centerline to come up with the integral expression \eqref{SBT2} that we regard as classical slender body theory. Since then, slender body theory has formed the basis for many numerical methods developed to model thin fibers in Stokes flow \cite{bouzarth2011modeling, bringley2008validation, cortez2005method, cortez2012slender, gotz2000interactions, shelley2000stokesian, tornberg2006numerical, tornberg2004simulating}. \\

Despite the many numerical results relying on this theory, there is a lack of rigorous error analysis for slender body theory itself. The theory is built on the assumption that, given only a force density curve $\bm{f}(s)$ along the centerline of a thin but inherently three-dimensional object, we can (approximately) solve for the fiber velocity. However, it is not possible to solve for Stokes flow in three dimensions using only data specified along a one-dimensional curve. In particular, it is not immediately obvious how to rigorously compare the slender body approximation to the actual PDE solution for Stokes flow about a 3D fiber, as it remains unclear what this ``true'' solution should be. Ideally, we should be able to define a unique notion of true solution to the slender body problem given only the force density $\bm{f}(s)$ and the fiber geometry, as this is the only information needed to build a slender body approximation.  \\

Many of the foundational papers in slender body theory compute some notion of asymptotic accuracy of the slender body approximation \cite{gotz2000interactions, johnson1980improved, keller1976slender, sellier1999stokes}. Previous studies \cite{bouzarth2011modeling} have also numerically verified the convergence of the slender body approximation as the slender body radius tends to zero, but to what exactly the approximation is converging remains unclear. Recently, Koens and Lauga \cite{koens2018boundary} derived the slender body expression as an asymptotic limit of the full boundary integral equations. However, this formulation of the slender body problem requires specifying the full two-dimensional surface traction at each point on the slender body surface in order to obtain a ``true'' solution. This notion of true solution, then, is not well-defined without specifying additional force data beyond the force-per-unit-length $\bm{f}(s)$. The question remains: is there a well-posed PDE for which slender body theory is an approximation that requires only the line force density $\bm{f}(s)$ and the fiber geometry as data? \\

In this paper, we address this question by giving meaning to a solution to the Stokes equations about a slender fiber in $\R^3$, given only one-dimensional force data $\bm{f}(s)$ and a ``fiber integrity condition" (see Section \ref{SBT_def}) common to most slender body theories. We prove well-posedness of the slender body PDE problem using only this data. Furthermore, we obtain a rigorous error estimate between the true solution and the slender body approximation both within the bulk fluid and along the fiber centerline. Note that, although many of the applications listed above deal with the dynamic problem of a fiber moving with the local fluid velocity, we consider only the static problem here. Making sense of such a solution in the static case is an important first step toward truly understanding slender body theory in the dynamic case. \\

Beyond serving as a theoretical justification for the use of slender body theory in modeling and simulation of thin fibers, our PDE framework can be applied more widely to construct slender body theories for different types of fluids. In particular, our formulation makes sense for the full Navier-Stokes equations and may serve as a first step toward a rigorous justification for models such as \cite{lim2008dynamics}. Our framework can also be used to study the case of near-intersection for multiple fibers, a regime where existing slender body theories break down due to nearby fibers introducing strong angular dependence into the velocity field near the opposing fiber centerline. 

\subsection{Slender body geometry}\label{geometric_constraints}
Before we can introduce the slender body approximation, we must precisely describe the slender geometries under consideration. \\

\begin{figure}[!h]
\centering
\includegraphics[scale=0.7]{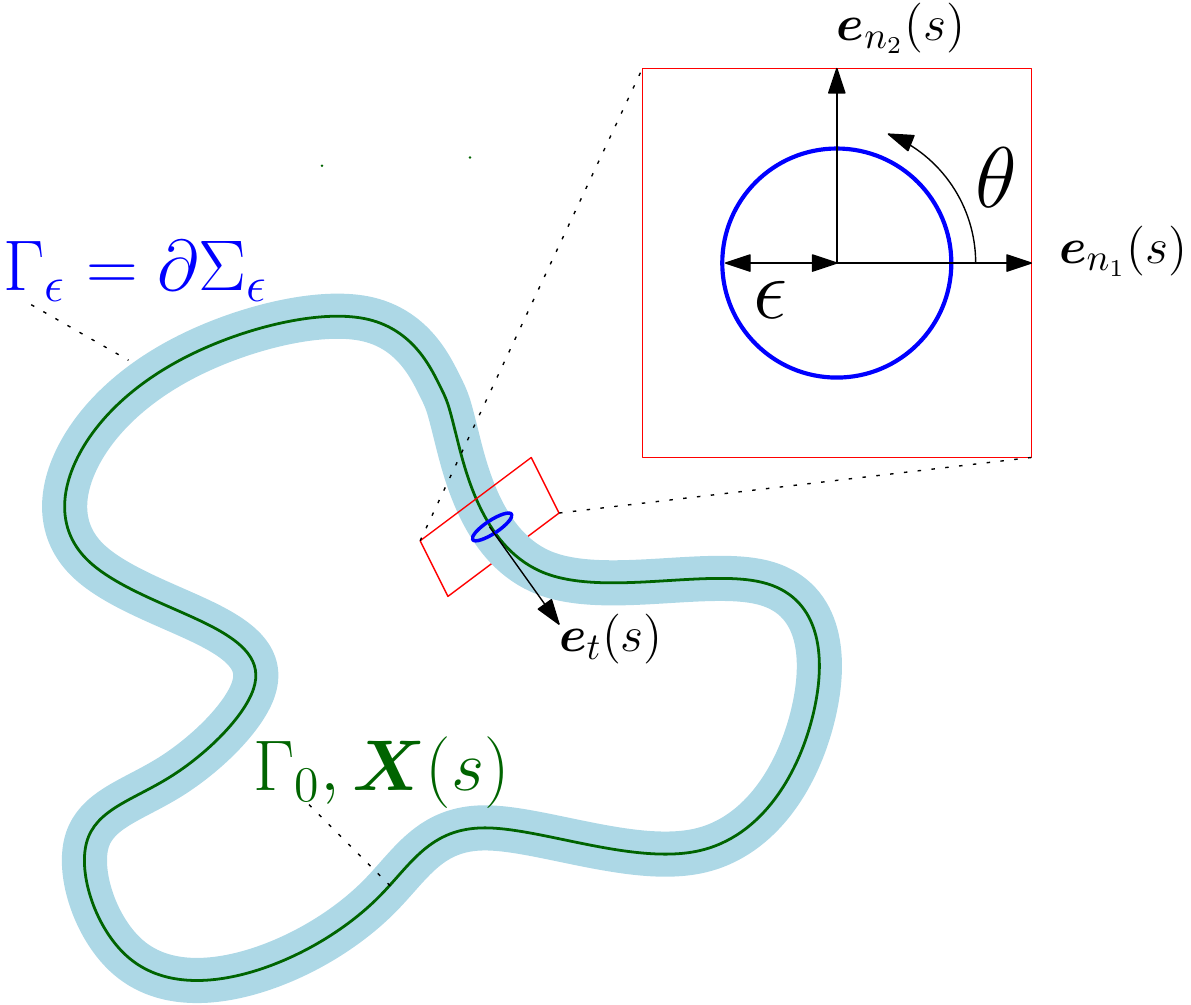}\\
\caption{The geometry of the fiber is specified via a $C^1$ orthonormal frame $\be_t(s)$, $\be_{n_1}(s)$, $\be_{n_2}(s)$. Local coordinates $\rho,\theta,s$ uniquely specify the location of a point $\bx$ in a neighborhood $\mathcal{O}$ of the slender body.}
\label{fig:coord_sys}
\end{figure}

Let $\X: \T \equiv \R/\Z \to \R^3$ denote the coordinates of a closed curve $\Gamma_0\in \R^3$, parameterized by arclength $s$ with the length of $\X$ normalized to 1. Let $C^k(\T)$, $k\in \N$, denote the space of $k$-continuously differentiable functions defined on $\T$ (we will use the same notation, without confusion, for scalar or $\R^3$-valued functions). We assume that $\X(s)\in C^2(\T)$ so that its curvature $\kappa(s) = \big| \frac{d^2\X}{d s^2} \big|$ is well-defined. \\

We assume that $\Gamma_0$ is non-self-intersecting; in particular, 
\begin{equation}\label{non_intersecting}
\inf_{s\neq t}\frac{|\X(s)-\X(t)|}{|s-t|} \ge c_{\Gamma}
\end{equation}
for some constant $c_{\Gamma}>0$. \\

For computational purposes, it will be convenient to consider a $C^1$ orthonormal frame along the slender body centerline $\Gamma_0$, periodic with respect to the arclength variable $s$. Such frames are commonly used in describing Kirchhoff rod dynamics (see \cite{antman2005nonlinear, goriely1997nonlinear} for a longer exposition). We begin by defining the tangent vector
\[\be_t(s)=\frac{d \X}{ds}. \]

We then choose a pair $\{\be_{n_1}(s),\be_{n_2}(s)\}$ of orthonormal vectors spanning the plane normal to $\be_t(s)$ at each $s\in \T$. By orthonormality, the vectors $\{\be_t,\be_{n_1},\be_{n_2}\}$ satisfy the ODE 
\begin{equation}\label{moving_ODE}
\frac{d}{ds}\begin{pmatrix}
\be_t(s) \\
\be_{n_1}(s) \\
\be_{n_2}(s) 
\end{pmatrix} = \begin{pmatrix}
0 & \kappa_1(s) & \kappa_2(s) \\
-\kappa_1(s) & 0 & \kappa_3(s) \\
-\kappa_2(s) & -\kappa_3(s) & 0 
\end{pmatrix} \begin{pmatrix}
\be_t(s) \\
\be_{n_1}(s) \\
\be_{n_2}(s) 
\end{pmatrix},
\end{equation}
where $\kappa_j$, $j=1,2,3$ are continuous functions of $s$. Note that if $\X$ is $C^3$ and the curvature $\kappa(s)$ is non-vanishing everywhere on $\T$, we can then use the simpler Frenet frame, where $\be_{n_1}(s) = \be_t'(s)/\kappa(s)$, $\kappa_1(s)=\kappa(s)$, $\kappa_2\equiv 0$, and $\kappa_3=\tau(s)$, the torsion of the curve $\X(s)$. This is useful because the ODE satisfied by the basis vectors is simpler and the coefficients have a clear geometric meaning. However, to allow for more general $C^2$ curves with possibly vanishing curvature at some points, we must refer to a frame that is well-defined when $\kappa(s)=0$.  \\

Although the geometric meaning of the general orthonormal frame coefficients $\kappa_j$ is less clear than for the Frenet frame, we note that the curvature $\kappa(s)$ of the fiber centerline always satisfies
\begin{equation}\label{kappa12}
\kappa(s)=\sqrt{\kappa_1^2(s)+\kappa_2^2(s)}.
\end{equation}

Furthermore, we may choose this orthonormal frame to satisfy the following lemma.  
 \begin{lemma}\label{lemmaorthonormal}
The coefficient $\kappa_3$ in \eqref{moving_ODE} may be made to satisfy
\begin{equation}\label{kappa3}
\kappa_3 \text{ does not depend on } s \text{ and } |\kappa_3| \le \pi.
\end{equation}
\end{lemma}
The proof of this statement is contained in Appendix \ref{moving_frame_pf}. In this construction, the orthonormal frame is almost the same as the Bishop frame \cite{bishop1975there} for open curves, except that $\kappa_3$ cannot necessarily be made to vanish for a closed curve. \\

We define
\begin{equation}\label{kappamax}
\kappa_{\max}=\max_{s\in\T} \abs{\kappa(s)}
\end{equation}
and note that, since $\X$ is a $C^2$ closed loop of length 1, we have $2\pi\le \kappa_{\max}<\infty$. \\

We also define the following cylindrical unit vectors with respect to the moving frame: 
\begin{align*}
\be_{\rho}(s,\theta) &:= \cos\theta \be_{n_1}(s) + \sin\theta\be_{n_2}(s) \\
\be_{\theta}(s,\theta) &:= -\sin\theta \be_{n_1}(s) + \cos\theta\be_{n_2}(s).
\end{align*}

Since the slender body is non-self-intersecting with $C^2$ centerline, there exists
\begin{equation}\label{rmax}
r_{\max} = r_{\max}(\kappa_{\max},c_\Gamma)
\end{equation}
such that points $\bx$ with ${\rm dist}(\bx,\X)< r_{\max}$ may be uniquely parameterized as a tube about the fiber centerline (see Figure \ref{fig:coord_sys}):
\begin{equation}\label{coordinates}
 \bx = \X(s)+\rho\be_{\rho}(s,\theta). 
 \end{equation}
 In fact, we claim that $r_{\max} \sim c_\Gamma/\kappa_{\max}$ should suffice, but we do not prove this here. \\

For $\epsilon < r_{\max}/4$, we then define a slender body $\Sigma_\epsilon$ with uniform radius $\epsilon$ by
\begin{equation}\label{slender_body}
\Sigma_{\epsilon} = \big\{\bx \in \R^3 \ts : \ts \bx= \X(s) + \rho \be_{\rho}(s,\theta), \quad \rho < \epsilon \big\}
\end{equation}
 
 We parameterize the surface of the slender body, $\Gamma_{\epsilon}=\p \Sigma_{\epsilon}$, as 
\begin{equation}\label{gamma_epsilon}
 \Gamma_{\epsilon}(s,\theta)= \X(s) + \epsilon\be_{\rho}(s,\theta). 
 \end{equation}

The surface element on $\Gamma_{\epsilon}$ is then given by
\begin{equation}\label{surface_element}
dS = \mathcal{J}_{\epsilon}(s,\theta) \ts d\theta ds, 
 \end{equation}

where we define
\begin{equation}\label{Jeps_def}
\mathcal{J}_{\epsilon}(s,\theta) := \epsilon\big(1-\epsilon(\kappa_1(s)\cos\theta+\kappa_2(s)\sin\theta) \big).
\end{equation}

We also define the neighborhood
\begin{equation}\label{region_O}
\mathcal{O} = \bigg\{\bx \in \Omega_{\epsilon} \ts : \ts \bx= \X(s) + \rho \be_{\rho}(s,\theta), \quad \epsilon < \rho<r_{\max} \bigg\}
\end{equation}
of the slender body to refer to fluid points $\bx$ near to the slender body.


\subsection{Classical slender body theory}\label{SBT_def}
With the geometric constraints specified above, we now define the corresponding slender body approximation to Stokes flow about the thin fiber. \\
 
The essential building block of slender body theory is the Stokeslet, the free-space Green's function for the Stokes equations \eqref{stokes}. The Stokeslet represents the Stokes flow in $\R^3$ resulting from a point source at $\bx_0$ of strength ${\bm g}$:
\begin{equation}
\begin{aligned}
-\mu \Delta \bu +\nabla p &= {\bm g}\delta(\bx-\bx_0) \\
\dive \ts \bu &=0 \\
|\bu| &\to 0 \quad \text{ as } |\bx| \to \infty,
\end{aligned}
\label{Stokes_Green}
\end{equation}
where $\delta(\bx)$ denotes the Dirac delta. We define the Stokeslet and its associated pressure tensor as 
\begin{align*}
\mathcal{S}(\widehat\bx) &= \frac{{\bf I}}{|\widehat\bx|} + \frac{{\widehat\bx}{\widehat\bx}^{\rm T}}{|\widehat\bx|^3}, \quad p^{S}({\widehat\bx}) = \nabla \left(\frac{1}{|\widehat\bx|}\right) = \frac{\widehat\bx}{|\widehat\bx|^3},
\end{align*}
where ${\bf I}$ is the identity matrix and $\widehat \bx = \bx-\bx_0$ (see \cite{pozrikidis1992boundary,childress1981mechanics} for a derivation). The solution to \eqref{Stokes_Green} is then given by 
\[ \bu = \frac{1}{8\pi\mu}\mc{S}(\wh \bx) \bm{g}, \quad p = \frac{1}{4\pi} p^S(\wh\bx)\cdot\bm{g}.\]

Since the singularly forced Stokes system \eqref{Stokes_Green} is linear, additional solutions may constructed by differentiating the Stokeslet and taking linear combinations of the Stokeslet and its higher-order derivatives -- dipoles, quadrupoles, octupoles, etc. Inclusion of these higher-order multipole terms in the expression of solutions to \eqref{Stokes_Green} can be useful especially in solving exterior problems, and is sometimes referred to as the method of singularities \cite{pozrikidis1992boundary}. \\

The higher-order term that plays the most important role in slender body theory, known as the doublet, is given by 
\[ \mathcal{D}( \widehat \bx) = \frac{1}{2}\Delta\mathcal{S}({\widehat \bx}) = \frac{{\bf I}}{|\widehat\bx|^3}-3\frac{{\widehat\bx}{\widehat\bx}^{\rm T}}{|\widehat\bx|^5}.\]

The idea of slender body theory is to approximate the velocity field around a thin filament in Stokes flow by integrating a superposition of Stokeslets, doublets, and possibly higher-order multipole terms along the centerline of the fiber. The slender body ansatz is given by the integral expression
\begin{equation}\label{SB_ansatz}
\bu^{\SB}(\bx) = \bu_{\infty}(\bx) + \frac{1}{8\pi\mu} \int_{\T} \bigg(\mathcal{S}(\bx-\X(t)) {\bm g}_1(t) +\mathcal{D}(\bx-\X(t)) {\bm g}_2(t)+\cdots \bigg)\ts dt,
\end{equation}
where $\bu_{\infty}$ is the undisturbed background fluid velocity, and the dots indicate the possibility of including higher-order multipole terms. The coefficients ${\bm g}_i$ of the higher-order terms are chosen to best preserve the structural integrity of the fiber (see below). \\

The simplest prescription for $\bm{g}_i$, $i=1,2,\dots$ would be to set $\bm{g}_1(t)=\bm{f}(t)$, $\bm{g}_i=0 \text{ for } i\ge 1$, where $\bm{f}(t)$ is the prescribed force density along the fiber centerline. The problem with this choice is that the surface velocity $\bu^{\SB}\big|_{\Gamma_\epsilon}(s,\theta)$ has a strong $\theta$-dependence on each constant-$s$ cross section (see left image of Figure \ref{fig:fiber_integ}). If the no-slip condition is satisfied on the fiber interface, this will lead to an instantaneous deformation of the fiber cross sectional geometry, destroying the structural integrity of the fiber. Setting $\bm{g}_2(t)=\frac{\epsilon^2}{2}\bm{g}_1(t)$ eliminates this $\theta$-dependence to leading order, so that the surface velocity is almost constant along cross sections (see right image of Figure \ref{fig:fiber_integ}). We term this $\theta$-independence constraint the {\em fiber integrity condition}. Note that the fiber integrity condition is a key feature of most slender body theories -- see, for example, \cite{tornberg2004simulating} and \cite{cox1970motion}. \\

We note that the fiber integrity constraint ignores torque and does not allow the fiber to simply rotate about its centerline. The additional consideration of torque along the fiber (explored in \cite{keller1976slender}; see also \cite{lim2004simulations}) is an extension to the classical slender body approximation \eqref{SBT2} that will be addressed in future work. \\

\begin{figure}[!h]
\centering
\includegraphics[scale=0.5]{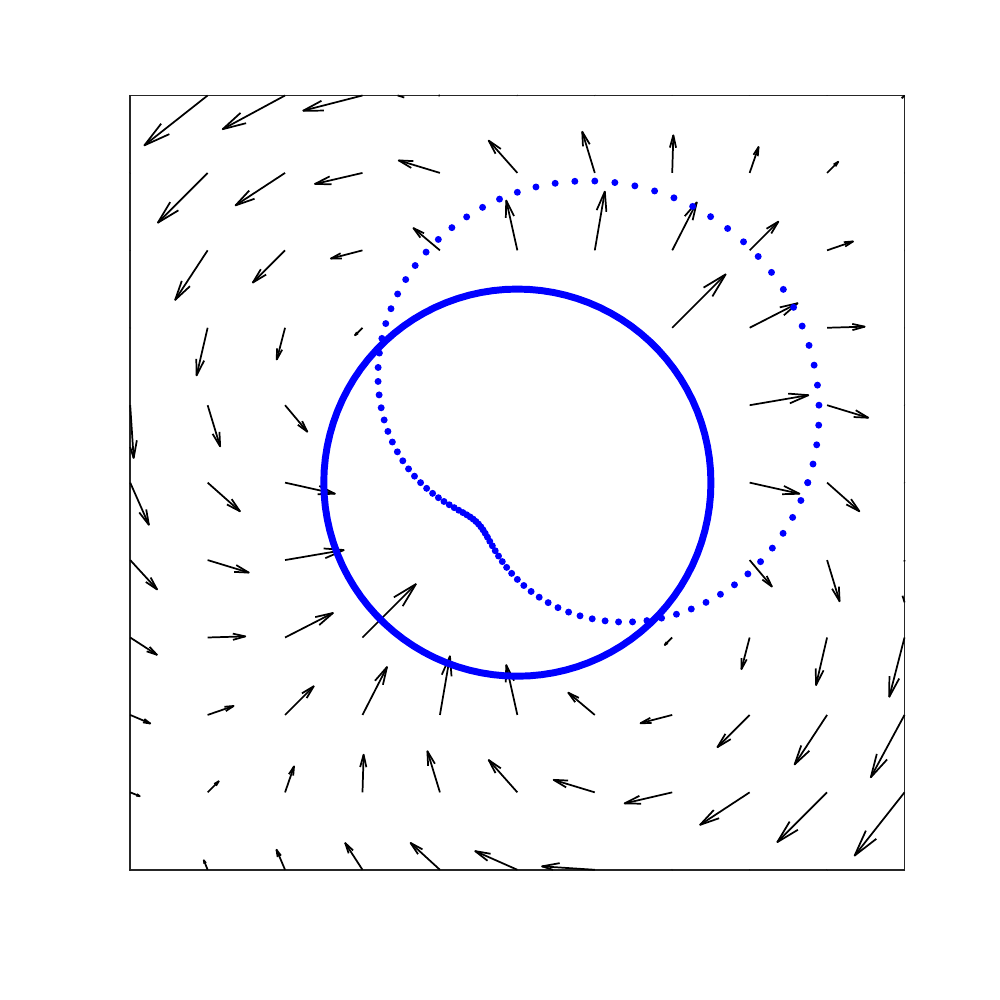}
\includegraphics[scale=0.5]{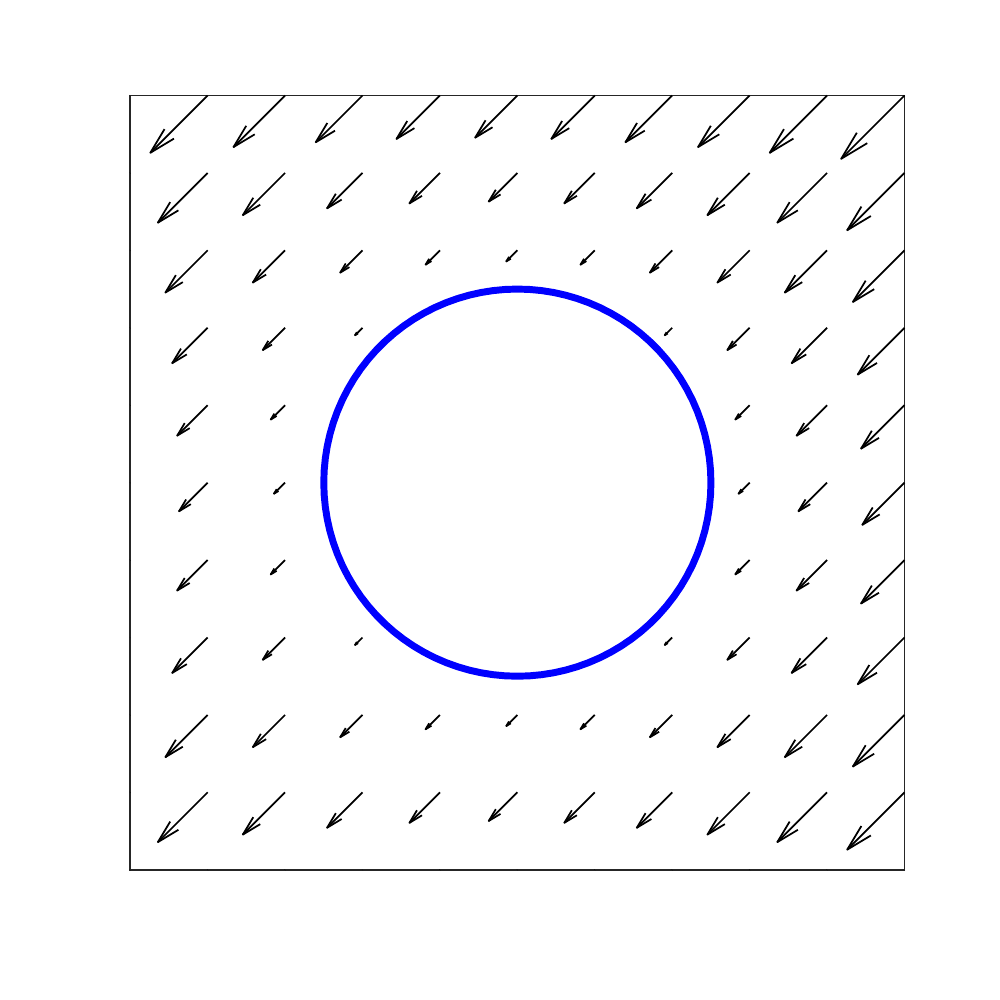}
\caption{A sketch of the reasoning behind the fiber integrity condition. If the fiber surface velocity $\bu^{\SB}\big|_{\Gamma_{\epsilon}}$ depends strongly on the angle $\theta$, the cross sectional shape of the fiber will deform in the next time instant (left image). When $\theta$-independence is imposed on the surface of each cross section (right image), we ensure the structural integrity of the fiber over time.}
\label{fig:fiber_integ}
\end{figure}

The classical (non-local) slender body approximation to the fluid velocity at a point $\bx$ away from the centerline is thus given by 
\begin{equation}\label{SBT2}
\begin{aligned}
8\pi\mu\bu^{\SB}(\bx) &= \int_{\T}\bigg( \mathcal{S}(\bm{R}) +\frac{\epsilon^2}{2}\mathcal{D}(\bm{R}) \bigg){\bm f}(t) \ts dt; \; \bm{R} = \bx-\X(t), \\
\mc{S}(\bm{R}) &= \frac{{\bf I}}{\abs{\bm{R}}} + \frac{\bm{R}\bm{R}^{\rm T}}{|\bm{R}|^3}, \quad \mc{D}(\bm{R}) = \frac{{\bf I}}{|\bm{R}|^3} - 3\frac{\bm{R}\bm{R}^{\rm T} }{|\bm{R}|^5}.
\end{aligned}
\end{equation}

The corresponding slender body approximation to the pressure in the fluid is given by 
\begin{equation}\label{SB_press0}
p^{\SB}(\bx) = \frac{1}{4\pi}\int_{\T} \frac{\bm{R}\cdot {\bm f}(t)}{|\bm{R}|^3} \ts dt. 
\end{equation}

To approximate the velocity of the slender body itself, a centerline expression $\bu^{\SB}_C(s)$ is often formulated following the matched asymptotics approach of Keller-Rubinow \cite{keller1976slender}. The expression \eqref{SBT2} is evaluated at $\rho=\epsilon$ and the resulting integral kernel $\mc{S}(s,\theta,t;\epsilon) + \frac{\epsilon^2}{2}\mc{D}(s,\theta,t;\epsilon)$ is expanded asymptotically about $\epsilon=0$ to obtain an integral equation on $\X(s)$ approximating $\bm{f}(s)$ given $\bu(s)$. For a periodic filament, the Keller-Rubinow formula (see \cite{shelley2000stokesian, cortez2012slender} for periodization of the original formula) is given by 

\begin{equation}\label{SBT_asymp}
\begin{aligned}
8\pi \mu \ts \bu^{\SB}_C(s) &= \big[({\bf I}- 3\be_t\be_t^{\rm T})-2({\bf I}+\be_t\be_t^{\rm T}) \log(\pi\epsilon/4) \big]{\bm f}(s) \\
&\qquad + \int_{\T} \left[ \left(\frac{{\bf I}}{|\bm{R}_0|}+ \frac{\bm{R}_0\bm{R}_0^{\rm T}}{|\bm{R}_0|^3}\right){\bm f}(t) - \frac{{\bf I}+\be_t(s)\be_t(s)^{\rm T} }{|\sin (\pi(s-t))/\pi|} {\bm f}(s)\right] \ts dt.
\end{aligned}
\end{equation}

Here $\bm{R}_0(s,t) := \X(s) -\X(t)$. The centerline expression \eqref{SBT_asymp} is typically used in numerical simulations to update the position of the fiber centerline. \\

Our aim is to establish a rigorous error estimate for the slender body approximation \eqref{SBT2} as well as the centerline approximation \eqref{SBT_asymp}.

\subsection{Slender body PDE formulation}
We must first determine a well-posed PDE for reconstructing a Stokes flow in $\R^3$ given only one-dimensional force data $\bm{f}(s)$. Since this total force alone is not sufficient information to uniquely solve a Stokes boundary value problem, we also impose a fiber integrity condition: the surface velocity of the fiber at each $s$ cross section must be independent of the angle $\theta$. 

We formulate the slender body problem as a boundary value problem for the Stokes system over the fluid domain $\Omega_{\epsilon}=\R^3\backslash \overline{\Sigma_{\epsilon}}$. Note that by rescaling, we can take the viscosity $\mu\equiv 1$. Let $\bm{\sigma}= \nabla \bu+(\nabla\bu)^{\rm T} -p{\bf I}$ denote the stress tensor and ${\bm n}=-\cos\theta\be_{n_1}(s)-\sin\theta\be_{n_2}(s)=-\be_{\rho}(s,\theta)$ denote the unit normal vector pointing into the slender body at each point $(s,\theta)\in \Gamma_{\epsilon}$. We define the slender body PDE as follows: 
\begin{equation}\label{exterior_stokes}
\begin{aligned}
-\Delta \bu +\nabla p &= 0, \; \dive \ts \bu = 0 \quad \text{in } \Omega_{\epsilon} = \R^3 \backslash \Sigma_{\epsilon}, \\
\int_0^{2\pi} (\bm{\sigma} {\bm n}) \ts \mathcal{J}_{\epsilon}(s,\theta) \ts d\theta &= {\bm f}(s) \hspace{1.5cm} \text{ on } \Gamma_{\epsilon}, \\
\bu|_{\Gamma_{\epsilon}} &= \bu(s) \quad \text{(unknown but independent of }\theta), \\
|\bu| \to 0 & \text{ as } |\bx|\to \infty.
\end{aligned}
\end{equation}
Here we use the expression for the Jacobian factor $\mathcal{J}_{\epsilon}(s,\theta)$ given by \eqref{Jeps_def}. In this formulation, the boundary data is specified as partial Neumann and partial Dirichlet information everywhere along the boundary $\Gamma_\epsilon$. Fiber movements are constrained by the partial Dirichlet condition $\bu\big|_{\Gamma_\epsilon}=\bu(s)$, so the fiber may bend along its centerline, but cross sections maintain their circular shape and radius $\epsilon$ over time. Since the expression for $\bu\big|_{\Gamma_\epsilon}$ is not specified beyond the $\theta$-independence, an infinite family of flows $\bu$ satisfy this constraint. The only given data in the above system is $\bm{f}:\T\to\R^3$, the one-dimensional force density along the fiber centerline. We define $\bm{f}$ to be the total surface force $(\bm{\sigma}\bm{n})\big|_{\Gamma_\epsilon}$ acting on the body over each cross section, weighted by the surface area of the fiber via $\mc{J}_\epsilon(s,\theta)$: greater surface area contributes more to the total force along the centerline; smaller surface area contributes less. To close the system, we require that the velocity $\bu$ decays to 0 as $\abs{\bx}\to\infty$.\\

Note that the boundary integral formulation in \cite{koens2018boundary} may be a more familiar representation of Stokes flow about a three-dimensional object, but assumes knowledge of the surface traction at each point over the slender body surface. In our formulation, the only data specified is the line force density $\bm{f}(s)$. Notice that the fiber integrity condition, common to all  slender body theories, then plays an essential role, allowing us to obtain a unique velocity field given only this one-dimensional force data. \\

\begin{figure}[!h]
\centering
\includegraphics[scale=0.67]{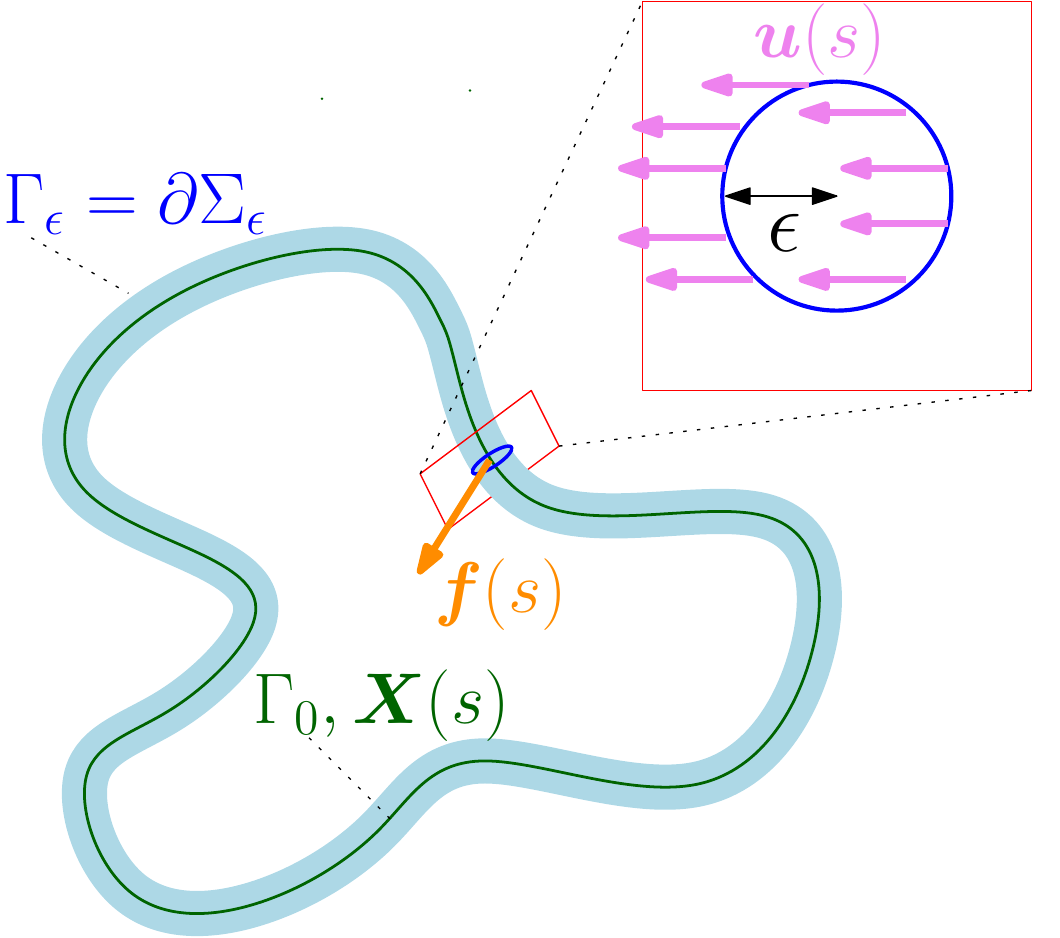}
\caption{In the slender body problem, we specify a line force density $\bm{f}(s)$ everywhere along $\Gamma_\epsilon$ and also require that the (unknown) fiber surface velocity $\bu\big|_{\Gamma_\epsilon}$ is independent of the angle $\theta$.}
\label{fig:fandu}
\end{figure}

As far as we know, this type of elliptic boundary value problem has not been explored in the literature. However, this formulation appears to be the natural PDE interpretation of the slender body problem, as any smooth enough solution to \eqref{exterior_stokes} satisfies the identity
\begin{align*}
\int_{\Omega_{\epsilon}} 2 \ts |\E(\bu)|^2\ts d\bx &= \int_{\Gamma_{\epsilon}} \bu(s)\cdot(\bm{\sigma}{\bm n}) \ts \mathcal{J}_{\epsilon} \ts d\theta ds \\
&= \int_{\T} \bu(s)\cdot {\bm f}(s)\ts ds, \hspace{1cm} \E(\bu) = \frac{\nabla\bu+(\nabla\bu)^{\rm T}}{2},
 \end{align*}
 where $\E(\bu)$ is the strain rate tensor, or symmetric gradient. This expression has a natural physical interpretation: the dissipation per unit time due to viscosity (left hand side) balances the power exerted by the slender body (right hand side). As we will see in Section \ref{PDE_stokes}, this identity is also the basis for our well-posedness theory.  \\
 

We show that the PDE \eqref{exterior_stokes} is well-posed in the homogeneous Sobolev space $D^{1,2}(\Omega_{\epsilon})$ (see \eqref{D12_definition} for a definition). Using the definition of weak solution given by Definition \ref{weak_sol_def} and \eqref{weak_exterior_p}, we show the following theorem: 
\begin{theorem}\emph{(Well-posedness of slender body PDE)}\label{stokes_theorem} 
Let $\Omega_{\epsilon}= \R^3\backslash \overline{\Sigma_{\epsilon}}$ for $\Sigma_{\epsilon}$ with $C^2$ centerline $\X(s)$ satisfying the geometric constraints in Section \ref{geometric_constraints}. Given ${\bm f}\in L^2(\T)$, there exists a unique weak solution $(\bu,p)\in D^{1,2}(\Omega_{\epsilon})\times L^2(\Omega_{\epsilon})$ to \eqref{exterior_stokes} satisfying 
\begin{equation}\label{stokes_est}
\|\bu\|_{D^{1,2}(\Omega_{\epsilon})} + \|p\|_{L^2(\Omega_{\epsilon})} \le |\log\epsilon|^{1/2}c_{\kappa} \|{\bm f}\|_{L^2(\T)},
\end{equation}
where the constant $c_{\kappa}$ depends only on the constants $c_{\Gamma}$ and $\kappa_{\max}$ characterizing the shape of the fiber centerline. 

\end{theorem}


The explicit $\epsilon$-dependence of the constant $c_{\kappa}|\log\epsilon|^{1/2}$ is determined by the various inequalities used in the well-posedness theory for \eqref{exterior_stokes}, which will be summarized in Section \ref{constants0}. We are ultimately interested in using the solution theory framework established for Theorem \ref{stokes_theorem} to estimate the error between the true solution and the slender body approximation in terms of the slender body radius $\epsilon$. For this, it is important to be able to characterize and control the $\epsilon$-dependence in any constants arising in the solution theory. From a numerical analysis perspective, determining the $\epsilon$-dependence in the well-posedness theory for the slender body PDE is analogous to establishing the stability of a numerical algorithm. We thus verify the $\epsilon$-dependence of the Korn inequality, trace inequality, and pressure estimate. These are each classical inequalities, but their dependence on the size of the radius in the exterior of a thin, flexible fiber may not have been well known previously. In particular, our trace inequality (Lemma \ref{Trace_inequality}) is genuinely new, as we rely on the fiber integrity constraint in an essential way. The Korn and pressure inequalities shown here (Lemmas \ref{korn_eps} and \ref{divv_p_lem}) apply to more general boundary value problems in the exterior of thin domains, but their dependence on the radius of the thin domain appears to not be well documented. \\

We now state our main result comparing this true solution $\bu$ of \eqref{exterior_stokes} to the slender body approximation $\bu^{\SB}$, defined by \eqref{SBT2}. From this we may also compare the actual slender body velocity $\bu\big|_{\Gamma_\epsilon}(s)$ to the centerline approximation $\bu^{\SB}_C(s)$ \eqref{SBT_asymp}.

\begin{theorem}\emph{(Slender body theory error estimate)}\label{stokes_err_theorem} 
Let $\Omega_{\epsilon}= \R^3\backslash \overline{\Sigma_{\epsilon}}$ for $\Sigma_{\epsilon}$ with $C^2$ centerline $\X(s)$ satisfying the geometric constraints in Section \ref{geometric_constraints}. Given a force ${\bm f}(s)\in C^1(\T)$, let $\bu$ be the true solution to the slender body PDE \eqref{exterior_stokes} and let $\bu^{\SB}$ be the corresponding slender body approximation \eqref{SBT2}. Then the difference $\bu^{\SB}- \bu$, $p^{\SB}-p$ satisfies 
\begin{equation}\label{err_stokes_thm}
\|\bu^{\SB}-\bu\|_{D^{1,2}(\Omega_{\epsilon})}+ \|p^{\SB}-p\|_{L^2(\Omega_{\epsilon})} \le \epsilon|\log\epsilon| \ts c_{\kappa} \ts \|{\bm f}\|_{C^1(\T)}.
\end{equation}
Furthermore, the difference between the true velocity ${\rm Tr}(\bu)(s)$ of the slender body itself and the centerline approximation $\bu^{\SB}_C(s)$, given by \eqref{SBT_asymp}, satisfies 
\begin{equation}\label{center_err_thm}
\norm{{\rm Tr}(\bu) - \bu^{\SB}_C}_{L^2(\T)} \le \epsilon|\log\epsilon|^{3/2} \ts c_{\kappa} \ts \|{\bm f}\|_{C^1(\T)} .
\end{equation}
Here the constants $c_{\kappa}$ depend only on $c_{\Gamma}$ and $\kappa_{\max}$. 
\end{theorem}

In particular, asymptotic calculations by Johnson \cite{johnson1980improved} show that the doublet correction in \eqref{SBT2} for a curved centerline $\X(s)\in C^2(\T)$ allows the surface velocity $\bu^{\SB}\big|_{\Gamma_{\epsilon}}$ to satisfy the $\theta$-independence condition up to $O(\epsilon\abs{\log\epsilon})$, where ``$O$'' is the usual order symbol. We are able to rigorously verify this claim in Proposition \ref{ur_and_derivs}. \\

Although the slender body PDE is well-posed for rough ${\bm f}$, in order to obtain an error estimate, the force must be more regular. It is not clear that ${\bm f}\in C^1(\T)$ is optimal; however, some additional regularity on ${\bm f}$ is required in order for slender body theory to actually be an approximation to the slender body PDE. We will see that this is due to the fact that the error depends crucially on the change in the total force distribution along the fiber centerline. The other sources of error stem from the nonzero curvature of the fiber centerline as well as the finite length of the fiber. These error sources are identified in Section \ref{residual_calc} by calculating the residual between the slender body approximation and the true force and velocity along $\Gamma_{\epsilon}$. Although slender body theory is a continuous approximation to a continuous problem, this step can be considered from a numerical analysis point of view as establishing the consistency of the slender body approximation. The exact form of the error estimates in Theorem \ref{stokes_err_theorem} is derived in Section \ref{error_est_section} by combining the estimates of the residuals from Section \ref{residual_calc} with the stability estimates of Section \ref{PDE_stokes}.


\section{Well-posedness of slender body PDE}\label{PDE_stokes}
In this section we prove Theorem \ref{stokes_theorem}. We begin by defining our notion of a weak solution to the slender body PDE \eqref{exterior_stokes} and, in Section \ref{constants0}, state the important inequalities arising in the well-posedness theory, as well as their dependence on $\epsilon$. Then, in Section \ref{EandU_stokes}, we show existence and uniqueness results for the weak solution to \eqref{exterior_stokes}, as well as the estimate \eqref{stokes_est}. \\

We must first define the function space $D^{1,2}(\Omega_{\epsilon})$ for which the well-posedness result is stated. We seek a solution $\bu$ to \eqref{exterior_stokes} defined over the exterior domain $\Omega_{\epsilon}=\R^3\backslash{\overline{\Sigma_{\epsilon}}}$ such that $\bu$ decays to 0 as $|\bx|\to \infty$. However, we do not expect this decay to be especially fast. In particular, we expect that $\bu$ solving \eqref{exterior_stokes} around a thin filament behaves like the Stokeslet far away from the slender body. Thus we expect $|\bu|$ to decay like $\frac{1}{|\bx|}$ as $|\bx| \to \infty$; as such, we do not expect $\bu$ to be in $L^2(\Omega_{\epsilon})$. Nevertheless, we do expect $\nabla \bu\in L^2(\Omega_{\epsilon})$, so we will consider functions in the homogeneous Sobolev space on $\Omega_{\epsilon} = \R^3\backslash \overline{\Sigma_{\epsilon}}$:
\begin{equation}\label{D12_definition}
 D^{1,2}(\Omega_{\epsilon}) = \{ \bu\in L^6(\Omega_{\epsilon}) \ts : \ts \nabla \bu\in L^2(\Omega_{\epsilon}) \}, 
 \end{equation}
 explored in detail in \cite{galdi2011introduction}, Chapter II.6 - II.10. By the Sobolev inequality 
\begin{equation}\label{sobolev_ineq}
\|\bu\|_{L^6(\Omega_{\epsilon})} \le c_S\|\nabla \bu\|_{L^2(\Omega_{\epsilon})}, \qquad c_S>0,
\end{equation}
valid in the exterior domain $\Omega_{\epsilon}\subset \R^3$, we have that
\begin{equation}\label{D12_norm}
\|\bu\|_{D^{1,2}(\Omega_{\epsilon})} \equiv \|\nabla \bu \|_{L^2(\Omega_{\epsilon})}
\end{equation}
is a norm on $D^{1,2}(\Omega_{\epsilon})$, and hence $D^{1,2}(\Omega_{\epsilon})$ is a Hilbert space arising naturally in the exterior domain $\Omega_{\epsilon}$. Letting $C_0^{\infty}(\Omega_{\epsilon})$ denote the space of smooth, compactly supported test functions in $\Omega_\epsilon$, we also define $D^{1,2}_0(\Omega_{\epsilon})$ as the closure of $C_0^{\infty}(\Omega_{\epsilon})$ in $D^{1,2}(\Omega_{\epsilon})$. \\
%

With this definition of the space $D^{1,2}(\Omega_\epsilon)$, we may define the notion of a weak solution to the slender body Stokes PDE. We begin by considering the variational formulation of \eqref{exterior_stokes}. We define the space
\[ \A_{\epsilon}^{\dive}= \{\bv\in D^{1,2}(\Omega_{\epsilon}) \ts : \ts \dive \ts\bv = 0, \bv|_{\Gamma_{\epsilon}}=\bv(s) \}, \]
where the value of the function $\bv(s)$ on the boundary $\Gamma_{\epsilon}$ is unspecified but independent of the surface angle $\theta$; $\bv\in\A_{\epsilon}^{\dive}$ is such that for any $\varphi\in C_0^\infty(\Gamma_\epsilon)$, we have 
\begin{equation}\label{theta_indep} 
 \int_{\Gamma_\epsilon} \bv \frac{\p\varphi}{\p\theta} \ts dS =0.
 \end{equation} 
Note, then, that the trace operator on $\A_{\epsilon}^{\dive}$ is a function defined on both $\Gamma_\epsilon$ and $\T$, as any $\bv\in \A_{\epsilon}^{\dive}$ satisfies 
\begin{align*}
 \|{\rm Tr}(\bv)\|_{L^2(\Gamma_{\epsilon})}^2 &= \int_{\T}\int_0^{2\pi} |\big({\rm Tr}(\bv)\big)(s)|^2 \ts \mathcal{J}_{\epsilon}(s,\theta) \ts d\theta ds \\
 &= 2\pi\epsilon \int_{\T}|\big({\rm Tr}(\bv)\big)(s)|^2 \ts ds = 2\pi\epsilon \|{\rm Tr}(\bv)\|_{L^2(\T)}^2.
 \end{align*}
Here we used that $\mathcal{J}_{\epsilon}(s,\theta)= \epsilon \big(1-\epsilon(\kappa_1(s)\cos\theta+\kappa_2(s)\sin\theta) \big)$. We will make a slight abuse of notation: the trace operator $\rm{Tr}$, when applied to $\A_{\epsilon}^{\dive}$ functions, will be considered as both a function on $\Gamma_\epsilon$ and on $\T$. We then have the following trace inequality for functions $\bv\in \A_{\epsilon}^{\dive}$:
\begin{equation}\label{trace_Adiv}
\frac{1}{\sqrt{2\pi\epsilon}}\|{\rm Tr}(\bv)\|_{L^2(\Gamma_{\epsilon})}=\|{\rm Tr}(\bv)\|_{L^2(\T)} \le c_T \|\nabla\bv\|_{L^2(\Omega_{\epsilon})},
\end{equation}
where the $\epsilon$-dependence of the constant $c_T$ will be specified in Section \ref{constants0}. The set $\A_{\epsilon}^{\dive}$ is nontrivial, as can be seen, for example, by taking any constant function on the surface $\Gamma_{\epsilon}$ and solving the corresponding Stokes boundary value problem in $\Omega_{\epsilon}$ with this boundary data (see \cite{galdi2011introduction}, Chapter V.2 for treatment of the Stokes Dirichlet boundary value problem). Furthermore, taking a sequence $\bv_k\in\A_{\epsilon}^{\dive}$ such that $\bv_k\to\bv$ in $L^2$, we immediately see that $\bv$ satisfies the $\theta$-independence condition \eqref{theta_indep} as well; hence $\A_{\epsilon}^{\dive}$ is a closed subspace of $D^{1,2}(\Omega_{\epsilon})$. \\

We can then define a weak solution to \eqref{exterior_stokes} as follows:
\begin{definition}{(Weak solution to slender body Stokes PDE)}\label{weak_sol_def} 
A weak solution $\bu\in \A_{\epsilon}^{\dive}$ to \eqref{exterior_stokes} satisfies
\begin{equation}\label{weak_exterior}
\int_{\Omega_{\epsilon}} 2\ts \mathcal{E}(\bu):\mathcal{E}(\bv)\ts d\bx - \int_{\T} \bv(s)\cdot{\bm f}(s)\ts ds =0
\end{equation}
for any $\bv \in \A_{\epsilon}^{\dive}$. 
\end{definition} 

\begin{remark}
To use the language of finite element analysis, we note that the partial Dirichlet data, given by the fiber integrity condition $\bu\big|_{\Gamma_{\epsilon}}=\bu(s)$, is enforced as part of the function space $\A_{\epsilon}^{\dive}$ (an \emph{essential} boundary condition), whereas the partial Neumann data -- the total force per fiber cross section equals ${\bm f}(s)$ -- arises out of the variational formulation \eqref{weak_exterior} itself (a \emph{natural} boundary condition). 
\end{remark}

To formally verify that weak solutions of the slender body PDE \eqref{exterior_stokes} satisfy \eqref{weak_exterior}, we first note that away from $\Gamma_{\epsilon}$, the Stokes equations can be rewritten in terms of the stress tensor $\bm{\sigma} = 2 \ts\E(\bu) - p{\bf I}$ as $\dive \ts \bm{\sigma} = 0$ in $\Omega_{\epsilon}$. Assume $\bu \in \A_{\epsilon}^{\dive}\cap C_0^{\infty}(\overline\Omega_{\epsilon})$ satisfies the slender body PDE \eqref{exterior_stokes}, where $C^{\infty}_0(\overline \Omega_{\epsilon})$ denotes smooth functions uniformly continuous up to $\Gamma_\epsilon$ that vanish outside of some ball containing $\Sigma_\epsilon$. Note that this differs from the function space $C^{\infty}_0(\Omega_{\epsilon})$, which includes only functions that vanish on $\Gamma_\epsilon$. The stress tensor corresponding to $\bu$ then satisfies $\dive\ts\bm{\sigma} = 0$ in $\Omega_{\epsilon}$. Multiplying this equation by any $\bv \in \A_{\epsilon}^{\dive}\cap C^\infty_0(\overline \Omega_\epsilon)$ and integrating by parts, we have 
\begin{align*}
0 &= -\int_{\Omega_{\epsilon}} \dive \ts \bm{\sigma} \cdot \bv \ts d\bx = \int_{\Omega_{\epsilon}} \bm{\sigma} : \nabla \bv \ts d\bx - \int_{\Gamma_{\epsilon}} \bv \cdot (\bm{\sigma}{\bm n}) \ts dS \\
&= \int_{\Omega_{\epsilon}} \big(2\ts \E(\bu): \nabla \bv - p\ts \dive\ts \bv\big) \ts d\bx - \int_{\T}\int_0^{2\pi}\bv(s) \cdot(\bm{\sigma}{\bm n}) \ts \mathcal{J}_{\epsilon}(s,\theta) \ts d\theta ds \\
&= \int_{\Omega_{\epsilon}} \big(\nabla \bu: \nabla \bv+\nabla\bu^{\rm T}:\nabla \bv\big) \ts d\bx - \int_{\T}\bv(s)\cdot \int_0^{2\pi} (\bm{\sigma}{\bm n}) \ts \mathcal{J}_{\epsilon}(s,\theta) \ts d\theta ds \\
&= \int_{\Omega_{\epsilon}} 2 \ts\E(\bu): \E(\bv) \ts d\bx - \int_{\T} \bv(s) \cdot {\bm f}(s) \ts ds. 
\end{align*}
By density, this computation then holds for any $\bv \in \A_{\epsilon}^{\dive}$. Note that in the second line, we have rewritten the integral over $\Gamma_\epsilon$ in terms of the moving frame coordinates $(s,\theta)$, so the surface element becomes $dS= \mathcal{J}_{\epsilon}(s,\theta) \ts d\theta ds$. In the third line, we use that $\bv\in \A_{\epsilon}^{\dive}$ to pull the boundary term $\bv(s)$ out of the $\theta$-integral. The remaining integral in $\theta$ is exactly the force density $\bm{f}(s)$ that we defined in \eqref{exterior_stokes}. \\

Using this definition of a weak solution, we verify the existence and uniqueness claim of Theorem \ref{stokes_theorem}. Additionally, we show that the following is an equivalent definition of weak solution to \eqref{exterior_stokes} that includes a corresponding weak pressure $p\in L^2(\Omega_{\epsilon})$: 
\begin{definition}{\emph{(Weak solution with pressure)}}\label{pressure_exist}
Given a weak solution $\bu$ satisfying \eqref{weak_exterior}, there exists a unique corresponding pressure $p\in L^2(\Omega_{\epsilon})$ satisfying
\begin{equation}\label{weak_exterior_p}
 \int_{\Omega_{\epsilon}}\big(2 \ts\E(\bu):\E(\bv) - p\ts\dive\ts \bv\big) \ts d\bx - \int_{\T} \bv(s)\cdot{\bm f}(s) \ts ds = 0
\end{equation}
for any $\bv\in \A_{\epsilon}= \{\bv\in D^{1,2}(\Omega_{\epsilon}) \ts : \ts \bv|_{\Gamma_{\epsilon}}=\bv(s) \}$, where we have removed the divergence-free restriction on $\bv$. 
\end{definition}
We show that Definitions \ref{weak_sol_def} and \ref{pressure_exist} are equivalent in Section \ref{EandU_stokes}. Note that if $(\bu,p)\in (\A_{\epsilon}^{\dive}\cap C_0^{\infty}(\overline\Omega_{\epsilon}))\times C_0^{\infty}(\overline\Omega_{\epsilon})$ satisfies \eqref{weak_exterior_p}, then, integrating by parts, 
\begin{align*}
 0 &= -\int_{\Omega_{\epsilon}}\left(2\ts\dive(\E(\bu))\cdot \bv - \nabla p\cdot \bv\right) \ts d\bx + \int_{\Gamma_{\epsilon}}\left(2\ts \E(\bu){\bm n} - p\ts{\bm n}\right)\cdot \bv \ts dS  - \int_{\T} \bv(s)\cdot{\bm f}(s) \ts ds \\
 &= -\int_{\Omega_{\epsilon}} (\Delta\bu-\nabla p) \cdot\bv \ts d\bx + \int_{\T}\int_0^{2\pi}(\bm{\sigma}{\bm n})\cdot \bv(s) \ts \mathcal{J}_{\epsilon}(s,\theta)\ts d\theta ds  - \int_{\T} \bv(s)\cdot{\bm f}(s) \ts ds \\
 &= \int_{\Omega_{\epsilon}} (-\Delta\bu+\nabla p) \cdot\bv \ts d\bx + \int_{\T}\bv(s)\cdot\bigg(\int_0^{2\pi}(\bm{\sigma}{\bm n}) \ts \mathcal{J}_{\epsilon}(s,\theta)\ts d\theta - {\bm f}(s)\bigg) ds.
\end{align*}

Since this holds for any $\bv\in \A_{\epsilon}\cap C^\infty_0(\overline \Omega_\epsilon)$, and thus, by density, for any $\bv\in  \A_{\epsilon}$, the pair $(\bu,p)$ in fact satisfies equation \eqref{exterior_stokes} pointwise almost everywhere. Therefore, any smooth enough solution pair $(\bu,p)$ satisfying the weak formulation \ref{weak_exterior_p} is a classical solution of \eqref{exterior_stokes}. \\

We begin by stating the $\epsilon$-dependence of the inequalities arising in the well-posedness theory for \eqref{exterior_stokes}, the proofs of which are given in Appendix \ref{constants}. Using these inequalities, we show the existence and uniqueness of weak solutions to \eqref{weak_exterior} and hence to \eqref{weak_exterior_p}, as well as the estimate \eqref{stokes_est} from Theorem \ref{stokes_theorem}. 

\subsection{Dependence of key inequalities on $\epsilon$}\label{constants0}
In this section we collect the key inequalities used in the well-posedness theory for \eqref{weak_exterior} and note their explicit dependence on the slender body radius $\epsilon$. This will allow us to prove the $\epsilon$-dependence in the constant arising in the estimate \eqref{stokes_est} of Theorem \ref{stokes_theorem}. As noted in the introduction, it will be important to characterize how constants in the well-posedness framework depend on $\epsilon$, as we are ultimately interested in proving the error estimate in Theorem \ref{stokes_err_theorem}. In addition, the explicit $\epsilon$-dependence in some of these inequalities is either completely new, as in the case of the trace inequality (Lemma \ref{Trace_inequality}), or not well-documented, as in the case of the Korn inequality (Lemma \ref{korn_eps}). The proofs of each inequality appear in Appendix \ref{constants}. \\

First, since we are working in the function space $D^{1,2}(\Omega_\epsilon)$ \eqref{D12_definition}, it will be useful to verify the $\epsilon$-independence of the Sobolev inequality \eqref{sobolev_ineq} on $\Omega_\epsilon$. 
 \begin{lemma}\emph{(Sobolev inequality)}\label{sobo_ineq}
Let $\Omega_{\epsilon}=\R^3\backslash\overline{\Sigma_{\epsilon}}$, the exterior of a slender body of radius $\epsilon$. For any $\bu\in D^{1,2}(\Omega_{\epsilon})$, we have
\begin{equation}\label{sobolev_const}
\| \bu\|_{L^6(\Omega_{\epsilon})} \le c_S\|\nabla\bu\|_{L^2(\Omega_{\epsilon})}
\end{equation}
with a constant $c_S$ that is bounded independent of $\epsilon$ as $\epsilon\to 0$. 
\end{lemma}
The proof of Lemma \ref{sobo_ineq} appears in Section \ref{Sob_ineq}. \\

We will also need to establish the $\epsilon$-dependence in the $\A_\epsilon$ trace inequality, which is the same as the $\A_\epsilon^\dive$ trace inequality \eqref{trace_Adiv}. Even though the slender body surface $\Gamma_{\epsilon}$ is codimension 1 and, for $\bu\in D^{1,2}(\Omega_{\epsilon})$, satisfies an $H^{1/2}(\Gamma_{\epsilon})$ trace inequality, the trace estimate needed for our existence theory and error bound is essentially a codimension 2 trace inequality, which appears to introduce an additional $1/\sqrt{\epsilon}$ that we must bound. However, we can show that the constant in the $L^2$ trace inequality grows only like $|\log\epsilon|^{1/2}$ as $\epsilon\to 0$.
\begin{lemma}\emph{(Trace inequality)}\label{Trace_inequality}
Let $\Omega_\epsilon = \R^3\backslash \overline{\Sigma_\epsilon}$ be as in Section \ref{geometric_constraints}. For $\bu\in \A_{\epsilon}$, the $\theta$-independent trace of $\bu$ on $\Gamma_{\epsilon}$ satisfies 
\begin{equation}\label{Trace_ineq} 
\|{\rm Tr}(\bu)\|_{L^2(\T)} \le c_{\kappa} |\log\epsilon|^{1/2} \| \nabla \bu\|_{L^2(\Omega_{\epsilon})}, 
\end{equation}
where ${\rm Tr} : D^{1,2}(\Omega_{\epsilon}) \to L^2(\T)$ is the trace operator and the constant $c_{\kappa}$ depends on the constants $\kappa_{\max}$ and $c_{\Gamma}$ but is independent of the fiber radius $\epsilon$. 
\end{lemma}

This $\epsilon$-dependence in the trace inequality is not surprising, as we expect that in the limit as $\epsilon\to 0$ the true solution will look something like the Stokeslet, which has unbounded velocity along the fiber centerline. In fact, this $\epsilon$ dependence should be optimal for the $L^2(\T)$ trace. The proof of Lemma \ref{Trace_inequality} is shown in Section \ref{trace_sec}. \\


Next, in order to show estimate \eqref{stokes_est}, we will need a Korn inequality bounding $\nabla \bu$ by $\E(\bu)$, the symmetric part of the gradient. We show in Section \ref{korn_proof} that the constant in the Korn inequality is bounded independently of $\epsilon$. 
\begin{lemma}\emph{(Korn inequality)}\label{korn_eps}
Let $\Omega_{\epsilon}=\R^3 \backslash \overline{\Sigma_{\epsilon}}$ be as in Section \ref{geometric_constraints}. Then any $\bu\in D^{1,2}(\Omega_{\epsilon})$ satisfies 
\begin{equation}\label{korn_ineq}
 \|\nabla \bu\|_{L^2(\Omega_{\epsilon})} \le c_K\|\E(\bu)\|_{L^2(\Omega_{\epsilon})}, 
 \end{equation}
 where the constant $c_K$ depends only on $\kappa_{\max}$ and $c_{\Gamma}$.
\end{lemma}


Finally, the $\epsilon$-dependence in the estimate \eqref{stokes_est} of Theorem \ref{stokes_theorem} relies on the $\epsilon$-independence of the following inequality, which is intimately tied to the pressure estimate \eqref{press_est} that will be used to show \eqref{stokes_est}. 

\begin{lemma}{\emph{(Solution to $\dive\ts \bv=p$)}}\label{divv_p_lem} 
Let $\Omega_{\epsilon}=\R^3\backslash \overline{\Sigma_{\epsilon}}$ be as in Section \ref{geometric_constraints}. There exists a function $\bv\in D^{1,2}_0(\Omega_{\epsilon})$ satisfying 
\begin{align*}
\dive\ts \bv &= p \quad \text{ in }\Omega_{\epsilon}; \\
\|\bv\|_{D^{1,2}(\Omega_{\epsilon})} &\le c_P\| p\|_{L^2(\Omega_{\epsilon})}, 
\end{align*}
where the constant $c_P$ depends on $\kappa_{\max}$ and $c_{\Gamma}$ but not on $\epsilon$. 
\end{lemma}
For fixed $\epsilon$, the existence of such a $\bv$ is guaranteed by \cite{galdi2011introduction}, Theorem III.3.6, which follows the original construction by Bogovskii \cite{bogovskii1980solutions}. In Section \ref{pressure_const}, we reiterate the proof of this theorem to determine the dependence of the constant $c_P$ on the slender body radius $\epsilon$.


\subsection{Existence and uniqueness}\label{EandU_stokes}
We now use the inequalities outlined in the previous section to prove Theorem \ref{stokes_theorem}. We begin by verifying the existence and uniqueness of a weak solution $\bu$ to \eqref{weak_exterior}. 

\begin{proof}[Proof of existence and uniqueness assertion in Theorem \ref{stokes_theorem}:] 
To show existence of a weak solution $\bu \in \A_{\epsilon}^{\dive}$ to \eqref{weak_exterior}, we first show that the bilinear form 
\[ \mathcal{B}[\bu,\bv]:=\int_{\Omega_{\epsilon}} 2 \ts \E(\bu):\E(\bv) \ts d\bx \]
is coercive on $\A_{\epsilon}^{\dive}$. Using the Korn inequality \eqref{korn_ineq}, for any $\bu\in \A_{\epsilon}^{\dive}$ we have
\begin{align*}
\mathcal{B}[\bu,\bu] &= \int_{\Omega_{\epsilon}} 2 \ts |\E(\bu)|^2 \ts d\bx \ge \int_{\Omega_{\epsilon}} \frac{2}{c_K^2} \ts |\nabla\bu|^2 \ts d\bx = \frac{2}{c_K^2}\|\nabla\bu\|_{L^2(\Omega_{\epsilon})}^2 ,
\end{align*}
so $\mathcal{B}[\cdot,\cdot]$ is coercive on $\A_{\epsilon}^{\dive}$. Also, $\mathcal{B}[\cdot,\cdot]$ is bounded, since
\begin{align*}
\abs{\mathcal{B}[\bu,\bv]} &\le \int_{\Omega_{\epsilon}} 2 |\E(\bu)| |\E(\bv)| \ts d\bx \le 2 \|\E(\bu)\|_{L^2(\Omega_{\epsilon})}\|\E(\bv)\|_{L^2(\Omega_{\epsilon})} \le 2 \|\nabla \bu\|_{L^2(\Omega_{\epsilon})}\|\nabla\bv\|_{L^2(\Omega_{\epsilon})}.
\end{align*}

Furthermore, for ${\bm f}\in L^2(\T)$ and $\bv\in \A_{\epsilon}^{\dive}$, the linear functional
\[ \ell({\bm f}) := \int_{\T} {\bm f}(s)\cdot \bv(s) \ts ds \]
is bounded: using Cauchy-Schwarz and the trace inequality (Lemma \ref{Trace_inequality}) in $\A_{\epsilon}^{\dive}$,
\begin{align*}
\int_{\T} \bv(s)\cdot{\bm f}(s) \ts ds &\le \|\bv\|_{L^2(\T)}\|{\bm f}\|_{L^2(\T)} \le c_T\|\nabla\bv\|_{L^2(\Omega_{\epsilon})}\|{\bm f}\|_{L^2(\T)} .
\end{align*}

Since the form $\mathcal{B}[\cdot,\cdot]$ is bounded and coercive on $\A_{\epsilon}^{\dive}$ and the functional $\ell(\cdot)$ is bounded on $\A_{\epsilon}^{\dive}$, by the Lax-Milgram theorem there exists a unique solution $\bu\in \A_{\epsilon}^{\dive}$ to \eqref{weak_exterior}. Furthermore, taking $\bv=\bu$ in \eqref{weak_exterior} and using the Korn inequality \eqref{korn_ineq}, we have that this solution $\bu$ satisfies
\begin{align*}
\|\nabla \bu\|_{L^2(\Omega_{\epsilon})}^2 &\le c_{K}^2\|\E(\bu)\|_{L^2(\Omega_{\epsilon})}^2 \le \frac{c_K^2}{2}\|{\bm f}\|_{L^2(\T)}\|\bu\|_{L^2(\T)} \\
&\le \frac{c_K^2}{2}\left(\frac{1}{4\delta}\|{\bm f}\|_{L^2(\T)}^2+\delta\|\bu\|_{L^2(\T)}^2\right) \le \frac{c_K^2}{2}\left(\frac{1}{4\delta}\|{\bm f}\|_{L^2(\T)}^2+\delta c_T^2\|\nabla\bu\|_{L^2(\Omega_{\epsilon})}^2\right).
\end{align*}

Taking $\delta=\frac{1}{c_T^2c_K^2}$, we obtain
\begin{equation}\label{u_est}
\|\nabla \bu\|_{L^2(\Omega_{\epsilon})} \le \frac{1}{2}c_K^2c_T\|{\bm f}\|_{L^2(\T)}.
\end{equation} 
\end{proof}

The existence of a unique velocity $\bu\in D^{1,2}(\Omega_{\epsilon})$ satisfying \eqref{weak_exterior} can be used to show the equivalence of Definitions \ref{weak_sol_def} and \ref{pressure_exist}, the characterization of a weak solution to \eqref{exterior_stokes} without and with the unique corresponding pressure $p\in L^2(\Omega_{\epsilon})$. The existence of the pressure relies on the following lemma, the proof of which can be found in \cite{galdi2011introduction}, Corollary III.5.1:
\begin{lemma}\emph{(de Rham Theorem)}\label{de_rham}
Let $\Omega_{\epsilon}=\R^3 \backslash\overline{\Sigma_{\epsilon}}$. Any bounded linear functional $\ell$ on $D^{1,2}_0(\Omega_{\epsilon})$ identically vanishing on the divergence-free subspace $D^{1,2}_{0,\dive}(\Omega_{\epsilon})$ is of the form 
\[ \ell(\bw)=\int_{\Omega_{\epsilon}} p \ts \dive\ts \bw \ts d\bx \qquad  \bw \in D_0^{1,2}(\Omega_{\epsilon})\]
for some uniquely determined $p\in L^2(\Omega_{\epsilon})$. 
\end{lemma}

 \begin{proof}[Proof of equivalence of Definitions \ref{weak_sol_def} and \ref{pressure_exist}:]
We begin by considering \eqref{weak_exterior} away from $\Gamma_{\epsilon}$. Recall the definition of $D^{1,2}_{0,\dive}(\Omega_{\epsilon})$ in Lemma \ref{pressure_exist}. Since $\bu$ is a weak solution to \eqref{weak_exterior}, we have
\[ \int_{\Omega_{\epsilon}} 2\ts\E(\bu):\E(\bw) \ts d\bx = 0 \qquad \text{ for all } \bw \in D^{1,2}_{0,\dive}(\Omega_\epsilon). \]

Using Lemma \ref{de_rham}, we then have
\begin{equation}\label{pressure_eqn}
\int_{\Omega_{\epsilon}} 2\ts\E(\bu):\E(\bw) \ts d\bx = \int_{\Omega_{\epsilon}} p \ts\dive\ts \bw \ts d\bx \qquad \text{ for all } \bw \in D_0^{1,2}(\Omega_{\epsilon}).
\end{equation}

Thus, removing the divergence-free restriction on $\bw\in D_0^{1,2}(\Omega_{\epsilon})$, we recover $p$ in $\Omega_{\epsilon}$ away from the slender body surface $\Gamma_{\epsilon}$. We now must show that this $p$ satisfies the correct boundary conditions for the total surface force over $\Gamma_{\epsilon}$ when integrated against arbitrary $\bv \in \A_{\epsilon}$. \\

Consider a solution $\bu \in \A_{\epsilon}^{\dive}$ satisfying \eqref{weak_exterior}. For any $\bv\in \A_{\epsilon}$ we write $\bv$ as 
\[ \bv = \bw+{\bm \psi} \]
where ${\bm \psi}$ is the unique (weak) solution to the classical exterior Stokes boundary value problem
\begin{equation}\label{classical_stokes}
\begin{aligned}
-\Delta {\bm \psi} + \nabla \pi &= 0, \quad \dive \ts {\bm \psi} = 0 \quad \text{ in }\Omega_\epsilon \\
{\bm \psi} \big|_{\Gamma_{\epsilon}} &= \bv(s) \\
{\bm \psi} &\to 0 \quad \text{as }|\bx| \to \infty
\end{aligned}
\end{equation}
in the space $D^{1,2}_{\dive}(\Omega_{\epsilon})$. Again the subscript ``div'' denotes the divergence-free subspace of $D^{1,2}(\Omega_{\epsilon})$. We refer to \cite{galdi2011introduction}, Chapter V.2 for details on the existence and uniqueness results for \eqref{classical_stokes}. \\

Thus ${\bm \psi}$ is in $\A_\epsilon^{\dive}$, so by Definition \ref{weak_sol_def} we have
\begin{equation}\label{div_free_part}
\int_{\Omega_{\epsilon}} 2 \ts \E(\bu):\E({\bm \psi}) \ts d\bx = \int_{\T} {\bm f}(s)\cdot\bv(s) \ts ds.
\end{equation}

Furthermore, we have that $\bw\in D_0^{1,2}(\Omega_{\epsilon})$ satisfies
\begin{equation}\label{trace_free_part}
\int_{\Omega_{\epsilon}} 2 \ts \E(\bu):\E(\bw) \ts d\bx = \int_{\Omega_{\epsilon}} p\ts \dive\ts\bw \ts d\bx,
\end{equation}
by equation \eqref{pressure_eqn}. \\

Adding \eqref{div_free_part} and \eqref{trace_free_part} we therefore have  
\begin{align*}
\int_{\Omega_{\epsilon}} 2 \ts \E(\bu):\E(\bv) \ts d\bx &= \int_{\Omega_{\epsilon}} 2 \ts \E(\bu):\E(\bw) \ts d\bx+\int_{\Omega_{\epsilon}} 2 \ts \E(\bu):\E({\bm \psi}) \ts d\bx \\
&= \int_{\Omega_{\epsilon}} p\ts \dive\ts\bw \ts d\bx + \int_{\T} {\bm f}(s)\cdot\bv(s) \ts ds.
\end{align*}

Hence the pressure $p$ from Lemma \ref{de_rham} satisfies the desired boundary condition on $\Gamma_{\epsilon}$, and therefore $(\bu,p)\in \A_{\epsilon}^{\dive}\times L^2(\Omega_{\epsilon})$ satisfies
\begin{equation}\label{weakstokes_pressure} 
\int_{\Omega_{\epsilon}} \bigg(2 \ts\mathcal{E}(\bu):\mathcal{E}(\bv) - p\ts\dive\ts \bv \bigg)\ts d\bx - \int_{\T} \bv(s)\cdot{\bm f}(s) \ts ds = 0
\end{equation}
for all $\bv \in \A_{\epsilon}$. We have thus removed the divergence-free constraint on $\bv$ to show the existence of a unique corresponding pressure $p\in L^2(\Omega_{\epsilon})$.
\end{proof} 

Finally, from \eqref{weakstokes_pressure}, we derive the energy estimate \eqref{stokes_est} in Theorem \ref{stokes_theorem}. For this we will need to use the $\epsilon$-independence of the constant $c_P$ established in Lemma \ref{divv_p_lem}. 

\begin{proof}[Proof of estimate \eqref{stokes_est}:]
Following \cite{galdi2011introduction}, we first show that for $(\bu,p)$ satisfying \eqref{weakstokes_pressure}, we have
\begin{equation}\label{press_est}
\|p\|_{L^2(\Omega_{\epsilon})} \le \tilde c_P\|\E(\bu)\|_{L^2(\Omega_{\epsilon})},
\end{equation}
for some constant $\tilde c_P>0$. To show \eqref{press_est}, we consider $\bv\in D^{1,2}_0(\Omega_{\epsilon})$ satisfying
\begin{equation}\label{divv_p}
\begin{aligned}
\dive\ts \bv &= p \qquad \text{ in }\Omega_{\epsilon}; \\
\|\bv\|_{D^{1,2}(\Omega_{\epsilon})} &\le c_P\| p\|_{L^2(\Omega_{\epsilon})}. 
\end{aligned}
\end{equation}
By Lemma \ref{divv_p_lem}, such a $\bv$ exists and the constant $c_P$ depends only on $c_\Gamma$ and $\kappa_{\max}$. \\

 Now, substituting $\bv$ satisfying \eqref{divv_p} into \eqref{weakstokes_pressure}, we have 
\begin{align*}
\int_{\Omega_{\epsilon}} p^2 \ts d\bx &= \int_{\Omega_{\epsilon}} 2 \ts\E(\bu):\E(\bv) \ts d\bx \le 2\|\E(\bu)\|_{L^2(\Omega_{\epsilon})}\|\E(\bv)\|_{L^2(\Omega_{\epsilon})} \le 2\|\E(\bu)\|_{L^2(\Omega_{\epsilon})}\|\nabla\bv\|_{L^2(\Omega_{\epsilon})} \\
& \le \frac{1}{\eta}\|\E(\bu)\|_{L^2(\Omega_{\epsilon})}^2+ \eta\|\nabla\bv\|_{L^2(\Omega_{\epsilon})}^2 \le \frac{1}{\eta}\|\E(\bu)\|_{L^2(\Omega_{\epsilon})}^2+ \eta c_P^2\|p\|_{L^2(\Omega_{\epsilon})}^2, \quad \eta\in\R_+. 
\end{align*}
Taking $\eta=\frac{1}{2c_P^2}$, we obtain \eqref{press_est}, with $\tilde c_P= 2c_P$. \\

Combining the pressure estimate \eqref{press_est} with the velocity estimate \eqref{u_est} and noting the $\epsilon$-dependence of the constants $c_T$, $c_K$, and $c_P$ established in Section \ref{constants0}, we obtain 
\begin{align*}
\|\bu\|_{D^{1,2}(\Omega_{\epsilon})} + \|p\|_{L^2(\Omega_{\epsilon})} &\le \frac{1}{2}c_K^2c_T(1+2c_P) \|{\bm f}\|_{L^2(\T)} \le c_{\kappa}|\log\epsilon|^{1/2}\|{\bm f}\|_{L^2(\T)}.
\end{align*}

\end{proof}

\section{Slender body residual calculations}\label{residual_calc}
Now that we have proved Theorem \ref{stokes_theorem}, we may proceed to the main aim of the paper: to compare the slender body approximation to the true solution and derive an error estimate in terms of the slender body radius $\epsilon$. In this section, we calculate the residual for the slender body force and velocity approximations, which will then be used in the next section to prove the error bounds in Theorem \ref{stokes_err_theorem}. 

\subsection{Slender body calculations: setup}
The proof of Theorem \ref{stokes_err_theorem} requires knowledge of two expressions: the total surface force ${\bm f}^{\SB}(s)$ exerted by the slender body approximation at each cross section $s$ along the true surface $\Gamma_{\epsilon}$, and the $\theta$-dependence in the slender body velocity $\bu^{\SB}\big|_{\Gamma_\epsilon}(s,\theta)$. Although the true surface velocity $\bu\big|_{\Gamma_{\epsilon}}(s)$ is unknown, we can measure the degree to which $\bu^{\SB}$ fails to satisfy the $\theta$-independence condition along $\Gamma_\epsilon$. In analogy with finite element analysis, the $\theta$-dependence in $\bu^{\SB}\big|_{\Gamma_{\epsilon}}(s,\theta)$ can be considered as the {\it non-conforming} residual, as the slender body approximation $\bu^{\SB}$ therefore does not belong to the function space $\A_{\epsilon}^{\dive}$ required by the well-posedness theory. The force residual ${\bm f}^{\SB}(s)-{\bm f}(s)$, on the other hand, can be considered as the {\it conforming} residual, as the slender body force approximation ${\bm f}^{\SB}$ belongs to the same function space as the prescribed force ${\bm f}$. To show the centerline estimate \eqref{center_err_thm} of Theorem \ref{stokes_err_theorem}, we will also need to consider the centerline residual $|\bu^{\SB}\big|_{\Gamma_\epsilon}(s,\theta)-\bu^{\SB}_C(s)|$ between the slender body approximation on the fiber surface and the centerline slender body approximation \eqref{SBT_asymp}. \\

In this section we will state and prove a few useful lemmas regarding integral estimates along $\T$. The estimates needed to bound both the conforming and non-conforming residuals can be summarized into Lemmas \ref{Rintest0}, \ref{Rintest1}, and \ref{Rintest2}. In addition, we show Lemma \ref{center_est_lem}, which will be used to bound the centerline residual $|\bu^{\SB}\big|_{\Gamma_\epsilon}(s,\theta)-\bu^{\SB}_C(s)|$. These bounds will then be used in Section \ref{SB_vel} to prove a series of propositions leading to Proposition \ref{ur_and_derivs}, which states a bound for the $\theta$-dependence in $\bu^{\SB}\big|_{\Gamma_\epsilon}$ and its derivatives. We will also use Lemma \ref{center_est_lem} to show the centerline residual bound in Proposition \ref{centerline_prop}. In Section \ref{SBforce_res}, we use Lemmas \ref{Rintest0} - \ref{Rintest2} as well as an additional Lemma \ref{theta_int} to estimate the slender body approximation $\bm{f}^{\SB}(s)$ to the force. Ultimately we show Proposition \ref{fSB_est} bounding the residual $\bm{f}^{\SB}(s) - \bm{f}(s)$. Throughout these sections, we will use $c_\kappa$ to denote any constant depending only on the fiber centerline shape through $c_\Gamma$ and $\kappa_{\max}$. \\

We assume that the slender body satisfies the geometric constraints in Section \ref{geometric_constraints}. Although a solution to the slender body PDE \eqref{exterior_stokes} is guaranteed so long as ${\bm f}$ is in $L^2(\T)$, some additional smoothness on $\bm{f}$ is required for the slender body approximation to actually approximate the slender body PDE. Here we will require ${\bm f}$ to be in $C^1(\T)$. We recall that the slender body approximation is given by 
\begin{align}\label{stokes_SB}
\bu^{\SB}(\bx) &=\frac{1}{8\pi}\int_{\T} \bigg( \mc{S}(\bm{R})+\frac{\epsilon^2}{2}\mc{D}(\bm{R}) \bigg)\bm{f}(t) \ts dt; \; 
\bm{R}=\bm{x}-\bm{X}(t),\\
\label{SD}
\mc{S}(\bm{R})&=\frac{{\bf I}}{\abs{\bm{R}}}+\frac{\bm{R}\bm{R}^{\rm T}}{\abs{\bm{R}}^3}, \; 
\mc{D}(\bm{R})=\frac{{\bf I}}{\abs{\bm{R}}^3}-\frac{3\bm{R}\bm{R}^{\rm T}}{\abs{\bm{R}}^5},
\end{align}
with the corresponding slender body pressure given by
\begin{equation}\label{SB_pressure}
p^{\SB}(\bx) = \frac{1}{4\pi}\int_{\T} \frac{\bm{R}\cdot {\bm f}(t)}{|\bm{R}|^3} \ts dt. 
\end{equation}

Recall that within the neighborhood $\mathcal{O}$ \eqref{region_O}, any point $\bx$ can be written
\[\bx(\rho,\theta,s) = \X(s)+\rho \be_{\rho}(s,\theta).\]
Then, for $\bx\in \mathcal{O}$, $\bm{R}$ has the form 
\begin{align*}
\bm{R}(\rho,\theta,s;t) &= \X(s)- \X(t) + \rho \be_{\rho}(s,\theta).
\end{align*}

Before we begin calculations to estimate $\bu^{\SB}$ and $\bm{f}^{\SB}$, we note some useful facts. Using the moving frame ODE \eqref{moving_ODE}, we have 
\begin{align}
\frac{\p \bm{R}}{\p \rho}&=\be_\rho(s,\theta), \label{rhoderiv}\\
\frac{1}{\rho}\frac{\p \bm{R}}{\p \theta}&=\be_\theta(s,\theta),\label{thetaderiv}\\
\frac{1}{1-\rho\widehat{\kappa}}\bigg(\frac{\p \bm{R}}{\p s}-\kappa_3 \frac{\p \bm{R}}{\p \theta}\bigg) &=\be_t(s), \label{sderiv}
\end{align}
where 
\begin{equation}\label{kappahat}
\widehat{\kappa}(s,\theta)=\kappa_1(s)\cos\theta+\kappa_2(s)\sin\theta.
\end{equation}

Next we note that, since $\X$ is a $C^2$ function, for $s,t\in \T$ we have
\begin{equation}\label{CQ}
\X(s)-\X(t)=(s-t)\be_t(s)+(s-t)^2\bm{Q}(s,t), \quad \abs{\bm{Q}(s,t)}\le \frac{\kappa_{\max}}{2}.
\end{equation}

Then, on the slender body surface $\Gamma_{\epsilon}$, we have 
\begin{equation}\label{Reps}
\bm{R}=-\bars \be_t(s)+\epsilon \be_\rho(s,\theta)+ \bars^2\bm{Q}, \quad \abs{\bm{Q}}\le \frac{\kappa_{\max}}{2}, \quad \bars = -(s-t),
\end{equation}
where we have set $\rho=\epsilon$. It will often be convenient to view $\bm{R}$ as a function of $\bars$ and $s$ rather than $t$ and $s$. We may use this expression for $\bm{R}$ to obtain the following two simple estimates.
\begin{lemma}\label{absRests}
Let $\bm{R}$ be as in \eqref{Reps}. Then, for sufficiently small $\epsilon$, we have:
\begin{align}
\label{RQ}
\abs{\abs{\bm{R}}-\sqrt{\bars^2+\epsilon^2}}&\le \frac{\kappa_{\max}}{2} \bars^2,\\
\label{Rlb}
\abs{\bm{R}}&\ge c_\kappa\sqrt{\bars^2+\epsilon^2},
\end{align}
where $\abs{\bars}\leq 1/2$ and $c_\kappa$ depends only on $c_\Gamma$ and $\kappa_{\max}$.
\end{lemma}

\begin{proof}
Note that
\begin{equation}
\abs{\bars \be_t +\epsilon\be_\rho}=\sqrt{\bars^2+\epsilon^2}.
\end{equation}
Inequality \eqref{RQ} then follows from the triangle inequality applied to \eqref{Reps}. To obtain \eqref{Rlb}, note from \eqref{RQ} that, if $\abs{\bars}\le 1/\kappa_{\max}$,
\[\abs{\bm{R}}\ge \sqrt{\bars^2+\epsilon^2}- \frac{\kappa_{\max}}{2}\bars^2\ge \frac{1}{2}\sqrt{\bars^2+\epsilon^2}
+\frac{\abs{\bars}}{2}- \frac{\kappa_{\max}}{2}\bars^2\ge \frac{1}{2}\sqrt{\bars^2+\epsilon^2}.\]
If $\kappa_{\max}\le 2$ we are done. Otherwise, suppose $1/\kappa_{\max}<\abs{\bars}\le 1/2$. Then we have
\begin{equation}
\abs{\bm{R}}\ge \abs{\bm{X}(s)-\bm{X}(t)}-\epsilon \ge c_\Gamma |\bars| -\epsilon\ge  
\frac{c_\Gamma}{\kappa_{\max}}-\epsilon\ge \frac{c_\Gamma}{2\kappa_{\max}},
\end{equation}
where we have used \eqref{non_intersecting} in the second inequality and have taken $\epsilon\le c_\Gamma/(2\kappa_{\max})$ in the last inequality. The above two estimates together imply \eqref{Rlb}.
\end{proof}

We will now make note of some integral estimates that will be used throughout the following section to bound integrals arising from the slender body expression \eqref{stokes_SB} in terms of the prescribed force $\bm{f}\in C^1(\T)$. We first note the following simple but useful calculus result, whose proof we omit. 
\begin{lemma}\label{defints}
Let $m,n$ be integers such that $m\geq 0$ and $n>0$. Then, for $\epsilon$ sufficiently small, we have
\begin{equation}
\int_{-1/2}^{1/2}\frac{\abs{\bars}^{m}}{(\bars^2+\epsilon^2)^{n/2}}d\bars \le 
\begin{cases}
3\abs{\log \epsilon}&\text{ if } n=m+1\\
\pi \epsilon^{m+1-n} &\text{ if } n\geq m+2
\end{cases}
\end{equation}
\end{lemma}

The following integral estimate then follows immediately from Lemmas \ref{absRests} and \ref{defints}. 
\begin{lemma}\label{Rintest0}
Let $\bm{R}$ be as in \eqref{Reps}. Suppose $m,n$ are integers such that $m\geq 0$ and $n>0$. For $\epsilon$ sufficiently small, we have
\begin{equation}
\int_{-1/2}^{1/2} \frac{\abs{\bars}^{m}}{\abs{\bm{R}}^n}d\bars \le 
\begin{cases}
c_\kappa\abs{\log \epsilon} &\text{ if } n=m+1,\\
c_\kappa \epsilon^{m+1-n} &\text{ if } n\ge m+2,
\end{cases}
\end{equation}
where the constants $c_\kappa$ depend only on $n$, $c_\Gamma$ and $\kappa_{\max}$.
\end{lemma}

For the next lemma, we will use the notation 
\begin{equation}
\norm{\bm{g}}_{C^1(\T)}=\norm{\bm{g}}_{C(\T)}+\norm{\bm{g}'}_{C(\T)}, \; 
\norm{\bm{g}}_{C(\T)}=\max_{s\in \T}\abs{\bm{g}(s)}.
\end{equation}

We show the following estimate.
\begin{lemma}\label{Rintest1}
Let $\bm{R}$ be as in \eqref{Reps}. Suppose $m>0$ is an odd integer and $n\ge m+2$, and 
let $\bm{g}\in C^1(\T)$. Then, for sufficiently small $\epsilon$, we have
\begin{equation}\label{Intest_ineq}
\abs{\int^{1/2}_{-1/2}\frac{\bars^m}{\abs{\bm{R}}^n}\bm{g}(\bars)d\bars} \le \begin{cases}
c_\kappa \norm{\bm{g}}_{C^1(\T)}\abs{\log \epsilon} &\text{ if } n=m+2,\\
c_\kappa \norm{\bm{g}}_{C^1(\T)}\epsilon^{m+2-n} &\text{ if } n\ge m+3,
\end{cases}
\end{equation}
where the constants $c_\kappa$ depend only on $n$, $c_\Gamma$ and $\kappa_{\max}$.
\end{lemma}

\begin{proof}
First, we observe that
\begin{equation}
\begin{split}
\int^{1/2}_{-1/2}&\frac{\bars^m}{\abs{\bm{R}}^n}\bm{g}(\bars)d\bars = \bm{I}_1+\bm{I}_2, \\
&\bm{I}_1=\int_{-1/2}^{1/2}\frac{\bars^m}{\abs{\bm{R}}^n}(\bm{g}(\bars)-\bm{g}(0)) \ts d\bars,\\
&\bm{I}_2=\int_{-1/2}^{1/2}\bars^m\bigg(\frac{1}{\abs{\bm{R}}^n}-\frac{1}{(\bars^2+\epsilon^2)^{n/2}}\bigg) \bm{g}(0) \ts d\bars,
\end{split}
\end{equation}
where we used the fact that $m$ is odd in the last equality. We first estimate $\bm{I}_1$. Note that
\[ \abs{\bm{g}(\bars)-\bm{g}(0)}\le \abs{\bars}\norm{\bm{g}'}_{C(\T)}. \]
We have
\begin{equation}\label{I1inRestint1}
\abs{\bm{I}_1}\le \int_{-1/2}^{1/2}\frac{\bars^{m+1}}{\abs{\bm{R}}^n}\norm{\bm{g}'}_{C(\T)}d\bars \le
\begin{cases}
c_\kappa \norm{\bm{g}'}_{C(\T)}\abs{\log \epsilon} &\text{ if } n=m+2,\\
c_\kappa \norm{\bm{g}'}_{C(\T)}\epsilon^{m+2-n} &\text{ if } n\ge m+3,
\end{cases}
\end{equation}
where we used Lemma \ref{Rintest0}. We turn to $\bm{I}_2$. Note that
\[ \frac{1}{\abs{\bm{R}}^n}-\frac{1}{(\sqrt{\bars^2+\epsilon^2})^n}=
\bigg(\frac{1}{\abs{\bm{R}}}-\frac{1}{\sqrt{\bars^2+\epsilon^2}}\bigg) \sum_{l=0}^{n-1}\frac{1}{\abs{\bm{R}}^{l}\big(\sqrt{\bars^2+\epsilon^2}\big)^{n-1-l}}. \]
Using Lemma \ref{absRests}, we have
\[ \abs{\frac{1}{\abs{\bm{R}}}-\frac{1}{\sqrt{\bars^2+\epsilon^2}}}
=\frac{\abs{\abs{\bm{R}}-\sqrt{\bars^2+\epsilon^2}}}{\abs{\bm{R}}\sqrt{\bars^2+\epsilon^2}}
\le \frac{c_\kappa\bars^2}{c_\kappa(\bars^2+\epsilon^2)}. \]
Then, using Lemma \ref{absRests} again, we have
\begin{equation}\label{Rnsigman}
 \abs{\frac{1}{\abs{\bm{R}}^n}-\frac{1}{(\sqrt{\bars^2+\epsilon^2})^n}} \le \frac{c_\kappa\bars^2}{(\bars^2+\epsilon^2)^{(n+1)/2}}.
 \end{equation}
 Thus,
\begin{equation}\label{I2inRestint1}
\begin{split}
\abs{\bm{I}_2}&\le c_\kappa\norm{\bm{g}}_{C(\T)}\int_{-1/2}^{1/2}\frac{\bars^{m+2}}{(\bars^2+\epsilon^2)^{(n+1)/2}}d\bars \\
&\le \begin{cases}
3 c_\kappa\norm{\bm{g}}_{C(\T)}\abs{\log \epsilon}&\text{ if } n=m+2,\\
\pi c_\kappa\norm{\bm{g}}_{C(\T)}\epsilon^{m+2-n} &\text{ if } n\ge m+3,
\end{cases}
\end{split}
\end{equation}
where we used Lemma \ref{defints} in the last inequality. Combining \eqref{I1inRestint1} and \eqref{I2inRestint1}, we obtain the inequality \eqref{Intest_ineq}.
\end{proof}

The final integral we estimate is the following.
\begin{lemma}\label{Rintest2}
Suppose $m\ge 0$ is an even integer, $n$ is an integer such that $n\ge m+3$ and let $\bm{g}\in C^1(\T)$. Then, for sufficiently small $\epsilon$, we have
\begin{equation}\label{epsdmn}
\begin{split}
\abs{\int_{-1/2}^{1/2} \frac{\bars^m}{\abs{\bm{R}}^n}\bm{g}(\bars)d\bars -\epsilon^{m+1-n}d_{mn} \bm{g}(0)} 
&\le c_\kappa\norm{\bm{g}}_{C^1(\T)}\epsilon^{m+2-n},\\
d_{mn}&=\int_{-\infty}^\infty \frac{\tau^m}{(\tau^2+1)^{n/2}}d\tau,
\end{split}
\end{equation}
where the constant $c_\kappa$ depends only on $n$, $c_\Gamma$ and $\kappa_{\max}$. For odd $n$, we have
\begin{equation}
d_{mn}=\sum_{k=0}^{m/2}(-1)^k\begin{pmatrix} m/2\\ 
k \end{pmatrix} d_{0,n-k},\;
d_{0n}=2\frac{(n-3)!!}{(n-2)!!}.
\end{equation}
For certain values of $m$ and $n$, this yields
\begin{equation}
d_{03}=2, \; d_{05}=\frac{4}{3}, \; d_{07}= \frac{16}{15}, \; d_{25}=\frac{2}{3}, \; d_{27}=\frac{4}{15}.
\end{equation}
\end{lemma}

Note that Lemma \ref{Rintest2} immediately implies that, for $\bm{g}\in C^1(\Gamma_\epsilon)$, we have  
\begin{equation}\label{epsdmn_theta}
\begin{aligned}
&\abs{\int_0^{2\pi}\int_{-1/2}^{1/2} \frac{\bars^m}{\abs{\bm{R}}^n}\bm{g}(\bars,\theta) d\bars \ts d\theta -\epsilon^{m+1-n}d_{mn} \int_0^{2\pi}\bm{g}(0,\theta)\ts d\theta} \\
&\hspace{6cm} \le c_\kappa\max_{0\le\theta<2\pi}\norm{\bm{g}(\cdot,\theta)}_{C^1(\T)}\epsilon^{m+2-n}.
\end{aligned}
\end{equation}

\begin{proof}[Proof of Lemma \ref{Rintest2}:]
First, note that
\begin{equation}
\begin{split}
\int^{1/2}_{-1/2}& \frac{\bars^m \bm{g}(\bars)}{\abs{\bm{R}}^n}d\bars - \bm{g}(0)\int_{-\infty}^{\infty} \frac{\bars^m}{(\bars^2+\epsilon^2)^{n/2}}d\bars =\bm{I}_1+\bm{I}_2+\bm{I}_3, \\
\bm{I}_1 &=\int_{-1/2}^{1/2}\frac{\bars^m (\bm{g}(\bars)-\bm{g}(0))}{\abs{\bm{R}}^n}d\bars,\\
\bm{I}_2 &=\int_{-1/2}^{1/2}\bars^m\bigg(\frac{1}{\abs{\bm{R}}^n}-\frac{1}{(\bars^2+\epsilon^2)^{n/2}}\bigg)\bm{g}(0) \ts d\bars,\\
\bm{I}_3 &=2\bm{g}(0)\int_{1/2}^\infty \frac{\bars^m}{(\bars^2+\epsilon^2)^{n/2}}d\bars.
\end{split}
\end{equation}
We may estimate $\bm{I}_1$ and $\bm{I}_2$ in exactly the same way as in the proof of Lemma \ref{Rintest1}. We find that
\begin{align*}
\abs{\bm{I}_1}&\le c_\kappa \norm{\bm{g}'}_{C(\T)}\epsilon^{m+2-n}, \quad \abs{\bm{I}_2}\le c_\kappa\norm{\bm{g}}_{C(\T)}\epsilon^{m+2-n}.
\end{align*}
For $\bm{I}_3$, a simple estimation yields
\begin{align*}
\abs{\bm{I}_3}&\le 2\norm{\bm{g}}_{C(\T)}\int_{1/2}^\infty \frac{\bars^m d\bars}{(\bars^2+\epsilon^2)^{n/2}} \le 2\norm{\bm{g}}_{C(\T)}\int_{1/2}^\infty \bars^{m-n} d\bars = \frac{2^{n-m}}{n-m-1}\norm{\bm{g}}_{C(\T)}.
\end{align*}
Finally, we have
\[ \int_{-\infty}^{\infty} \frac{\bars^m}{(\bars^2+\epsilon^2)^{n/2}}d\bars = \epsilon^{m+1-n}\int_{-\infty}^{\infty} \frac{\tau^m}{(\tau^2+1)^{n/2}} d\tau \equiv \epsilon^{m+1-n}d_{nm}. \]
Combining all of the above, we obtain \eqref{epsdmn}. Note that, since $m$ is even,
\[d_{mn}=\int_{-\infty}^\infty\frac{(\tau^2+1-1)^{m/2}d\tau}{(\tau^2+1)^{n/2}} =\sum_{k=0}^{m/2}(-1)^k\begin{pmatrix} m/2\\ k\end{pmatrix} d_{0,n-k}. \]
For $n$ odd, we have
\begin{align*}
d_{0n}&=\int_{-\pi/2}^{\pi/2} \cos^{n-2}\varphi d\varphi=
\frac{n-3}{n-2}\int_{-\pi/2}^{\pi/2}\cos^{n-4}\varphi d\varphi\\
&=\cdots=\frac{(n-3)(n-5)\cdots 4\cdot 2}{(n-2)(n-4)\cdot 3}\int_{-\pi/2}^{\pi/2}\cos \varphi d\varphi
=2\frac{(n-3)!!}{(n-2)!!}.
\end{align*}
\end{proof}

Finally, we make note of the following lemma, which will be useful for estimating the centerline expression \eqref{SBT_asymp} to obtain the estimate \eqref{center_err_thm} of Theorem \ref{stokes_err_theorem}. 
Recalling the notation $\bm{R}_0(s,\bars) = \X(s) - \X(s+\bars)$, we show:
\begin{lemma}\label{center_est_lem}
Let $\bm{R}$ be as in \eqref{Reps} and suppose $n= 1$ or $n=3$. Then for $\bm{g}\in C^1(\T)$ and $\epsilon$ sufficiently small, we have 
\begin{equation}\label{cent_lem_eq}
\begin{aligned}
\bigg|\int_{-1/2}^{1/2} \frac{\bars^{n-1}}{\abs{\bm{R}}^n}\bm{g}(\bars)d\bars - &\int_{-1/2}^{1/2} \bigg(\frac{\bars^{n-1}}{\abs{\bm{R}_0}^n} \bm{g}(\bars)-\frac{\bm{g}(0)}{\abs{\bars} } \bigg)d\bars + \bm{g}(0)\log(\epsilon^2) + (n-1)\bm{g}(0) \bigg| \\
& \le \epsilon \abs{\log\epsilon} c_\kappa\norm{\bm{g}}_{C^1(\T)},
\end{aligned}
\end{equation}
where the constant $c_\kappa$ depends only on $n$, $c_\Gamma$, and $\kappa_{\max}$. 
\end{lemma} 

\begin{proof}
We begin by considering 
\begin{equation}\label{bmJ}
\bm{J} = \int_{-1/2}^{1/2} \bigg[\bigg(\frac{\bars^{n-1}}{\abs{\bm{R}}^n} - \frac{\bars^{n-1}}{\abs{\bm{R}_0}^n} \bigg) \bm{g}(\bars)+ \frac{\epsilon^2 \bm{g}(0)}{\abs{\bars}\sqrt{\bars^2+\epsilon^2} (\abs{\bars}+\sqrt{\bars^2+\epsilon^2})} + (n-1)\bm{g}(0)\bigg]d\bars.
\end{equation}

We may estimate $\bm{J}$ as 
\begin{align*}
\bm{J} &= \bm{J}_1 + \bm{J}_2+\bm{J}_3; \\
\bm{J}_1 &:= \int_{-1/2}^{1/2} \bigg(\frac{\bars^{n-1}}{\abs{\bm{R}}^n} - \frac{\bars^{n-1}}{\abs{\bm{R}_0}^n} \bigg) (\bm{g}(\bars)-\bm{g}(0))d\bars \\
\bm{J}_2 &:= \int_{-1/2}^{1/2} \bigg(\frac{1}{\abs{\bm{R}}} - \frac{1}{\abs{\bm{R}_0}} +\frac{\epsilon^2}{\abs{\bars}\sqrt{\bars^2+\epsilon^2} (\abs{\bars}+\sqrt{\bars^2+\epsilon^2})}  \bigg) \bm{g}(0) d\bars \\
\bm{J}_3 &:= \int_{-1/2}^{1/2} \bigg(\frac{\bars^{n-1}}{\abs{\bm{R}}^n} -\frac{1}{\abs{\bm{R}}} - \frac{\bars^{n-1}}{\abs{\bm{R}_0}^n} +\frac{1}{\abs{\bm{R}_0}} \bigg) \bm{g}(0) d\bars + (n-1) \bm{g}(0). 
\end{align*}

To estimate each $\bm{J}_i$, it will be convenient to define 
\begin{equation}\label{IR_def}
I_R := \frac{1}{\abs{\bm{R}}} - \frac{1}{\abs{\bm{R}_0}} = \frac{-\epsilon^2 - 2\epsilon\bars^2\bm{Q}\cdot\be_\rho}{\abs{\bm{R}}\abs{\bm{R}_0}(\abs{\bm{R}_0}+\abs{\bm{R}})},
\end{equation}
where we have used \eqref{CQ} and \eqref{Reps}. \\

Note that, using \eqref{IR_def} along with \eqref{Rlb} and \eqref{non_intersecting}, we have
\[ \abs{\frac{\bars^{n-1}}{\abs{\bm{R}}^n} - \frac{\bars^{n-1}}{\abs{\bm{R}_0}^n}} \le c_\kappa\abs{I_R}  \le c_\kappa\bigg( \frac{\epsilon^2}{\abs{\bars}(\bars^2+\epsilon^2)} + \frac{\epsilon \abs{\bars}}{\bars^2+\epsilon^2}\bigg). \]
Therefore, using that $\bm{g}\in C^1(\T)$, we can estimate $\bm{J}_1$ as 
\[ \abs{\bm{J}_1} \le c_\kappa \norm{\bm{g}'}_{C(\T)} \int_{-1/2}^{1/2} \frac{\epsilon^2+\epsilon\bars^2}{\bars^2+\epsilon^2} d\bars \le c_\kappa \epsilon \norm{\bm{g}'}_{C(\T)}.  \]

Furthermore, using the notation \eqref{IR_def}, the integrand of $\bm{J}_2$ satisfies
\begin{align*}
\abs{I_R +\frac{\epsilon^2}{\abs{\bars}\sqrt{\bars^2+\epsilon^2} (\abs{\bars}+\sqrt{\bars^2+\epsilon^2})} } &\le \abs{\frac{1}{\sqrt{\bars^2+\epsilon^2}} - \frac{1}{\abs{\bm{R}}}} \frac{\epsilon^2}{\abs{\bm{R}_0} (\abs{\bm{R}_0}+\abs{\bm{R}})}  \\
&\quad + \abs{\frac{1}{\abs{\bars}} - \frac{1}{\abs{\bm{R}_0}} } \frac{\epsilon^2}{\sqrt{\bars^2+\epsilon^2} (\abs{\bm{R}_0}+\abs{\bm{R}})}  \\ 
&\quad +\abs{\frac{1}{(\abs{\bars}+\sqrt{\bars^2+\epsilon^2})}- \frac{1}{(\abs{\bm{R}_0}+\abs{\bm{R}})}}\frac{\epsilon^2}{\abs{\bars}\sqrt{\bars^2+\epsilon^2}} \\
&\quad + \frac{c_\kappa\epsilon \bars^2}{\abs{\bm{R}}\abs{\bm{R}_0}(\abs{\bm{R}_0}+\abs{\bm{R}})} \\
&\le c_\kappa \frac{\epsilon^2+\epsilon\abs{\bars}}{\bars^2+\epsilon^2} ,
\end{align*}
where we have used \eqref{RQ} and \eqref{CQ} along with the triangle inequality to bound the difference expressions and \eqref{Rlb} and \eqref{non_intersecting} to bound each of the denominators. Then $\bm{J}_2$ satisfies
\begin{align*}
\abs{\bm{J}_2} \le c_\kappa \int_{-1/2}^{1/2} \frac{(\epsilon^2+\epsilon\abs{\bars})|\bm{g}(0)|}{\bars^2+\epsilon^2} d\bars \le c_\kappa \epsilon\abs{\log\epsilon} \norm{\bm{g}}_{C(\T)},
\end{align*}
by Lemma \ref{defints}. \\

If $n=1$, we are done. For $n=3$, we must also estimate $\bm{J}_3$. We have that $\bm{J}_3$ satisfies 
\begin{align*}
\abs{\bm{J}_3} &\le \norm{\bm{g}}_{C(\T)}\int_{-1/2}^{1/2} \bigg( \bigg|\frac{\bars^2+\epsilon^2-\abs{\bm{R}}^2}{\abs{\bm{R}}^3} - \frac{\bars^2-\abs{\bm{R}_0}^2}{\abs{\bm{R}_0}^3}\bigg| +\bigg|\frac{\epsilon^2}{\abs{\bm{R}}^3} - \frac{\epsilon^2}{\sqrt{\bars^2+\epsilon^2}^3} \bigg|  \bigg) d\bars \\
&\qquad + \norm{\bm{g}}_{C(\T)}\bigg|\int_{-1/2}^{1/2} \frac{\epsilon^2}{\sqrt{\bars^2+\epsilon^2}^3}d\bars - 2 \bigg| \\
&\le \norm{\bm{g}}_{C(\T)}\int_{-1/2}^{1/2} \bigg|\frac{2\bars^3\bm{Q}\cdot\be_t-\bars^4\abs{\bm{Q}}^2-2\epsilon\bars^2\bm{Q}\cdot\be_\rho}{\abs{\bm{R}}^3} - \frac{2\bars^3\bm{Q}\cdot\be_t-\bars^4\abs{\bm{Q}}^2}{\abs{\bm{R}_0}^3} \bigg| d\bars \\
&\qquad +c_\kappa \norm{\bm{g}}_{C(\T)}\int_{-1/2}^{1/2} \frac{\epsilon^2\bars^2}{(\bars^2+\epsilon^2)^2} d\bars  + \norm{\bm{g}}_{C(\T)}\bigg|\frac{2}{\sqrt{1+4\epsilon^2}} - 2 \bigg| \\
&\le c_\kappa\norm{\bm{g}}_{C(\T)}\int_{-1/2}^{1/2} (|\bars|^3+\bars^4)\bigg(\abs{I_R} \sum_{\ell=0}^2\frac{1}{\abs{\bm{R}_0}^\ell\abs{\bm{R}}^{2-\ell}}  \bigg) d\bars+c_\kappa \epsilon \abs{\log\epsilon} \norm{\bm{g}}_{C(\T)} \\
&\le c_\kappa\norm{\bm{g}}_{C(\T)}\int_{-1/2}^{1/2} \frac{\epsilon^2+\epsilon\bars^2 + \epsilon^2|\bars|+\epsilon|\bars|^3}{\abs{\bm{R}}^2} d\bars+c_\kappa \epsilon \abs{\log\epsilon} \norm{\bm{g}}_{C(\T)} \\ 
&\le c_\kappa \epsilon \abs{\log\epsilon} \norm{\bm{g}}_{C(\T)},
\end{align*}
using \eqref{Reps}, \eqref{CQ}, and \eqref{Rnsigman} in the second inequality, definition \ref{IR_def} along with Lemmas \ref{defints} and \ref{Rintest0} in the third inequality, and \eqref{non_intersecting} in the fourth inequality. \\

Finally, we show that the expression for $\bm{J}$ \eqref{bmJ} closely matches the expression on the left hand side of \eqref{cent_lem_eq}. We evaluate
\begin{align*}
\int_{-1/2}^{1/2}\bigg(\frac{\epsilon^2}{\abs{\bars}\sqrt{\bars^2+\epsilon^2} (\abs{\bars}+\sqrt{\bars^2+\epsilon^2})} -\frac{1}{\abs{\bars}} \bigg) d\bars &= -\int_{-1/2}^{1/2}\frac{\abs{\bars}\sqrt{\bars^2+\epsilon^2}+\bars^2}{\abs{\bars}\sqrt{\bars^2+\epsilon^2} (\abs{\bars}+\sqrt{\bars^2+\epsilon^2})} d\bars \\
&= -\int_{-1/2}^{1/2}\frac{1}{\sqrt{\bars^2+\epsilon^2} } d\bars \\
&= \log\bigg(\frac{\epsilon^2}{\frac{1}{2}+(\epsilon^2+\frac{1}{4})^{1/2}+\epsilon^2} \bigg).
\end{align*}

Using that 
\[ \bigg|\log\bigg(\frac{\epsilon^2}{\frac{1}{2}+(\epsilon^2+\frac{1}{4})^{1/2}+\epsilon^2}\bigg)  - \log(\epsilon^2)\bigg| \le c \epsilon^2, \]
we obtain Lemma \ref{center_est_lem}.
\end{proof}

\subsection{Slender body velocity residual}\label{SB_vel}
We will now use Lemmas  \ref{Rintest0}, \ref{Rintest1}, and \ref{Rintest2} to obtain an estimate on the non-conforming error -- the degree to which $\bu^{\SB}\big|_{\Gamma_\epsilon}(s,\theta)$ fails to satisfy the $\theta$-independence condition along the fiber surface $\Gamma_\epsilon$. We establish some estimates on $\bu^{\SB}$ and its derivatives along $\Gamma_{\epsilon}$. The derivative estimates will be needed in Section \ref{error_est_section} to obtain an actual error estimate between the slender body approximation $\bu^{\SB}$ and the true solution $\bu$. \\

We show the following proposition. 
\begin{proposition}\label{prop:uSBtheta}
Consider $\bu^{\SB}(\bx)$ for $\bx \in \Gamma_\epsilon$. For sufficiently small $\epsilon$, we have
\begin{equation}
\abs{\frac{1}{\epsilon}\frac{\p\bu^{\SB}}{\p \theta}}\le c_\kappa\norm{\bm{f}}_{C^1(\T)}\abs{\log \epsilon} 
\end{equation}
where the constant $c_\kappa$ depends only on $c_\Gamma$ and $\kappa_{\max}$.
\end{proposition}

\begin{proof} 
Write $\bx=\X(s)+\epsilon \be_\rho$. Using \eqref{stokes_SB}, we have:
\begin{equation}\label{ISDtheta}
\begin{split}
\frac{8\pi}{\epsilon}\frac{\p\bu^{\SB}}{\p \theta} &=\bm{I}_{\mc{S}}+\bm{I}_{\mc{D}}; \\
\bm{I}_\mc{S}&=\frac{1}{\epsilon}\frac{\p}{\p\theta}\int_{-1/2}^{1/2}\mc{S}(\bm{R})\bm{f}(s+\bars) d\bars,\\
\bm{I}_\mc{D}&=\frac{1}{\epsilon}\frac{\p}{\p\theta}\int_{-1/2}^{1/2}\frac{\epsilon^2}{2}\mc{D}(\bm{R})\bm{f}(s+\bars) d\bars.
\end{split}
\end{equation}
We first consider $\bm{I}_{\mc{S}}$. Using \eqref{SD} and \eqref{thetaderiv}, we have
\begin{align*}
\bm{I}_{\mc{S}} &=\bm{I}_{\mc{S},1}+\bm{I}_{\mc{S},2};\\
\bm{I}_{\mc{S},1} &=-\int_{-1/2}^{1/2} \bigg(\frac{\bm{R}\cdot\be_\theta}{\abs{\bm{R}}^3}\bm{f}
+3\frac{(\bm{R}\cdot\be_\theta)(\bm{R}\cdot\bm{f})}{\abs{\bm{R}}^5}\bm{R} \bigg)d\bars,\\
\bm{I}_{\mc{S},2}&=\int_{-1/2}^{1/2} \bigg(\frac{(\bm{R}\cdot \bm{f})\be_\theta+(\be_\theta\cdot \bm{f})\bm{R}}{\abs{\bm{R}}^3}\bigg)d\bars.
\end{align*}
We estimate $\bm{I}_{\mc{S},1}$. First, note from \eqref{CQ} that
\[\abs{\bm{R}\cdot \be_\theta}=\bars^2\abs{\bm{Q}\cdot \be_\theta}\le \frac{\kappa_{\max}}{2}\bars^2. \]
Applying Lemma \ref{Rintest0}, we then have
\begin{equation}\label{IS1}
\abs{\bm{I}_{\mc{S},1}}\le \int_{-1/2}^{1/2}4\frac{\abs{\bm{R}\cdot\be_\theta}}{\abs{\bm{R}}^3}\norm{\bm{f}}_{C(\T)} d\bars \le c_\kappa\norm{\bm{f}}_{C(\T)}\abs{\log \epsilon}.
\end{equation}
Turning to $\bm{I}_{\mc{S},2}$, we note that $(\bm{R}\cdot \bm{f})\be_\theta+(\be_\theta\cdot \bm{f})\bm{R}=\epsilon \bm{g}_0+\bars\bm{g}_1+\bars^2\bm{g}_2$, where
\begin{equation}\label{g012def}
\begin{split}
\bm{g}_0(\bars;s) &=\be_\rho(s)(\be_\theta(s)\cdot \bm{f}(s+\bars))+ \be_\theta(s)(\be_\rho(s)\cdot \bm{f}(s+\bars)),\\
\bm{g}_1(\bars;s) &=\be_t(s)(\be_\theta(s)\cdot \bm{f}(s+\bars))+ \be_\theta(s)(\be_t(s)\cdot \bm{f}(s+\bars)),\\
\bm{g}_2 &=\bm{Q}(\be_\theta\cdot \bm{f})+\be_\theta(\bm{Q}\cdot \bm{f}),
\end{split}
\end{equation}
and we have written out the explicit dependence of $\bm{g}_0$ and $\bm{g}_1$ on $\bars$ and $s$. Applying Lemma \ref{Rintest0} and \eqref{CQ}, we have
\[\abs{\int_{-1/2}^{1/2} \frac{\bars^2 \bm{g}_2}{\abs{\bm{R}}^3}d\bars}
\le \norm{\bm{g}_2(\cdot;s)}_{C(\T)}\int_{-1/2}^{1/2} \frac{\bars^2}{\abs{\bm{R}}^3}d\bars
\le c_\kappa\norm{\bm{f}}_{C(\T)}\abs{\log \epsilon}. \]

Using Lemma \ref{Rintest1}, we have
\[ \abs{\int_{-1/2}^{1/2} \frac{\bars \bm{g}_1(\bars;s)}{\abs{\bm{R}}^3}d\bars}
\le c_\kappa\norm{\bm{g}_1(\cdot;s)}_{C^1(\T)}\abs{\log \epsilon}\le c_\kappa\norm{\bm{f}}_{C^1(\T)}\abs{\log \epsilon}. \]

Finally, using Lemma \ref{Rintest2} with $m=0$, $n=3$, we have 
\begin{equation}\label{hdef}
\begin{split}
\abs{\int_{-1/2}^{1/2} \frac{\epsilon \bm{g}_0(\bars;s)}{\abs{\bm{R}}^3}d\bars - \frac{2}{\epsilon}\bm{g}_0(0; s)}
&\le c_\kappa\norm{\bm{g}_0(\cdot;s)}_{C^1(\T)}\le c_\kappa\norm{\bm{f}}_{C^1(\T)};\\
\bm{g}_0(0;s)&=\be_\rho(s)(\be_\theta(s)\cdot \bm{f}(s))+\be_\theta(s)(\be_\rho(s)\cdot \bm{f}(s)) =: \bm{h}(s).
\end{split}
\end{equation}

Combining the above estimates, we obtain
\begin{equation}\label{IS2}
\abs{\bm{I}_{\mc{S},2}-\frac{2}{\epsilon}\bm{h}(s)}\le c_\kappa\norm{\bm{f}}_{C^1(\T)}\abs{\log\epsilon}.
\end{equation}

Finally, combining \eqref{IS1} and \eqref{IS2}, we have
\begin{equation}\label{ISh}
\abs{\bm{I}_{\mc{S}}-\frac{2}{\epsilon}\bm{h}(s)}\le 
\abs{\bm{I}_{\mc{S},1}}+\abs{\bm{I}_{\mc{S},2}-\frac{2}{\epsilon}\bm{h}(s)}\le
c_\kappa\norm{\bm{f}}_{C^1(\T)}\abs{\log\epsilon}.
\end{equation}

We next consider $\bm{I}_{\mc{D}}$ in \eqref{ISDtheta}. We write
\begin{equation}\label{IDdef}
\begin{split}
\bm{I}_{\mc{D}}&=\frac{3\epsilon^2}{2}\big(\bm{I}_{\mc{D},1}+\bm{I}_{\mc{D},2} \big),\\
\bm{I}_{\mc{D},1}&=\int_{-1/2}^{1/2} \bigg(-\frac{\bm{R}\cdot\be_\theta}{\abs{\bm{R}}^5}\bm{f}+5\frac{(\bm{R}\cdot\be_\theta)(\bm{R}\cdot\bm{f})}{\abs{\bm{R}}^7}\bm{R} \bigg)d\bars,\\
\bm{I}_{\mc{D},2} &= -\int_{-1/2}^{1/2} \bigg(\frac{(\bm{R}\cdot \bm{f})\be_\theta+(\be_\theta \cdot \bm{f})\bm{R}}{\abs{\bm{R}}^5}\bigg) d\bars.
\end{split}
\end{equation}
Following the same steps used to estimate $\bm{I}_{\mc{S},1}$ and $\bm{I}_{\mc{S},2}$ in \eqref{IS1} and \eqref{IS2}, we obtain
\begin{align*}
\abs{\bm{I}_{\mc{D},1}}&\le c_\kappa\norm{\bm{f}}_{C(\T)}\epsilon^{-2},\\
\abs{\bm{I}_{\mc{D},2}+\frac{4}{3\epsilon^3}\bm{h}(s)}&\le c_\kappa\norm{\bm{f}}_{C^1(\T)}\epsilon^{-2},
\end{align*}
where $\bm{h}(s)$ was given in \eqref{hdef}. In particular, in the second estimate, we used Lemma \ref{Rintest2} with $m=0$ and $n=5$. \\

Combining the above, we have
\begin{equation}\label{IDh}
\abs{\bm{I}_{\mc{D}}+\frac{2}{\epsilon}\bm{h}(s)}\le \frac{3\epsilon^2}{2}\bigg(\abs{\bm{I}_{\mc{D},1}}+\abs{\bm{I}_{\mc{D},2}+\frac{4}{3\epsilon^3}\bm{h}(s)}\bigg) \le c_\kappa\norm{\bm{f}}_{C^1(\T)}.
\end{equation}

We finally estimate \eqref{ISDtheta} as 
\begin{align*}
\abs{\frac{1}{\epsilon}\frac{\p\bu^{\SB}}{\p \theta}}&\le \frac{1}{8\pi}\bigg(\abs{\bm{I}_{\mc{S}} -\frac{2}{\epsilon}\bm{h}(s)}+\abs{\bm{I}_{\mc{D}}+\frac{2}{\epsilon}\bm{h}(s)} \bigg) \le c_\kappa\norm{\bm{f}}_{C^1(\T)}\abs{\log \epsilon},
\end{align*}
where we used \eqref{ISh} and \eqref{IDh} in the last inequality.
\end{proof}

We next show the following proposition.
\begin{proposition}\label{prop:uSBstheta}
Consider $\bu^{\SB}(\bx)$ for $\bx\in \Gamma_\epsilon$. The following estimate holds for sufficiently small $\epsilon$:
\begin{equation}
\abs{\frac{\p}{\p\theta}\bigg(\frac{\p\bu^{\SB}}{\p s} - \kappa_3\frac{\p\bu^{\SB}}{\p\theta} \bigg)}\le c_\kappa\norm{\bm{f}}_{C^1(\T)},
\end{equation}
where the constant $c_\kappa$ depends only on the constants $c_\Gamma$ and $\kappa_{\max}$.
\end{proposition}

\begin{proof}
First, note that 
\begin{equation}\label{uSBstheta}
\begin{split}
\frac{\p}{\p\theta}&\bigg(\frac{\p\bu^{\SB}}{\p s}-\kappa_3\frac{\p\bu^{\SB}}{\p\theta}\bigg) = \frac{1}{8\pi}\bigg((1-\epsilon\wh\kappa)\frac{\p \bm{I}^{\SB}}{\p \theta} - \epsilon\frac{\p \wh{\kappa}}{\p \theta}\bm{I}^{\SB} \bigg);\\
&\bm{I}^{\SB}=\frac{8\pi}{1-\epsilon\wh{\kappa}}\bigg(\frac{\p \bu^{\SB}}{\p s} -\kappa_3\frac{\p\bu^{\SB}}{\p \theta}\bigg) = \bm{I}_{\mc{S}}+\frac{3\epsilon^2}{2}\bm{I}_{\mc{D}},\\
&\bm{I}_{\mc{S}}=\int_{-1/2}^{1/2}\bigg(\frac{-\bm{R}\cdot\be_t}{\abs{\bm{R}}^3}\bm{f} + \frac{(\bm{R}\cdot \bm{f})\be_t +(\be_t\cdot \bm{f})\bm{R}}{\abs{\bm{R}}^3} - 3\frac{(\bm{R}\cdot\be_t)(\bm{R}\cdot\bm{f})}{\abs{\bm{R}}^5}\bm{R} \bigg)d\bars,\\
&\bm{I}_{\mc{D}}= \int_{-1/2}^{1/2}\bigg(\frac{-\bm{R}\cdot\be_t}{\abs{\bm{R}}^5}\bm{f} -\frac{(\bm{R}\cdot \bm{f})\be_t +(\be_t\cdot \bm{f})\bm{R}}{\abs{\bm{R}}^5} +5\frac{(\bm{R}\cdot\be_t)(\bm{R}\cdot\bm{f})}{\abs{\bm{R}}^7}\bm{R} \bigg)d\bars,
\end{split}
\end{equation}
where we used \eqref{stokes_SB} and \eqref{sderiv} to obtain the expression for $\bm{I}_{\mc S}$ and $\bm{I}_{\mc D}$.\\

Let us estimate $\bm{I}^{\SB}$. We have
\begin{align*}
\abs{\bm{I}_{\mc{S}}} &\le \int_{-1/2}^{1/2}\abs{\frac{-\bm{R}\cdot\be_t}{\abs{\bm{R}}^3}\bm{f} +\frac{(\bm{R}\cdot \bm{f})\be_t+(\be_t\cdot \bm{f})\bm{R}}{\abs{\bm{R}}^3} - 3\frac{(\bm{R}\cdot\be_t)(\bm{R}\cdot\bm{f})}{\abs{\bm{R}}^5}\bm{R}} d\bars \\
&\le \int_{-1/2}^{1/2} \frac{6\norm{\bm{f}}_{C(\T)}}{\abs{\bm{R}}^2}d\bars
\le c_\kappa \norm{\bm{f}}_{C(\T)}\epsilon^{-1},
\end{align*}
where we used Lemma \ref{Rintest0} in the last inequality. Likewise,
\begin{align*}
\abs{\bm{I}_{\mc{D}}} &\le \int_{-1/2}^{1/2}\abs{\frac{-\bm{R}\cdot\be_t}{\abs{\bm{R}}^5}\bm{f} -\frac{(\bm{R}\cdot \bm{f})\be_t +(\be_t\cdot \bm{f})\bm{R}}{\abs{\bm{R}}^5} +5\frac{(\bm{R}\cdot\be_t)(\bm{R}\cdot\bm{f})}{\abs{\bm{R}}^7}\bm{R}} d\bars \\
&\le \int_{-1/2}^{1/2} \frac{8\norm{\bm{f}}_{C(\T)}}{\abs{\bm{R}}^4}d\bars \le c_\kappa \norm{\bm{f}}_{C(\T)}\epsilon^{-3},
\end{align*}
where we again used Lemma \ref{Rintest0} in the last inequality. Using the above estimates, we have
\begin{equation}\label{CI}
\abs{\bm{I}^{\SB}} \le \abs{\bm{I}_{\mc S}}+\frac{3\epsilon^2}{2}\abs{\bm{I}_{\mc D}}
\le c_\kappa\norm{\bm{f}}_{C(\T)}\epsilon^{-1}.
\end{equation}

We now estimate $\p\bm{I}^{\SB}/\p \theta$. We have
\begin{equation}\label{ISstheta}
\begin{split}
\frac{\p \bm{I}_{\mc S}}{\p \theta} &= \epsilon\big(\bm{I}_{\mc{S},1}+\bm{I}_{\mc{S},2}+\bm{I}_{\mc{S},3}+\bm{I}_{\mc{S},4}\big);\\
\bm{I}_{\mc{S},1} &= 3\int_{-1/2}^{1/2}\bigg(\frac{(\bm{R}\cdot\be_\theta)}{\abs{\bm{R}}^5} ((\bm{R}\cdot\be_t)\bm{f}-(\bm{R}\cdot \bm{f})\be_t -(\be_t\cdot \bm{f})\bm{R}) \bigg) d\bars,\\
\bm{I}_{\mc{S},2} &= \int_{-1/2}^{1/2}\frac{(\be_\theta\cdot \bm{f})\be_t +(\be_t\cdot \bm{f})\be_\theta}{\abs{\bm{R}}^3}d\bars,\\
\bm{I}_{\mc{S},3} &= -3\int_{-1/2}^{1/2}\frac{(\bm{R}\cdot\be_t)}{\abs{\bm{R}}^5} \big((\bm{R}\cdot\bm{f})\bm{e}_\theta+(\bm{e}_\theta\cdot\bm{f})\bm{R} \big) d\bars,\\
\bm{I}_{\mc{S},4} &= 15\int_{1/2}^{1/2} \frac{(\bm{R}\cdot\be_t)(\bm{R}\cdot\bm{f})(\bm{R}\cdot\be_\theta)}{\abs{\bm{R}}^7}\bm{R}d\bars.
\end{split}
\end{equation}
We estimate each term in turn. Using \eqref{Reps} and Lemma \ref{Rintest0}, we have that $\bm{I}_{\mc{S},1}$ satisfies
\begin{equation}\label{IS1stheta}
\abs{\bm{I}_{\mc{S},1}} \le c_\kappa\int_{-1/2}^{1/2}\frac{\bars^2}{\abs{\bm{R}}^4}\norm{\bm{f}}_{C(\T)}d\bars
\le c_\kappa\norm{\bm{f}}_{C(\T)}\epsilon^{-1}.
\end{equation}


Next, to estimate $\bm{I}_{\mc{S},2}$, we define $\bm{g}_1$ as in \eqref{g012def}. Using Lemma \ref{Rintest2} with $m=0, n=3$, we have
\begin{equation}\label{hsdef}
\begin{split}
\abs{\bm{I}_{\mc{S},2}-\frac{2}{\epsilon^2}\bm{h}(s)} &\le c_\kappa\norm{\bm{g}_1(\cdot;s)}_{C^1(\T)}\epsilon^{-1}
\le c_\kappa\norm{\bm{f}}_{C^1(\T)}\epsilon^{-1}; \\
\bm{h}(s)&=\bm{g}_1(0;s) = (\be_\theta(s,\theta)\cdot \bm{f}(s))\be_t(s)+(\be_t(s)\cdot \bm{f}(s))\be_\theta(s,\theta).
\end{split}
\end{equation}

To estimate $\bm{I}_{\mc{S},3}$, let
\begin{align*}
\bm{I}_{\mc{S},3}&=\bm{I}_{\mc{S},31}+\bm{I}_{\mc{S},32}; \\
\bm{I}_{\mc{S},31} &= -3\int_{-1/2}^{1/2}\frac{\bars^2(\bm{Q}\cdot\be_t)}{\abs{\bm{R}}^5} \big((\bm{R}\cdot\bm{f})\be_\theta+(\be_\theta\cdot\bm{f})\bm{R} \big) d\bars,\\
\bm{I}_{\mc{S},32}&= -3\int_{-1/2}^{1/2}\frac{\bars}{\abs{\bm{R}}^5} \big((\bm{R}\cdot\bm{f})\be_\theta+(\be_\theta\cdot\bm{f})\bm{R} \big) d\bars.
\end{align*}
Using Lemma \ref{Rintest0}, $\bm{I}_{\mc{S},31}$ may be estimated as
\[ \abs{\bm{I}_{\mc{S},31}}\le c_\kappa\int_{-1/2}^{1/2}\frac{\bars^2}{\abs{\bm{R}}^4}\norm{\bm{f}}_{C(\T)}d\bars \le c_\kappa\norm{\bm{f}}_{C(\T)}\epsilon^{-1}. \]

To estimate $\bm{I}_{\mc{S},32}$, define $\bm{g}_0$, $\bm{g}_1$, and $\bm{g}_2$ as in \eqref{g012def}. We first have
\[\abs{\int_{-1/2}^{1/2} \frac{\bars^3 \bm{g}_2}{\abs{\bm{R}}^5}d\bars} \le c_\kappa\int_{-1/2}^{1/2} \frac{|\bars|^3\norm{\bm{f}}_{C(\T)}}{\abs{\bm{R}}^5}d\bars \le c_\kappa\norm{\bm{f}}_{C(\T)}\epsilon^{-1}, \]
where we used \eqref{Reps} and Lemma \ref{Rintest0}. Next, we have
\[ \abs{\int_{-1/2}^{1/2} \frac{\bars^2 \bm{g}_1}{\abs{\bm{R}}^5}d\bars - \frac{2}{3\epsilon^2}\bm{h}(s)} \le c_\kappa\norm{\bm{g}_1(\cdot;s)}_{C^1(\T)}\epsilon^{-1} \le c_\kappa\norm{\bm{f}}_{C^1(\T)}\epsilon^{-1}, \]
where we used Lemma \ref{Rintest2} with $m=2, n=5$ and $\bm{h}(s)$ as defined in \eqref{hsdef}. For $\bm{g}_2$, we have
\[\abs{\int_{-1/2}^{1/2} \frac{\epsilon \bars \bm{g}_0}{\abs{\bm{R}}^5}d\bars} \le c_\kappa\norm{\bm{g}_0(\cdot;s)}_{C^1(\T)}\epsilon^{-1}\le c_\kappa\norm{\bm{f}}_{C^1(\T)}\epsilon^{-1}, \]
where we used Lemma \ref{Rintest1}. Combining the above estimates, we have
\begin{equation}\label{IS3stheta}
\abs{\bm{I}_{\mc{S},3}+\frac{2}{\epsilon^2}\bm{h}(s)}\le c_\kappa\norm{\bm{f}}_{C^1(\T)}\epsilon^{-1}.
\end{equation}

Finally, we estimate $\bm{I}_{\mc{S},4}$ as 
\begin{equation}\label{IS4stheta}
\abs{\bm{I}_{\mc{S},4}}\le c_\kappa\int_{-1/2}^{1/2}\frac{\bars^2}{\abs{\bm{R}}^4}\norm{\bm{f}}_{C(\T)}d\bars \le c_\kappa\norm{\bm{f}}_{C(\T)}\epsilon^{-1}.
\end{equation}

Using the estimates \eqref{IS1stheta}, \eqref{hsdef}, \eqref{IS3stheta} and \eqref{IS4stheta} in \eqref{ISstheta}, 
we obtain 
\begin{equation}\label{ISthetaest}
\abs{\frac{\p \bm{I}_{\mc{S}}}{\p \theta}}\le c_\kappa\norm{\bm{f}}_{C^1(\T)}.
\end{equation}

We may estimate $\partial \bm{I}_{\mc{D}}/\partial \theta$ in exactly the same way. We have
\begin{equation}\label{IDstheta}
\begin{split}
\frac{\p \bm{I}_{\mc D}}{\p \theta} &= \epsilon \big(\bm{I}_{\mc{D},1}+\bm{I}_{\mc{D},2} +\bm{I}_{\mc{D},3}+\bm{I}_{\mc{D},4}\big); \\
\bm{I}_{\mc{D},1} &= 5\int_{-1/2}^{1/2}\bigg(\frac{(\bm{R}\cdot\bm{e}_\theta)}{\abs{\bm{R}}^7} \big((\bm{R}\cdot\be_t)\bm{f}+(\bm{R}\cdot \bm{f})\be_t +(\be_t\cdot \bm{f})\bm{R} \big) \bigg) d\bars, \\
\bm{I}_{\mc{D},2} &= -\int_{-1/2}^{1/2}\frac{(\be_\theta\cdot \bm{f})\be_t +(\be_t\cdot \bm{f})\be_\theta}{\abs{\bm{R}}^5}d\bars,\\
\bm{I}_{\mc{D},3} &= 5\int_{-1/2}^{1/2}\frac{(\bm{R}\cdot\be_t)}{\abs{\bm{R}}^7} \big((\bm{R}\cdot\bm{f})\be_\theta+(\be_\theta\cdot\bm{f})\bm{R} \big) d\bars,\\
\bm{I}_{\mc{D},4}&=-35\int_{1/2}^{1/2} \frac{(\bm{R}\cdot\be_t)(\bm{R}\cdot\bm{f})(\bm{R}\cdot\be_\theta)}{\abs{\bm{R}}^9}\bm{R} \ts d\bars.
\end{split}
\end{equation}

The estimation of $\bm{I}_{\mc{D},1}$ follows the same pattern as that for $\bm{I}_{\mc{S},1}$ obtained in \eqref{IS1stheta}:
\[\abs{\bm{I}_{\mc{D},1}}\le c_\kappa\norm{f}_{C(\T)}\epsilon^{-3}. \]

The estimation of $\bm{I}_{\mc{D},2}$ is similar to \eqref{hsdef}:
\[ \abs{\bm{I}_{\mc{D},2}+\frac{4}{3\epsilon^4}\bm{h}(s)}\le c_\kappa\norm{\bm{f}}_{C^1(\T)}\epsilon^{-3}, \]
where we used Lemma \ref{Rintest2} with $m=0, n=5$. We estimate $\bm{I}_{\mc{D},3}$ following the steps of estimate \eqref{IS3stheta}. We obtain
\[ \abs{\bm{I}_{\mc{D},3}-\frac{4}{3\epsilon^4}\bm{h}(s)}\le c_\kappa\norm{\bm{f}}_{C^1(\T)}\epsilon^{-3}, \]
where we used Lemma \ref{Rintest2} with $m=2, n=7$. Finally, the estimation of $\bm{I}_{\mc{D},4}$ is similar to \eqref{IS4stheta}, yielding
\[ \abs{\bm{I}_{\mc{D},4}}\le c_\kappa\norm{\bm{f}}_{C(\T)}\epsilon^{-3}. \]
Combining the above estimates, we obtain
\begin{equation}\label{IDthetaest}
\abs{\frac{\p \bm{I}_{\mc{D}}}{\p \theta}}\le c_\kappa\norm{\bm{f}}_{C^1(\T)}\epsilon^{-2}.
\end{equation}

Combining \eqref{ISthetaest} and \eqref{IDthetaest} and recalling the definition of $\bm{I}^{\SB}$ in \eqref{uSBstheta}, we have
\begin{equation}\label{CItheta}
\abs{\frac{\p \bm{I}^{\SB}}{\p \theta}}\le \abs{\frac{\p \bm{I}_{\mc S}}{\p \theta}} +\frac{3\epsilon^2}{2}\abs{\frac{\p \bm{I}_{\mc D}}{\p\theta}} \le c_\kappa \norm{\bm{f}}_{C^1(\T)}.
\end{equation}
We may finally use \eqref{CI} and \eqref{CItheta} together in \eqref{uSBstheta} to obtain
\begin{equation}\label{mixed_est}
\begin{split}
\abs{\frac{\p}{\p \theta}\bigg(\frac{\p \bu^{\SB}}{\p s}-\kappa_3\frac{\p \bu^{\SB}}{\p \theta} \bigg) }
&\le \frac{1}{8\pi}\bigg((1+\epsilon|\wh\kappa|)\abs{\frac{\p \bm{I}^{\SB}}{\p \theta}}+\epsilon\abs{\frac{\p \wh{\kappa}}{\p \theta}}\abs{\bm{I}^{\SB}} \bigg) \le c_\kappa \norm{\bm{f}}_{C^1(\T)},
\end{split}
\end{equation}
where, in the last inequality, we used that
\[ \epsilon\abs{\wh\kappa}\le 2\epsilon\kappa_{\max} \le \frac{1}{4}, \]
by \eqref{kappahat} and \eqref{slender_body}, and
\[\abs{\frac{\p \wh{\kappa}}{\p \theta}} = \abs{-\kappa_1\sin \theta+\kappa_2\cos\theta} \le 2\sqrt{\kappa_1^2+\kappa_2^2} = 2\kappa\le  2\kappa_{\max}, \]
by \eqref{kappahat} and \eqref{kappa12}. 
\end{proof}

With Propositions \ref{prop:uSBtheta} and \ref{prop:uSBstheta}, we are finally equipped to estimate the degree to which $\bu^{\SB}$ fails to satisfy the $\theta$-independence condition along $\Gamma_{\epsilon}$. We define the residual $\bu^{\rm r}(s,\theta)$ as 
\begin{equation}\label{ur}
\bu^{\rm r}(\theta,s) = \bu^{\rm SB}(\epsilon,\theta,s) - \frac{1}{2\pi}\int_0^{2\pi} \bu^{\SB}(\epsilon,\varphi,s) \ts d\varphi.
\end{equation}
Note that the function $\bu^{\rm r}$ measures the deviation of $\bu^{\SB}$ from a $\theta$-independent function. We show the following estimates for $\bu^{\rm r}$.
\begin{proposition}\label{ur_and_derivs}
Consider the residual $\bu^{\rm r}$ defined in \eqref{ur}. For sufficiently small $\epsilon$, we have 
\begin{align}
\label{urest}
\abs{\bu^{\rm r}}&\le c_\kappa \norm{\bm{f}}_{C^1(\T)}\epsilon\abs{\log\epsilon},\\
\label{urtheta}
\abs{\frac{1}{\epsilon}\frac{\p \bu^{\rm r}}{\p \theta}}&\le c_\kappa \norm{\bm{f}}_{C^1(\T)}\abs{\log\epsilon},\\
\label{urs}
\abs{\frac{\p \bu^{\rm r}}{\p s}}&\le c_\kappa \norm{\bm{f}}_{C^1(\T)},
\end{align}
where the constants $c_\kappa$ depend only on $c_\Gamma$ and $\kappa_{\max}$.
\end{proposition}

Note that the estimate \eqref{urest} provides a rigorous proof of the asymptotic calculations done by Johnson in \cite{johnson1980improved}. 

\begin{proof}
Let $\bu^{\rm r}=(u^{\rm r}_1,u^{\rm r}_2, u^{\rm r}_3)$ and likewise for $\bm{u}^{\rm SB}$. We work component-wise. For each fixed $s$, we can find $\theta_0$ satisfying
\[ u^{\SB}_k(\epsilon,\theta_0,s)=\frac{1}{2\pi}\int_0^{2\pi} u^{\SB}_k(\epsilon,\varphi,s) \ts d\varphi. \]
Thus we can write
\[ u^{\rm r}_k(\theta,s)=u^{\SB}_k(\epsilon,\theta,s) - u^{\SB}_k(\epsilon,\theta_0,s) = \int_{\theta_0}^\theta \frac{\p u^{\SB}_k}{\p \theta}(\epsilon,\varphi,s)\ts d\varphi. \]
Using Proposition \ref{prop:uSBtheta}, we have
\begin{equation}
\begin{split}
\abs{u^{\rm r}_k(\theta,s)}&\le \int_{\theta_0}^\theta \abs{\frac{\p u^{\SB}_k}{\p \theta}(\epsilon,\varphi,s)} \ts d\varphi 
\le c_\kappa\abs{\theta-\theta_0}\norm{\bm{f}}_{C^1(\T)}\epsilon \abs{\log\epsilon}\\
&\le c_\kappa\pi\norm{\bm{f}}_{C^1(\T)}\epsilon \abs{\log\epsilon},
\end{split}
\end{equation}
where, in the last equality, we used the fact that $\theta$ and $\theta_0$ are at most $\pi$ apart. This establishes \eqref{urest}. \\

The estimate \eqref{urtheta} is a direct consequence of Proposition \ref{prop:uSBtheta}. \\

We finally establish \eqref{urs}. For each fixed $s$, we find a $\theta_1$ satisfying
\[ \frac{\p u^{\SB}_k}{\p s}(\epsilon,\theta_1,s) = \frac{1}{2\pi}\int_{0}^{2\pi} \frac{\p u^{\SB}_k}{\p s}(\epsilon,\varphi,s)\ts d\varphi. \]
Then we can write
\begin{align*}
\frac{\p u^{\rm r}_k}{\p s}(\theta,s)&=\frac{\p u^{\SB}_k}{\p s}(\epsilon,\theta,s) - \frac{\p u^{\SB}_k}{\p s}(\epsilon,\theta_1,s)
=\int_{\theta_1}^{\theta} \frac{\p}{\p \theta}\bigg(\frac{\p u^{\SB}_k}{\p s}\bigg)(\epsilon,\varphi,s)\ts d\varphi\\
&=\int_{\theta_1}^{\theta} \frac{\p }{\p \theta}\bigg(\frac{\p u^{\SB}_k}{\p s}-\kappa_3\frac{\p u^{\SB}_k}{\p \theta} \bigg)(\epsilon,\varphi,s) \ts d\varphi\\
&\quad +\kappa_3\bigg(\frac{\p u^{\SB}_k}{\p \theta}(\epsilon,\theta,s) - \frac{\p u^{\SB}_k}{\p \theta}(\epsilon,\theta_1,s) \bigg)
\end{align*}
Thus, using Proposition \ref{prop:uSBstheta} and Proposition \ref{prop:uSBtheta}, we have
\begin{equation}
\abs{\frac{\p u^{\rm r}_k}{\p s}(\theta,s)} \le c_\kappa\abs{\theta-\theta_1}\norm{\bm{f}}_{C^1(\T)} + 2\abs{\kappa_3} c_\kappa\norm{\bm{f}}_{C^1(\T)}\epsilon\abs{\log\epsilon}.
\end{equation}
Noting that $\abs{\theta-\theta_1}\le \pi$ and $\abs{\kappa_3}\leq \pi$ by Lemma \ref{lemmaorthonormal}, we obtain the desired estimate.
\end{proof}

Finally, using Lemma \ref{center_est_lem}, we show the following residual estimate for the difference $\bu^{\SB}(s,\theta)-\bu^{\SB}_C(s)$ between the slender body approximation \eqref{stokes_SB} on the fiber surface and the asymptotic centerline expression \eqref{SBT_asymp}.
\begin{proposition}\label{centerline_prop}
Let $\bu^{\SB}(s,\theta)$ be \eqref{stokes_SB} evaluated on the slender body surface $\Gamma_\epsilon$, and let $\bu^{\SB}_C(s)$ be the centerline equation \eqref{SBT_asymp}. Then the difference $\bu^{\SB}(s,\theta)-\bu^{\SB}_C(s)$ satisfies
\begin{equation}\label{centerline_resid}
\abs{\bu^{\SB}(s,\theta)-\bu^{\SB}_C(s)} \le c_\kappa\epsilon\abs{\log\epsilon}\norm{\bm{f}}_{C^1(\T)},
\end{equation}
where $c_\kappa$ depends only on $c_\Gamma$ and $\kappa_{\max}$.
\end{proposition}

\begin{proof}
We begin by writing the Stokeslet term of $\bu^{SB}(s,\theta)$ as
\begin{equation}
\begin{aligned}
\int_{-1/2}^{1/2}&\mc{S}(\bm{R})\bm{f}(s+\bars)d\bars = \mc{S}_1 + \mc{S}_2; \\
\mc{S}_1 &:= \int_{-1/2}^{1/2} \frac{\bm{f}(s+\bars)}{\abs{\bm{R}}} d\bars, \quad \mc{S}_2 := \int_{-1/2}^{1/2} \frac{\bm{R}\bm{R}^{\rm T}}{\abs{\bm{R}}^3} \bm{f}(s+\bars) d\bars.
\end{aligned}
\end{equation}

Now, letting 
\begin{equation}\label{cent_int1}
\bm{J}_{\mc{S},1} = \int_{-1/2}^{1/2} \bigg(\frac{ \bm{f}(s+\bars)}{\abs{\bm{R}_0}}- \frac{\bm{f}(s)}{\abs{\bars}} \bigg) d\bars - \bm{f}(s)\log(\epsilon^2),
\end{equation}
a direct application of Lemma \ref{center_est_lem} yields
\begin{align*}
\abs{\mc{S}_1 - \bm{J}_{\mc{S},1}} &\le \epsilon \abs{\log\epsilon}c_\kappa \norm{\bm{f}}_{C^1(\T)}.
\end{align*}

Furthermore, letting 
\begin{equation}\label{cent_int2}
\bm{J}_{\mc{S},2} = \int_{-1/2}^{1/2} \bigg(\frac{ \bm{R}_0\bm{R}_0^{\rm T}}{\abs{\bm{R}_0}^3}\bm{f}(s+\bars) - \frac{\be_t(s)\be_t(s)^{\rm T}}{\abs{\bars}}\bm{f}(s) \bigg) d\bars - \big[\log(\epsilon^2) + 2\big]\be_t(\be_t\cdot\bm{f}(s))
\end{equation}
and using \eqref{Reps} and \eqref{CQ} along with Lemma \ref{center_est_lem}, we have
 \begin{align*}
\bigg|\mc{S}_2- \bm{J}_{\mc{S},2} - \int_{-1/2}^{1/2} \frac{\epsilon^2\be_\rho\be_{\rho}^{\rm T}}{\abs{\bm{R}}^3} \bm{f}(s+\bars)d\bars \bigg|  &\le c_\kappa\epsilon\abs{\log\epsilon}\norm{\bm{f}}_{C^1(\T)} + c_\kappa\int_{-1/2}^{1/2} \frac{\epsilon\bars^2+\epsilon \abs{\bars}}{\abs{\bm{R}}^3}\abs{\bm{f}}d\bars \\
&\qquad + c_\kappa\norm{\bm{f}}_{C(\T)}\int_{-1/2}^{1/2} \big(|\bars|^3 +\bars^4 \big)\bigg|\frac{1}{\abs{\bm{R}}^3} - \frac{1}{\abs{\bm{R}_0}^3} \bigg| d\bars  \\
&\le c_\kappa\epsilon\abs{\log\epsilon}\norm{\bm{f}}_{C^1(\T)} \\
&\qquad + c_\kappa\norm{\bm{f}}_{C(\T)}\int_{-1/2}^{1/2}\frac{\epsilon^2+\epsilon\bars^2+ \epsilon^2\abs{\bars}+\epsilon|\bars|^3}{\bars^2+\epsilon^2} d\bars \\
&\le c_\kappa\epsilon\abs{\log\epsilon}\norm{\bm{f}}_{C^1(\T)},
\end{align*}
where we have used Lemmas \ref{Rintest0} and \ref{Rintest1} in the second inequality, and \eqref{IR_def}, \eqref{Rlb}, and \eqref{non_intersecting} in the third inequality. By Lemma \ref{Rintest2}, we then have
\begin{align*}
\big|\mc{S}_2- \bm{J}_{\mc{S},2} - 2\be_\rho(\be_{\rho}\cdot\bm{f}(s)) \big|  &\le c_\kappa\epsilon\abs{\log\epsilon}\norm{\bm{f}}_{C^1(\T)}.
\end{align*}

Together, the Stokeslet terms satisfy 
\begin{equation}\label{stokeslet_terms}
\abs{\int_{-1/2}^{1/2}\mc{S}(\bm{R})\bm{f}(s+\bars)d\bars - \bm{J}_{\mc{S},1}-\bm{J}_{\mc{S},2}  - 2\be_\rho(\be_{\rho}\cdot\bm{f}(s))} \le c_\kappa\epsilon\abs{\log\epsilon}\norm{\bm{f}}_{C^1(\T)}.
\end{equation}


Similarly, we may write the doublet term of \eqref{stokes_SB} as
\begin{equation}
\begin{aligned}
\int_{-1/2}^{1/2}&\mc{D}(\bm{R})\bm{f}(s+\bars)d\bars = \mc{D}_1 + \mc{D}_2; \\
\mc{D}_1&:= \int_{-1/2}^{1/2} \frac{\bm{f}(s+\bars)}{\abs{\bm{R}}^3} \ts d\bars, \quad \mc{D}_2 := -3\int_{-1/2}^{1/2} \frac{\bm{R}\bm{R}^{\rm T}}{\abs{\bm{R}}^5} \bm{f}(s+\bars) \ts d\bars.
\end{aligned}
\end{equation}

Using Lemma \ref{Rintest2}, we have
\begin{align*}
\abs{\mc{D}_1- \epsilon^{-2}2 \bm{f}(s)} \le c_\kappa \epsilon^{-1}\norm{\bm{f}}_{C^1(\T)}.
\end{align*}

Furthermore, using \eqref{Reps} along with Lemma \ref{Rintest0}, the second term $\D_2$ satisfies 
\begin{align*}
\abs{\mc{D}_2+ \epsilon^{-2}(2\be_t\be_t^{\rm T}+4\be_\rho\be_\rho^{\rm T})\bm{f}(s)} &\le \abs{3\int_{-1/2}^{1/2} \frac{\bars^2\be_t\be_t^{\rm T}}{\abs{\bm{R}}^5} \bm{f}(s+\bars)d\bars + \epsilon^{-2}2\be_t(\be_t\cdot\bm{f}(s))} \\
&\quad +\abs{3\epsilon^2\int_{-1/2}^{1/2} \frac{\be_\rho\be_\rho^{\rm T}}{\abs{\bm{R}}^5} \bm{f}(s+\bars)d\bars + \epsilon^{-2}4\be_\rho(\be_\rho\cdot\bm{f}(s))} + c_\kappa \epsilon^{-1}\norm{\bm{f}}_{C(\T)} \\
&\le  c_\kappa \epsilon^{-1}\norm{\bm{f}}_{C^1(\T)}, 
\end{align*}
where we have used Lemma \ref{Rintest2} in the second inequality. Letting
\begin{equation}\label{cent_int3}
\bm{J}_{\mc{D},1} = ({\bf I}-\be_t\be_t^{\rm T})\bm{f}(s),
\end{equation}
the doublet terms together yield
\begin{equation}\label{doublet_terms}
\abs{\frac{\epsilon^2}{2}\int_{-1/2}^{1/2}\mc{D}(\bm{R})\bm{f}(s+\bars)d\bars -\bm{J}_{\mc{D},1}  + 2\be_\rho(\be_\rho\cdot\bm{f}(s))} \le  c_\kappa \epsilon\norm{\bm{f}}_{C^1(\T)}. 
\end{equation}

Combining \eqref{stokeslet_terms} and \eqref{doublet_terms}, we obtain the following estimate for $\bu^{\SB}$ along $\Gamma_\epsilon$:
\begin{equation}\label{uSB_center0}
\abs{\bu^{\SB}(s,\theta) - \bm{J}_{\mc{S},1}- \bm{J}_{\mc{S},2}- \bm{J}_{\mc{D},1}} \le c_\kappa\epsilon\abs{\log\epsilon}\norm{\bm{f}}_{C^1(\T)}.
\end{equation}

Now, recalling the periodic expression \eqref{SBT_asymp} for $\bu^{\SB}_C(s)$ as well as the identity
\[ \int_{-1/2}^{1/2}\bigg(\frac{1}{\abs{\sin(\pi\bars)/\pi}}-\frac{1}{\abs{\bars}} \bigg)d\bars = 2\log(4/\pi),\]
we notice that
\begin{align*}
\bu^{\SB}_C(s)& - \bm{J}_{\mc{S},1}- \bm{J}_{\mc{S},2}- \bm{J}_{\mc{D},1} \\
&= -({\bf I}+\be_t\be_t^{\rm T})\bm{f}(s)\int_{-1/2}^{1/2}\bigg(\frac{1}{\abs{\sin(\pi\bars)/\pi}} -\frac{1}{\abs{\bars}} \bigg)d\bars + 2\log(4/\pi) ({\bf I}+\be_t\be_t^{\rm T})\bm{f}(s) =0,
\end{align*}
and therefore \eqref{uSB_center0} implies Proposition \ref{centerline_prop}.

\end{proof}


\subsection{Slender body force residual}\label{SBforce_res}
It remains to calculate the slender body approximation to the total force at each cross section $s\in \T$, given by
\begin{equation}\label{fSB_expr}
{\bm f}^{\SB}(s)= \int_0^{2\pi}\bigg(-p^{\SB}{\bf I}+2\E(\bu^{\SB}) \bigg) {\bm n} \ts \mc{J}_\epsilon(s,\theta)\ts d\theta.
\end{equation}

The estimation of the slender body force expression \eqref{fSB_expr} will proceed similarly to the calculations for the velocity residual in the previous section, relying on Lemmas \ref{Rintest0} - \ref{Rintest2} to bound the resulting integral terms. Because of the structure of \eqref{fSB_expr}, we will also be able to use a stronger bound (Lemma \ref{theta_int}) relying on $\theta$ integration to remove the $\log\epsilon$ dependence in the force residual estimate.  \\

From \eqref{fSB_expr}, calculating the slender body force requires two main components: the force due to the slender body pressure \eqref{SB_pressure} and the force due to the surface strain rate $\E(\bu^{\SB})\bm{n}\big|_{\Gamma_\epsilon}$. Recalling that ${\bm n}=-\be_\rho$, we can express the surface strain rate with respect to the moving frame basis $\be_t(s)$, $\be_\rho(s,\theta)$, $\be_\theta(s,\theta)$ as
\begin{equation}\label{SB_strain}
\begin{aligned}
2\E(\bu){\bm n} &= -\frac{\p\bu}{\p\rho} -\left(\frac{\p \bu}{\p \rho}\cdot\be_{\rho}\right)\be_{\rho} - \frac{1}{\epsilon}\left(\frac{\p \bu}{\p \theta}\cdot\be_{\rho} \right)\be_{\theta}- \frac{1}{1-\epsilon\wh\kappa}\left(\left(\frac{\p \bu}{\p s}-\kappa_3\frac{\p \bu}{\p\theta}\right)\cdot \be_{\rho}\right) \be_{t}.
\end{aligned}
\end{equation}

\begin{remark} 
Before we estimate $\bm{f}^{\SB}$, we consider the (purely heuristic) slender body approximation about an infinitely long fiber with a straight centerline and constant total force $\bm{f}^c$ over each cross section. In this case, although the slender body velocity approximation diverges logarithmically at infinity, the velocity does exactly satisfy the $\theta$-independence condition on the the slender body surface due to the doublet correction with coefficient $\frac{\epsilon^2}{2}$. This is essentially the scenario for which slender body theory is designed to work.  \\
 
Indeed, in the straight centerline/constant force scenario, the slender body force expression \eqref{fSB_expr} also exactly recovers the prescribed force $\bm{f}^c$. When $\kappa\equiv0$, we have $\bm{R}= (s-t)\be_t+\epsilon\be_{\rho}(\theta)$, where the basis vectors no longer depend on the cross section $s$. We can then directly integrate the slender body approximation \eqref{SBT2} in $t$ to obtain: 
\begin{equation}\label{straight_center}
-\frac{\p \bu^{\SB}_{\text{str}}}{\p\rho} = \frac{1}{\epsilon2\pi} \bigg[\bm{f}^c - \be_{\rho}(\be_{\rho}\cdot\bm{f}^c)\bigg], \quad \bigg(\frac{\p \bu^{\SB}_{\text{str}}}{\p\rho}\cdot\be_{\rho}\bigg)\be_{\rho}  = 0, \quad
\frac{1}{\epsilon}\frac{\p \bu^{\SB}_{\text{str}}}{\p \theta} =0, \quad
\frac{\p \bu^{\SB}_{\text{str}}}{\p s} = 0.
\end{equation}
Additionally, the slender body pressure contribution to the total force is given by
\begin{equation}\label{pressure_exact}
\begin{aligned}
p^{\SB}_{\text{str}}(s,\theta) &= \frac{1}{ 2\pi \epsilon} \be_{\rho}\cdot \bm{f}^c.
\end{aligned}
\end{equation}

Thus the slender body approximation to the constant force $\bm{f}^c$ prescribed along an infinite straight cylinder is given by
\begin{equation}
\begin{aligned}
{\bm f}^{\SB}_{\text{str}} &= \int_0^{2\pi} \bigg[-p^{\SB}_{\text{str}}{\bm n} + 2\E(\bu^{\SB}_{\text{str}}){\bm n} \bigg] \epsilon \ts d\theta \\
&= \int_0^{2\pi} \bigg[\frac{1}{2\pi} (\be_{\rho}\cdot\bm{f}^c)\be_{\rho} + \frac{1}{2\pi}\big(\bm{f}^c - \be_{\rho}(\be_{\rho}\cdot\bm{f}^c)\big) \bigg] \ts d\theta \\
&= \int_0^{2\pi} \frac{1}{2\pi}\bm{f}^c \ts d\theta = \bm{f}^c,
\end{aligned}
\end{equation}
so we exactly recover the force $\bm{f}^c$ at each cross section along the fiber. \\

Again, the straight centerline/constant force calculations are purely heuristic, but serve to show that the error in the slender body approximation to the total force, as well as the $\theta$-dependence in the slender body surface velocity, will arise due to the curvature of the fiber centerline, the finite fiber length, and variations in the prescribed force along the centerline. \\
\end{remark}


Given a curved centerline and non-constant prescribed force ${\bm f}(s)$, we compute the slender body approximation to the force, ${\bm f}^{\SB}(s)$ using essentially the same perturbative argument as in the velocity estimation, where we relied on the straight centerline integrand to derive integral bounds for the curved centerline. \\

Although Lemmas \ref{Rintest0}, \ref{Rintest1}, and \ref{Rintest2} are actually enough to obtain an $O(\epsilon\abs{\log\epsilon})$ bound on the residual $\bm{f}^{\SB}-\bm{f}$, it turns out that we can use the $\theta$-integration in the slender body force expression \eqref{fSB_expr} to obtain a slightly stronger $O(\epsilon)$ bound. In particular, for $m=n+2$, we can improve upon the $\abs{\log\epsilon}$ bound guaranteed by Lemma \ref{Rintest1} by relying on cancellation upon integration in $\theta$, rather than symmetry cancellation due to $m$ being odd. For $m=n+1$, we gain an additional $\epsilon$ factor over the Lemma \ref{Rintest0} bound.

\begin{lemma}\label{theta_int}
Let $\bm{R}$ be as in \eqref{Reps}. Suppose $m$ is a non-negative integer and $n= m+1$ or $m+2$. Furthermore, assume $g\in C(\T)$. For $k\in \Z$, $k\neq 0$, $\theta_0\in \R$ and $\epsilon>0$ sufficiently small, we have
\begin{equation}\label{theta_int_eq}
\abs{\int_0^{2\pi}\int_{-1/2}^{1/2}\frac{\bars^m g(\bars)}{\abs{\bm{R}}^n}\cos(k(\theta+\theta_0)) \ts d\bars d\theta}
\le \begin{cases}
c_\kappa \epsilon\abs{\log\epsilon}\norm{g}_{C(\T)}, & n=m+1, \\
c_\kappa \norm{g}_{C(\T)}, & n=m+2,
\end{cases} 
\end{equation}
where the constants $c_\kappa$ depend only on $c_\Gamma, \kappa_{\max}$, and $n$. 
\end{lemma}

\begin{remark}\label{theta_int_rmk}
Note that by plugging in the correct values of $k$ and $\theta_0$, Lemma \ref{theta_int} also covers integrands of the form $\bars^m g(\bars)/\abs{\bm{R}}^n$ integrated against $\sin\theta$ or agains odd triples $\sin^j\theta\cos^k\theta$, $k+j=3$, $k,j\ge0$, via the trigonometric identities  
\begin{align*}
\cos^3\theta &= \frac{1}{4}(3 \cos\theta + \cos(3\theta)), \quad \sin\theta\cos^2\theta = \frac{1}{4}(\sin\theta + \sin(3 \theta)), \\
\sin^3\theta &= \frac{1}{4}(3 \sin\theta - \sin(3 \theta)), \quad \sin^2\theta\cos\theta = \frac{1}{4}(\cos\theta - \cos(3\theta)).
\end{align*}

Note in particular that Lemma \ref{theta_int} applies to integrands of the form $\frac{\bars^m}{\abs{\bm{R}}^n}\be_\rho(\bm{A}(\bars)\cdot\be_\rho)(\bm{B}(\bars)\cdot\be_\rho)$ and $\frac{\bars^m}{\abs{\bm{R}}^n}\be_\theta(\bm{A}\cdot\be_\rho)(\bm{B}\cdot\be_\theta)$, where $\bm{A}=(a_1,a_2,a_3)^{\rm T}$ and $\bm{B}=(b_1,b_2,b_3)^{\rm T}$ are vector-valued functions that do not depend on $\theta$. We can expand these quantities as
\begin{align*}
\be_\rho(\bm{A}\cdot\be_\rho)(\bm{B}\cdot\be_\rho) &= \big(a_2b_2\cos^3\theta + (a_2b_3+b_2a_3)\cos^2\theta\sin\theta+b_3a_3\sin^2\theta\cos\theta \big)\be_{n_1}(s) \\
&\qquad+\big(a_3b_3\sin^3\theta + (a_2b_3+b_2a_3)\sin^2\theta\cos\theta+b_2a_2\cos^2\theta\sin\theta \big)\be_{n_2}(s), \\
\be_\theta(\bm{A}\cdot\be_\rho)(\bm{B}\cdot\be_\theta) &= \big(a_3b_2\sin^3\theta + (a_2b_2-b_3a_3)\sin^2\theta\cos\theta-b_3a_2\cos^2\theta\sin\theta \big)\be_{n_1}(s) \\
&\qquad+\big(a_2b_3\cos^3\theta + (a_3b_3-b_2a_2)\sin\theta\cos^2\theta-b_2a_3\cos\theta\sin^2\theta \big)\be_{n_2}(s),
\end{align*}
and, using the above trigonometric identities, apply Lemma \ref{theta_int} to each term.
\end{remark}


\begin{proof}[Proof of Lemma \ref{theta_int}]
First note that, for $\bm{R}_0(s,\bars)=\X(s) - \X(s+\bars)$, we may write
\begin{equation}
\begin{aligned}
I &=\int_0^{2\pi}\int_{-1/2}^{1/2}\frac{\bars^m g(\bars)}{\abs{\bm{R}}^n}\cos(k(\theta+\theta_0))d\bars \ts d\theta \\
&= \int_0^{2\pi}\int_{-1/2}^{1/2}\bigg(\frac{1}{\abs{\bm{R}}^n}-\frac{1}{(\abs{\bm{R}_0}^2+\epsilon^2)^{n/2}} \bigg)\bars^m g(\bars)\cos(k(\theta+\theta_0))d\bars \ts d\theta,
\end{aligned}
\end{equation}
where we have used that the second term integrates to zero in $\theta$. Then
\begin{align*}
\abs{I} &\le \int_0^{2\pi}\int_{-1/2}^{1/2}\frac{\norm{g}_{C(\T)}\abs{\bars}^m\abs{\abs{\bm{R}}^2-(\abs{\bm{R}_0}^2+\epsilon^2)}}{\abs{\bm{R}}(\abs{\bm{R}_0}^2+\epsilon^2)^{1/2}(\abs{\bm{R}}+(\abs{\bm{R}_0}^2+\epsilon^2)^{1/2})} \sum_{j=0}^{n-1}\frac{1}{\abs{\bm{R}}^j(\abs{\bm{R}_0}^2+\epsilon^2)^{(n-1-j)/2}} \ts d\bars d\theta.
\end{align*}

Now, by \eqref{CQ} and \eqref{Reps}, we have
\[ \abs{\abs{\bm{R}}^2-\abs{\bm{R}_0}^2-\epsilon^2}=2\epsilon\bars^2\abs{\be_\rho\cdot\bm{Q}}\le \epsilon \kappa_{\max}\bars^2,\]
while by Lemma \ref{absRests} and \eqref{non_intersecting} we have
\[ \abs{\bm{R}}\ge c_\kappa\sqrt{\bars^2+\epsilon^2}, \quad \abs{\bm{R}_0}\ge c_\Gamma\abs{\bars}.\]

Thus
\begin{align*}
\abs{I}&\le c_\kappa \epsilon \norm{g}_{C(\T)}\int_0^{2\pi}\int_{-1/2}^{1/2}\frac{\abs{\bars}^{m+2}}{(\bars^2+\epsilon^2)^{(n+2)/2}} \ts d\bars d\theta \le \begin{cases}
 c_\kappa \epsilon\abs{\log\epsilon}\norm{g}_{C(\T)}, & n=m+1 \\
 c_\kappa \norm{g}_{C(\T)}, & n=m+2,
 \end{cases}
\end{align*}
by Lemma \ref{defints}.

\end{proof}


We now proceed to estimate the slender body force \eqref{fSB_expr} for a fiber satisfying the geometric constraints of Section \ref{geometric_constraints} given a true force ${\bm f}(s)$ in $C^1(\T)$. Since the stress tensor $\bm{\sigma}^{\SB} = -p^{\SB}{\bf I}+ 2\E(\bu^{\SB})$ with $\E(\bu^{\SB})$ given by \eqref{SB_strain} essentially consists of five distinct terms, each of which in turn consists of derivatives of the slender body expression \eqref{SBT2}, it will be convenient to estimate each of the components of $\bm{f}^{\SB}$ separately. We label the five components of the $\bm{f}^{\SB}$ expression as follows. 
\begin{equation}\label{force_components}
\begin{aligned}
\bm{f}^{\SB} &= \bm{f}^{\SB}_p + \bm{f}^{\SB}_1 + \bm{f}^{\SB}_2 + \bm{f}^{\SB}_3 + \bm{f}^{\SB}_4; \\
\bm{f}^{\SB}_p &:= \int_0^{2\pi} -p^{\SB}\bm{n} \ts \mc{J}_\epsilon \ts d\theta \\ 
\bm{f}^{\SB}_1 &:= -\int_0^{2\pi}  \frac{\p\bu^{\SB}}{\p\rho} \ts \mc{J}_\epsilon d\theta \\ 
\bm{f}^{\SB}_2 &:=  -\int_0^{2\pi} \left(\frac{\p \bu^{\SB}}{\p \rho}\cdot\be_{\rho}\right)\be_{\rho} \ts \mc{J}_\epsilon d\theta  \\ 
\bm{f}^{\SB}_3 &:= -\int_0^{2\pi} \frac{1}{\epsilon}\left(\frac{\p \bu^{\SB}}{\p \theta}\cdot\be_{\rho} \right)\be_{\theta} \ts \mc{J}_\epsilon d\theta \\ 
\bm{f}^{\SB}_4 &:= -\int_0^{2\pi} \frac{1}{1-\epsilon\wh\kappa}\left(\left(\frac{\p \bu^{\SB}}{\p s}-\kappa_3\frac{\p \bu^{\SB}}{\p\theta}\right)\cdot \be_{\rho}\right) \be_{t}\ts \mc{J}_\epsilon d\theta
\end{aligned}
\end{equation}


We begin by estimating $\bm{f}^{\SB}_p$, the contribution of the slender body pressure $p^{\SB}$ to the total force. We show the following proposition:
\begin{proposition}\label{fSBp_est}
Let the slender body $\Sigma_\epsilon$ be as in Section \ref{geometric_constraints}. Given $\bm{f}\in C^1(\T)$, let $\bm{f}^{\SB}_p(s)$ be the pressure component of the slender body force, defined in \eqref{force_components}. Then $\bm{f}^{\SB}_p$ satisfies
 \begin{equation}
 \abs{\bm{f}^{\SB}_p(s)- \frac{1}{2} \big((\bm{f}(s)\cdot\be_{n_1}(s))\be_{n_1}(s) + (\bm{f}(s)\cdot\be_{n_2}(s))\be_{n_2}(s) \big)} \le \epsilon c_\kappa \norm{\bm{f}}_{C^1(\T)},
 \end{equation}
 where the constant $c_\kappa$ depends only on $c_{\Gamma}$ and $\kappa_{\max}$. 
\end{proposition}

\begin{proof}
 As in the velocity residual computation, we will view $\bm{R}=\bm{R}_0+\epsilon\be_\rho(s,\theta)$ as a function of $\theta$, $s$, and $\bars=-(s-t)$, rather than as a function of $\theta$, $s$, and $t$. Then, using the expression \eqref{SB_pressure} for the pressure, along with \eqref{Jeps_def} and \eqref{Reps}, we calculate
 \begin{equation}\label{fSBp}
\begin{aligned}
\bm{f}^{\SB}_p(s) &= \frac{1}{4\pi}\big(\bm{F}_1 +\bm{F}_2 +\bm{F}_3\big); \\
\bm{F}_1 &= \int_0^{2\pi}\int_{-1/2}^{1/2} \frac{\epsilon\be_{\rho}\cdot \bm{f}(s+\bars)}{\abs{\bm{R}}^3} \be_\rho \ts \epsilon \ts d\bars d\theta \\
\bm{F}_2 &= \int_0^{2\pi}\int_{-1/2}^{1/2} \frac{-\bars\be_t\cdot \bm{f}(s+\bars) + \bars^2\bm{Q}\cdot\bm{f}(s+\bars)}{\abs{\bm{R}}^3} \be_\rho \ts \epsilon\ts d\bars d\theta \\
\bm{F}_3 &= -\int_0^{2\pi}\int_{-1/2}^{1/2} \frac{\bm{R}\cdot \bm{f}(s+\bars)}{\abs{\bm{R}}^3} \be_\rho \ts \epsilon^2\wh\kappa \ts d\bars d\theta \\
\end{aligned}
\end{equation}

First note that, using Lemma \ref{Rintest0}, and recalling that $\abs{\wh\kappa}\le 2\kappa_{\max}$, we have that $\bm{F}_3$ satisfies 
\begin{align*}
\abs{\bm{F}_3} &\le 2\pi \norm{\bm{f}}_{C^(\T)} \int_{-1/2}^{1/2} \frac{1}{\abs{\bm{R}}^2} \epsilon^2\abs{\wh\kappa} \ts d\bars \le \epsilon c_\kappa \norm{\bm{f}}_{C(\T)}.
\end{align*}

Next we estimate $\bm{F}_2$. Recalling that $\be_{\rho}(s,\theta)= \cos\theta\be_{n_1}(s)+\sin\theta \be_{n_2}(s)$ while $\bm{f}(s+\bars)$, $\be_t(s)$, and $\bm{Q}(s,\bars)$ are all independent of $\theta$, we can use Lemma \ref{theta_int} to show 
\begin{align*}
\abs{\bm{F}_2} &\le \epsilon c_\kappa \norm{\bm{f}}_{C(\T)}.
\end{align*}

Finally, using Lemma \ref{Rintest2} with $m=0$ and $n=3$, we have that $\bm{F}_1$ satisfies
\begin{equation}\label{hf_def}
\begin{aligned}
\abs{\bm{F}_1 - 2\bm{h}_f(s)} &\le \epsilon c_\kappa\norm{\bm{f}}_{C^1(\T)}; \\
\bm{h}_f(s) &:= \int_0^{2\pi}\be_{\rho}(s,\theta)(\be_{\rho}(s,\theta)\cdot\bm{f}(s)) \ts d\theta \\
&= \pi \big((\bm{f}(s)\cdot\be_{n_1}(s))\be_{n_1}(s) + (\bm{f}(s)\cdot\be_{n_2}(s))\be_{n_2}(s) \big).
\end{aligned}
\end{equation}
 
 Combining these estimates, we obtain 
 \begin{equation}\label{fSBp_est0}
 \abs{\bm{f}^{\SB}_p(s)- \frac{1}{2\pi}\bm{h}_f(s)} \le \frac{1}{4\pi}\big(\abs{\bm{F}_1- 2\bm{h}_f(s)} +\abs{\bm{F}_2} +\abs{\bm{F}_3}\big) \le \epsilon c_\kappa\norm{\bm{f}}_{C^1(\T)}.
 \end{equation}
Recalling the definition of $\bm{h}_f(s)$ \eqref{hf_def}, we obtain Proposition \ref{fSBp_est}.
 \end{proof}

We now proceed to estimate $\bm{f}^{\SB}_1(s)$, the next term in the expression \eqref{force_components} for $\bm{f}^{\SB}$. In particular, we show the following:
\begin{proposition}\label{fSB1_est}
Let $\bm{f}^{\SB}_1(s)$ be as defined in \eqref{force_components}. Then $\bm{f}^{\SB}_1$ satisfies 
\begin{equation}
\abs{\bm{f}^{\SB}_1 - \frac{1}{2}\bigg(\bm{f}(s) + (\bm{f}\cdot\be_t(s))\be_t(s)\bigg) } \le \epsilon c_\kappa\norm{\bm{f}}_{C^1(\T)}
\end{equation}
where the constant $c_\kappa$ depends only on $c_{\Gamma}$ and $\kappa_{\max}$.
\end{proposition}

\begin{proof}
Using the expression \eqref{force_components} for $\bm{f}^{\SB}_1(s)$ and recalling the slender body approximation \eqref{SBT2}, we consider $\bm{f}^{\SB}_1(s)$ as the sum of a Stokeslet and a doublet term. Again considering $\bm{R}$ as a function of $\theta$, $s$, and $\bars$, we can write
\begin{equation}\label{force_1}
\begin{aligned}
\bm{f}^{\SB}_1 &= \frac{1}{8\pi}\bigg(\bm{F}_{\mc{S},1} +\frac{\epsilon^2}{2}\bm{F}_{\mc{D},1}\bigg); \\
\bm{F}_{\mc{S},1} &:= -\int_0^{2\pi}\int_{-1/2}^{1/2} \frac{\p}{\p\rho}\mc{S}(\bm{R})\bm{f}(s+\bars) \ts d\bars \ts \epsilon(1-\epsilon\wh\kappa)d\theta \\
\bm{F}_{\mc{D},1}&:=  -\int_0^{2\pi}\int_{-1/2}^{1/2} \frac{\p}{\p\rho}\mc{D}(\bm{R})\bm{f}(s+\bars) \ts d\bars \ts \epsilon(1-\epsilon\wh\kappa) d\theta.
\end{aligned}
\end{equation}

We begin by estimating $\bm{F}_{\mc{S},1}$. Recalling the notation $\bm{R}_0(s,\bars):= \X(s)-\X(s+\bars)$, we have 
\begin{equation}\label{FS1}
\begin{aligned}
\bm{F}_{\mc{S},1} &= \bm{F}_{\mc{S},11} + \bm{F}_{\mc{S},12} + \bm{F}_{\mc{S},13}+ \bm{F}_{\mc{S},14}; \\
\bm{F}_{\mc{S},11}&= \int_0^{2\pi}\int_{-1/2}^{1/2} \bigg[\frac{\epsilon{\bm f}}{|\bm{R}|^3} +\frac{3\epsilon\bm{R}(\bm{R}\cdot\bm{f})}{|\bm{R}|^5} \bigg]  d\bars \ts \epsilon \ts d\theta\\
\bm{F}_{\mc{S},12}&= \int_0^{2\pi}\int_{-1/2}^{1/2} \bigg[\frac{\bm{R}_0\cdot\be_{\rho}}{|\bm{R}|^3}{\bm f} +\frac{3\bm{R}(\bm{R}\cdot\bm{f})(\bm{R}_0\cdot\be_{\rho})}{|\bm{R}|^5} \bigg]  d\bars \ts \epsilon \ts d\theta\\
\bm{F}_{\mc{S},13}&= -\int_0^{2\pi}\int_{-1/2}^{1/2} \frac{\be_{\rho}(\bm{R}\cdot\bm{f})+\bm{R}(\be_{\rho}\cdot\bm{f}) }{|\bm{R}|^3} d\bars \ts \epsilon \ts d\theta\\
\bm{F}_{\mc{S},14}&= -\int_0^{2\pi}\int_{-1/2}^{1/2} \bigg[\frac{\bm{R}_0\cdot\be_{\rho}+\epsilon}{|\bm{R}|^3}{\bm f} - \frac{\be_{\rho}(\bm{R}\cdot\bm{f})+\bm{R}(\be_{\rho}\cdot\bm{f})}{|\bm{R}|^3} \\
&\hspace{5cm} +\frac{3\bm{R}(\bm{R}\cdot\bm{f})(\bm{R}_0\cdot\be_{\rho}+\epsilon)}{|\bm{R}|^5}\bigg] \ts d\bars \ts \epsilon^2\wh\kappa \ts d\theta.
\end{aligned}
\end{equation}

We estimate each of these terms in turn, relying on Lemmas \ref{Rintest0}, \ref{Rintest1}, \ref{Rintest2}, and \ref{theta_int} accordingly, as we did in the proof of Proposition \ref{fSBp_est}. \\

Using Lemma \ref{Rintest0}, we have
\begin{equation}\label{FS14_est}
\abs{\bm{F}_{\mc{S},14}} \le \epsilon c_\kappa\norm{\bm{f}}_{C(\T)},
\end{equation}
while by Lemmas \ref{theta_int} and \ref{Rintest2} we can show
\begin{equation}\label{FS13_est}
 \abs{\bm{F}_{\mc{S},13} + 4\bm{h}_f(s)} \le \epsilon c_\kappa \norm{\bm{f}}_{C^1(\T)},
 \end{equation}
 where $\bm{h}_f(s)$ is as in \eqref{hf_def}. Similarly, using \eqref{Rlb} along with Lemmas \ref{Rintest0} and \ref{theta_int}, we have
 \begin{equation}\label{FS12_est}
  \abs{\bm{F}_{\mc{S},12}} \le \epsilon c_\kappa \norm{\bm{f}}_{C(\T)}.
 \end{equation}
Finally, by Lemmas \ref{Rintest0}, \ref{Rintest1}, and \ref{Rintest2}, we obtain
 \begin{equation}\label{FS11_est}
 \begin{aligned}
\abs{\bm{F}_{\mc{S},11}- 2\bm{h}_a(s)-4\bm{h}_f(s)} &\le \epsilon c_\kappa\norm{\bm{f}}_{C^1(\T)}; \\
\bm{h}_a(s) &:= 2\pi \big(\bm{f}(s) + \be_t(s)(\be_t(s)\cdot\bm{f}(s))\big),
\end{aligned}
\end{equation}
where $\bm{h}_f(s)$ is again as in \eqref{hf_def}. \\

Combining the estimates \eqref{FS14_est}, \eqref{FS13_est}, \eqref{FS12_est}, and \eqref{FS11_est}, we obtain
\begin{equation}\label{FS1_est}
\abs{\bm{F}_{\mc{S},1}- 2\bm{h}_a(s)} \le \epsilon c_\kappa \norm{\bm{f}}_{C^1(\T)}.
\end{equation}

Now we estimate $\bm{F}_{\mc{D},1}$. Following the same outline as in the $\bm{F}_{\mc{S},1}$ estimate, we write 
\begin{equation}\label{FD1}
\begin{aligned}
 \bm{F}_{\mc{D},1} &= 3(\bm{F}_{\mc{D},11} + \bm{F}_{\mc{D},12} + \bm{F}_{\mc{D},13}+ \bm{F}_{\mc{D},14}); \\
\bm{F}_{\mc{D},11}&= \int_0^{2\pi}\int_{-1/2}^{1/2} \bigg[\frac{\epsilon{\bm f}}{|\bm{R}|^5} -\frac{5\epsilon\bm{R}(\bm{R}\cdot\bm{f})}{|\bm{R}|^7} \bigg] d\bars \ts \epsilon \ts d\theta\\
\bm{F}_{\mc{D},12}&= \int_0^{2\pi}\int_{-1/2}^{1/2} \bigg[\frac{\bm{R}_0\cdot\be_{\rho}}{|\bm{R}|^5}{\bm f} -\frac{5\bm{R}(\bm{R}\cdot\bm{f})(\bm{R}_0\cdot\be_{\rho})}{|\bm{R}|^7} \bigg] d\bars \ts \epsilon \ts d\theta\\
\bm{F}_{\mc{D},13}&= \int_0^{2\pi}\int_{-1/2}^{1/2} \frac{\be_{\rho}(\bm{R}\cdot\bm{f})+\bm{R}(\be_{\rho}\cdot\bm{f}) }{|\bm{R}|^5}  d\bars \ts \epsilon \ts d\theta\\
\bm{F}_{\mc{D},14}&= -\int_0^{2\pi}\int_{-1/2}^{1/2} \bigg[\frac{\bm{R}_0\cdot\be_{\rho}+\epsilon}{|\bm{R}|^5}{\bm f} + \frac{\be_{\rho}(\bm{R}\cdot\bm{f})+\bm{R}(\be_{\rho}\cdot\bm{f})}{|\bm{R}|^5} \\
&\hspace{5cm} -\frac{5\bm{R}(\bm{R}\cdot\bm{f})(\bm{R}_0\cdot\be_{\rho}+\epsilon)}{|\bm{R}|^7}\bigg] \ts d\bars \ts \epsilon^2\wh\kappa \ts d\theta.
\end{aligned}
\end{equation}

Now, using Lemma \ref{Rintest0}, we can show
\begin{equation}\label{FD14_est}
\abs{\bm{F}_{\mc{D},14}} \le \epsilon^{-1} c_\kappa\norm{\bm{f}}_{C(\T)}. 
\end{equation}

Furthermore, by Lemmas \ref{Rintest0}, \ref{Rintest1}, and \ref{Rintest2}, we have
\begin{equation}\label{FD13_est}
\abs{\bm{F}_{\mc{D},13} - \epsilon^{-2}\frac{8}{3} \bm{h}_f(s)} \le \epsilon^{-1}c_\kappa\norm{\bm{f}}_{C^1(\T)}.
\end{equation}
where $\bm{h}_f(s)$ is as in \eqref{hf_def}. Then, via \ref{Rintest0}, we can show
\begin{equation}\label{FD12_est}
\abs{\bm{F}_{\mc{D},12}} \le \epsilon^{-1}c_\kappa\norm{\bm{f}}_{C(\T)},
\end{equation}
while Lemmas \ref{Rintest0}, \ref{Rintest1}, and \ref{Rintest2} yield
\begin{equation}\label{FD11_est}
\abs{\bm{F}_{\mc{D},11}+\epsilon^{-2} \frac{8}{3}\bm{h}_f(s)} \le \epsilon^{-1}c_\kappa\norm{\bm{f}}_{C^1(\T)}.
\end{equation}

Combining estimates \eqref{FD14_est}, \eqref{FD13_est}, \eqref{FD12_est}, and \eqref{FD11_est}, we obtain
\begin{equation}\label{FD1_est}
\begin{aligned}
\abs{\bm{F}_{\mc{D},1}} &\le \abs{3\bm{F}_{\mc{D},11}+\epsilon^{-2} 8\bm{h}_f(s)} + \abs{3\bm{F}_{\mc{D},12}}+\abs{3\bm{F}_{\mc{D},13}-\epsilon^{-2} 8\bm{h}_f(s)} + \abs{3\bm{F}_{\mc{D},14}} \\
&\le \epsilon^{-1} c_\kappa\norm{\bm{f}}_{C^1(\T)}.
\end{aligned}
\end{equation}

Finally, using the estimates \eqref{FS1_est} and \eqref{FD1_est}, as well as the expression \eqref{force_1} for $\bm{f}_1^{\SB}$, we obtain
\begin{equation}\label{f1sb_est0}
\abs{\bm{f}^{\SB}_1 - \frac{1}{4\pi}\bm{h}_a(s)} \le \frac{1}{8\pi}\bigg(\abs{\bm{F}_{\mc{S},1}- 2\bm{h}_a(s)} +\frac{\epsilon^2}{2}\abs{\bm{F}_{\mc{D},1}}\bigg) \le \epsilon c_\kappa \norm{\bm{f}}_{C^1(\T)},
\end{equation}
from which, using the form of $\bm{h}_a(s)$ in \eqref{FS11_est}, we obtain Proposition \ref{fSB1_est}.
\end{proof}


Next we show the following bound for the component $\bm{f}^{\SB}_2(s)$ of the slender body force, given by \eqref{force_components}.
\begin{proposition}\label{fSB2_est}
Let the slender body $\Sigma_\epsilon$ be as in Section \ref{geometric_constraints}. Given $\bm{f}\in C^1(\T)$, let $\bm{f}^{\SB}_2(s)$ be defined as in \eqref{force_components}. Then $\bm{f}^{\SB}_2$ satisfies 
\begin{equation}
\abs{\bm{f}^{\SB}_2 } \le \epsilon c_\kappa\norm{\bm{f}}_{C^1(\T)},
\end{equation}
where the constant $c_\kappa$ depends only on $c_{\Gamma}$ and $\kappa_{\max}$.
\end{proposition}

\begin{proof}
Using the $\bm{f}^{\SB}_1$ computation as a guide, we again use \eqref{SBT2} to consider $\bm{f}^{\SB}_2$ as the sum of a Stokeslet and doublet term: 
\begin{equation}\label{force_2}
\begin{aligned}
\bm{f}^{\SB}_2 &= \frac{1}{8\pi} \bigg(\bm{F}_{\mc{S},2} + \frac{\epsilon^2}{2}\bm{F}_{\mc{D},2} \bigg); \\
\bm{F}_{\mc{S},2} &:= -\int_0^{2\pi}\int_{-1/2}^{1/2} \be_\rho \bigg(\frac{\p}{\p\rho}\mc{S}(\bm{R})\bm{f}(s+\bars)\bigg)\cdot\be_{\rho}  \ts d\bars \ts \epsilon(1-\epsilon\wh\kappa)d\theta \\
\bm{F}_{\mc{D},2}&:=  -\int_0^{2\pi}\int_{-1/2}^{1/2} \be_\rho \bigg(\frac{\p}{\p\rho}\mc{D}(\bm{R})\bm{f}(s+\bars)\bigg)\cdot\be_{\rho}  \ts d\bars \ts \epsilon(1-\epsilon\wh\kappa)d\theta  .
\end{aligned}
\end{equation}

As we did for $\bm{f}^{\SB}_1$, we begin by estimating the Stokeslet term $\bm{F}_{\mc{S},2}$. We write $\bm{F}_{\mc{S},2}$ as 
\begin{equation}\label{FS2}
\begin{aligned}
\bm{F}_{\mc{S},2} &= \bm{F}_{\mc{S},21} + \bm{F}_{\mc{S},22} + \bm{F}_{\mc{S},23}+ \bm{F}_{\mc{S},24}; \\
\bm{F}_{\mc{S},21} &:= \int_0^{2\pi}\int_{-1/2}^{1/2} \bigg[\frac{\epsilon}{|\bm{R}|^3}({\bm f}\cdot\be_\rho)\be_\rho +\frac{3\epsilon\be_\rho(\bm{R}\cdot\be_{\rho})(\bm{R}\cdot\bm{f})}{|\bm{R}|^5} \bigg] d\bars \ts \epsilon \ts d\theta \\
\bm{F}_{\mc{S},22} &:= \int_0^{2\pi}\int_{-1/2}^{1/2} \bigg[\frac{\bm{R}_0\cdot\be_{\rho}}{|\bm{R}|^3}({\bm f}\cdot\be_\rho)\be_\rho +\frac{3\be_\rho(\bm{R}\cdot\be_{\rho})(\bm{R}\cdot\bm{f})(\bm{R}_0\cdot\be_{\rho})}{|\bm{R}|^5}\bigg] d\bars \ts \epsilon \ts d\theta \\
\bm{F}_{\mc{S},23} &:= -\int_0^{2\pi}\int_{-1/2}^{1/2} \frac{\be_{\rho}(\bm{R}\cdot\bm{f})+\be_\rho(\bm{R}\cdot\be_{\rho})(\be_{\rho}\cdot\bm{f})}{|\bm{R}|^3} d\bars \ts \epsilon \ts d\theta \\
\bm{F}_{\mc{S},24} &:= -\int_0^{2\pi}\int_{-1/2}^{1/2} \bigg[\frac{\bm{R}_0\cdot\be_{\rho}+\epsilon}{|\bm{R}|^3}({\bm f}\cdot\be_\rho)\be_\rho - \frac{\be_{\rho}(\bm{R}\cdot\bm{f})+\be_\rho(\bm{R}\cdot\be_{\rho})(\be_{\rho}\cdot\bm{f})}{|\bm{R}|^3} \\
&\hspace{4cm}+\frac{3\be_\rho(\bm{R}\cdot\be_{\rho})(\bm{R}\cdot\bm{f})(\bm{R}_0\cdot\be_{\rho}+\epsilon)}{|\bm{R}|^5}\bigg] d\bars \ts\epsilon^2\wh\kappa \ts d\theta. 
\end{aligned}
\end{equation}

We again rely on Lemmas \ref{Rintest0}, \ref{Rintest1}, \ref{Rintest2}, and \ref{theta_int} to estimate each of the above components of $\bm{F}_{\mc{S},2}$ in the same way as in the proofs of Propositions \ref{fSBp_est} and \ref{fSB1_est}. By Lemma \ref{Rintest0}, we have
\begin{equation}\label{FS24_est}
\abs{\bm{F}_{\mc{S},24}} \le \epsilon c_\kappa \norm{\bm{f}}_{C(\T)}.
\end{equation}
Additionally, by Lemmas \ref{theta_int} and \ref{Rintest2}, we have
\begin{equation}\label{FS23_est}
 \abs{\bm{F}_{\mc{S},23} + 4\bm{h}_f(s)} \le \epsilon c_\kappa\norm{\bm{f}}_{C^1(\T)}
 \end{equation}
for $\bm{h}_f(s)$ as in \eqref{hf_def}.
From equation \eqref{Rlb} and Lemmas \ref{Rintest0} and \ref{theta_int}, we also obtain
 \begin{equation}\label{FS22_est}
 \abs{\bm{F}_{\mc{S},22}}\le \epsilon c_\kappa\norm{\bm{f}}_{C(\T)}.
 \end{equation}
Finally, using Lemmas \ref{Rintest0}, \ref{Rintest1}, and \ref{Rintest2}, we can show
\begin{equation}\label{FS21_est}
  \abs{\bm{F}_{\mc{S},21}- 6 \bm{h}_f(s)} \le \epsilon c_\kappa \norm{\bm{f}}_{C^1(\T)}.
 \end{equation}

Combining estimates \eqref{FS24_est}, \eqref{FS23_est}, \eqref{FS22_est}, and \eqref{FS21_est}, we obtain the bound
\begin{equation}\label{FS2_est}
\abs{\bm{F}_{\mc{S},2}- 2 \bm{h}_f(s)} \le \epsilon c_\kappa \norm{\bm{f}}_{C^1(\T)}.
\end{equation}


Now we estimate the doublet term of the expression \eqref{force_2} for $\bm{f}^{\SB}_2$. We have that $\bm{F}_{\mc{D},2}$ can be expressed as
\begin{equation}\label{FD2}
\begin{aligned}
\bm{F}_{\mc{D},2} &= 3(\bm{F}_{\mc{D},21} + \bm{F}_{\mc{D},22} + \bm{F}_{\mc{D},23}+ \bm{F}_{\mc{D},24}); \\
\bm{F}_{\mc{D},21} &:= \int_0^{2\pi}\int_{-1/2}^{1/2} \bigg[\frac{\epsilon({\bm f}\cdot\be_\rho)\be_\rho}{|\bm{R}|^5} - \frac{5\epsilon\be_\rho(\bm{R}\cdot\be_{\rho})(\bm{R}\cdot\bm{f})}{|\bm{R}|^7} \bigg] d\bars \ts \epsilon \ts d\theta \\
\bm{F}_{\mc{D},22} &:= \int_0^{2\pi}\int_{-1/2}^{1/2} \bigg[\frac{\bm{R}_0\cdot\be_{\rho}}{|\bm{R}|^5}({\bm f}\cdot\be_\rho)\be_\rho -\frac{5\be_\rho(\bm{R}\cdot\be_{\rho})(\bm{R}\cdot\bm{f})(\bm{R}_0\cdot\be_{\rho})}{|\bm{R}|^7}\bigg] d\bars \ts \epsilon \ts d\theta \\
\bm{F}_{\mc{D},23} &:= \int_0^{2\pi}\int_{-1/2}^{1/2} \frac{\be_{\rho}(\bm{R}\cdot\bm{f})+\be_\rho(\bm{R}\cdot\be_{\rho})(\be_{\rho}\cdot\bm{f})}{|\bm{R}|^5} d\bars \ts \epsilon \ts d\theta \\
\bm{F}_{\mc{D},24} &:= -\int_0^{2\pi}\int_{-1/2}^{1/2} \bigg[\frac{\bm{R}_0\cdot\be_{\rho}+\epsilon}{|\bm{R}|^5}({\bm f}\cdot\be_\rho)\be_\rho + \frac{\be_{\rho}(\bm{R}\cdot\bm{f})+\be_\rho(\bm{R}\cdot\be_{\rho})(\be_{\rho}\cdot\bm{f})}{|\bm{R}|^5} \\
&\hspace{4cm} -\frac{5\be_\rho(\bm{R}\cdot\be_{\rho})(\bm{R}\cdot\bm{f})(\bm{R}_0\cdot\be_{\rho}+\epsilon)}{|\bm{R}|^7}\bigg] d\bars \ts\epsilon^2\wh\kappa \ts d\theta 
\end{aligned}
\end{equation}

We estimate each of these terms following the same procedure as in the Stokeslet term estimate. In particular, using Lemma \ref{Rintest0}, we can show
\begin{equation}\label{FD24_est}
\abs{\bm{F}_{\mc{D},24}} \le \epsilon^{-1}c_\kappa \norm{\bm{f}}_{C(\T)}.
\end{equation}
Furthermore, by Lemmas \ref{Rintest0}, \ref{Rintest1}, and \ref{Rintest2}, we obtain
\begin{equation}\label{FD23_est}
\abs{\bm{F}_{\mc{D},23}- \epsilon^{-2}\frac{8}{3}\bm{h}_f(s)} \le \epsilon^{-1}c_\kappa \norm{\bm{f}}_{C^1(\T)}
\end{equation}
for $\bm{h}_f(s)$ as in \eqref{hf_def}. Next, by Lemma \ref{Rintest0}, we have
\begin{equation}\label{FD22_est}
\abs{\bm{F}_{\mc{D},22}} \le \epsilon^{-1}c_\kappa\norm{\bm{f}}_{C(\T)}.
\end{equation}
Finally, by Lemmas \ref{Rintest0}, \ref{Rintest1}, and \ref{Rintest2}, we can show
\begin{equation}\label{FD21_est}
\abs{\bm{F}_{\mc{D},21} +\epsilon^{-2}4\bm{h}_f(s)} \le \epsilon^{-1}c_\kappa \norm{\bm{f}}_{C^1(\T)}.
\end{equation}

Combining the estimates \eqref{FD24_est}, \eqref{FD23_est}, \eqref{FD22_est}, and \eqref{FD21_est}, we have that $\bm{F}_{\mc{D},2}$ satisfies
\begin{equation}\label{FD2_est}
\abs{\bm{F}_{\mc{D},2} +\epsilon^{-2}4\bm{h}_f(s)} \le \epsilon^{-1} c_\kappa \norm{\bm{f}}_{C^1(\T)}.
\end{equation}

Then, using the expression \eqref{force_2} for $\bm{f}^{\SB}_2$, along with the estimates \eqref{FS2_est} and \eqref{FD2_est}, we obtain the bound
\begin{equation}\label{f2sb_est0}
\abs{\bm{f}^{\SB}_2} \le \frac{1}{8\pi}\bigg( \abs{\bm{F}_{\mc{S},2} - 2\bm{h}_f(s)} + \frac{\epsilon^2}{2}\abs{\bm{F}_{\mc{D},2} +\epsilon^{-2}4\bm{h}_f(s)}\bigg) \le \epsilon c_\kappa \norm{\bm{f}}_{C^1(\T)}.
\end{equation}
\end{proof}


A similar bound to Proposition \ref{fSB2_est} also holds for the next force component $\bm{f}^{\SB}_3(s)$. 
\begin{proposition}\label{fSB3_est}
Let the slender body $\Sigma_\epsilon$ be as in Section \ref{geometric_constraints}. Given $\bm{f}\in C^1(\T)$, let $\bm{f}^{\SB}_3(s)$ be defined as in \eqref{force_components}. Then $\bm{f}^{\SB}_3$ satisfies the bound
\begin{equation}
\abs{\bm{f}^{\SB}_3 } \le \epsilon c_\kappa\norm{\bm{f}}_{C^1(\T)}
\end{equation}
where the constant $c_\kappa$ depends only on $c_{\Gamma}$ and $\kappa_{\max}$.
\end{proposition}

\begin{proof}
Following the same steps as in the calculations of $\bm{f}^{\SB}_1$ and $\bm{f}^{\SB}_2$, we use \eqref{SBT2} in the expression \eqref{force_components} for $\bm{f}^{\SB}_3$ to consider $\bm{f}^{\SB}_3$ as the sum 
\begin{equation}\label{force_3}
\begin{aligned}
\bm{f}^{\SB}_3 &=\frac{1}{8\pi}\bigg(\bm{F}_{\mc{S},3} +\frac{\epsilon^2}{2}\bm{F}_{\mc{D},3}\bigg); \\
\bm{F}_{\mc{S},3} &:= -\int_0^{2\pi}\int_{-1/2}^{1/2} \be_\theta \bigg(\frac{\p}{\p\theta}\mc{S}(\bm{R})\bm{f}(s+\bars)\bigg)\cdot\be_\rho \ts d\bars \ts (1-\epsilon\wh\kappa)d\theta \\
\bm{F}_{\mc{D},3}&:=  -\int_0^{2\pi}\int_{-1/2}^{1/2} \be_\theta \bigg(\frac{\p}{\p\theta}\mc{D}(\bm{R})\bm{f}(s+\bars)\bigg)\cdot\be_\rho \ts d\bars \ts (1-\epsilon\wh\kappa) d\theta.
\end{aligned}
\end{equation}

As before, we begin by estimating $\bm{F}_{\mc{S},3}$. We write
\begin{equation}\label{FS3}
\begin{aligned}
\bm{F}_{\mc{S},3} &= \bm{F}_{\mc{S},31} + \bm{F}_{\mc{S},32} + \bm{F}_{\mc{S},33} ; \\
\bm{F}_{\mc{S},31} &= \int_0^{2\pi}\int_{-1/2}^{1/2} \bigg[\frac{\bm{R}_0\cdot\be_{\theta}}{|\bm{R}|^3}({\bm f}\cdot\be_{\rho})\be_\theta +\frac{3(\bm{R}\cdot\be_\rho)(\bm{R}\cdot\bm{f})(\bm{R}_0\cdot\be_{\theta})\be_\theta}{|\bm{R}|^5} \bigg] d\bars \ts \epsilon \ts d\theta \\
\bm{F}_{\mc{S},32} &= -\int_0^{2\pi}\int_{-1/2}^{1/2} \frac{\be_\theta(\bm{R}\cdot\be_\rho)(\be_{\theta}\cdot\bm{f})}{|\bm{R}|^3} d\bars \ts \epsilon \ts d\theta \\
\bm{F}_{\mc{S},33} &= -\int_0^{2\pi}\int_{-1/2}^{1/2} \bigg[\frac{\bm{R}_0\cdot\be_{\theta}}{|\bm{R}|^3}({\bm f}\cdot\be_{\rho})\be_\theta - \frac{\be_\theta(\bm{R}\cdot\be_\rho)(\be_{\theta}\cdot\bm{f})}{|\bm{R}|^3} \\
&\hspace{5cm} +\frac{3(\bm{R}\cdot\be_\rho)(\bm{R}\cdot\bm{f})(\bm{R}_0\cdot\be_{\theta})\be_\theta}{|\bm{R}|^5} \bigg] d\bars \ts \epsilon^2\wh\kappa \ts d\theta.
\end{aligned}
\end{equation}

As in the previous estimates of $\bm{F}_{\mc{S},1}$ and $\bm{F}_{\mc{S},2}$, we first use Lemma \ref{Rintest0} to show
\begin{equation}\label{FS33_est}
\abs{\bm{F}_{\mc{S},33}}  \le \epsilon c_\kappa\norm{\bm{f}}_{C(\T)}.
\end{equation}
Next, by Lemmas \ref{theta_int} and \ref{Rintest2}, we have
\begin{equation}\label{FS32_est}
\begin{aligned}
\abs{\bm{F}_{\mc{S},32} + 2\bm{h}_b(s)} &\le \epsilon c_\kappa\norm{\bm{f}}_{C^1(\T)}; \\
\bm{h}_b(s) &:=  \pi \big((\bm{f}(s)\cdot\be_{n_1}(s))\be_{n_1}(s) + (\bm{f}(s)\cdot\be_{n_2}(s))\be_{n_2}(s) \big).
\end{aligned}
\end{equation}
Finally, using \eqref{Rlb} along with Lemmas \ref{Rintest0} and \ref{theta_int}, we have
\begin{equation}\label{FS31_est}
\abs{\bm{F}_{\mc{S},31}} \le \epsilon c_\kappa \norm{\bm{f}}_{C(\T)}.
\end{equation}

Combining the estimates \eqref{FS33_est}, \eqref{FS32_est}, and \eqref{FS31_est}, we have that $\bm{F}_{\mc{S},3}$ satisfies
\begin{equation}\label{FS3_est}
\abs{\bm{F}_{\mc{S},3}+2\bm{h}_b(s)} \le \epsilon c_\kappa \norm{\bm{f}}_{C^1(\T)},
\end{equation}
where $\bm{h}_b(s)$ is as in \eqref{FS32_est}. \\

Now we estimate the doublet term $\bm{F}_{\mc{D},3}$ in the expression \eqref{force_3} for $\bm{f}^{\SB}_3$. As we did for the Stokeslet term, we decompose $\bm{F}_{\mc{D},3}$ as
\begin{equation}\label{FD3}
\begin{aligned}
\bm{F}_{\mc{D},3} &= 3(\bm{F}_{\mc{D},31} + \bm{F}_{\mc{D},32} + \bm{F}_{\mc{D},33}) ; \\
\bm{F}_{\mc{D},31} &= \int_0^{2\pi}\int_{-1/2}^{1/2} \bigg[\frac{\bm{R}_0\cdot\be_{\theta}}{|\bm{R}|^5}({\bm f}\cdot\be_{\rho})\be_\theta -\frac{5(\bm{R}\cdot\be_\rho)(\bm{R}\cdot\bm{f})(\bm{R}_0\cdot\be_{\theta})\be_\theta}{|\bm{R}|^7} \bigg] d\bars \ts \epsilon \ts d\theta \\
\bm{F}_{\mc{D},32} &= \int_0^{2\pi}\int_{-1/2}^{1/2} \frac{\be_\theta(\bm{R}\cdot\be_\rho)(\be_{\theta}\cdot\bm{f})}{|\bm{R}|^5} d\bars \ts \epsilon \ts d\theta \\
\bm{F}_{\mc{D},33} &= -\int_0^{2\pi}\int_{-1/2}^{1/2} \bigg[\frac{\bm{R}_0\cdot\be_{\theta}}{|\bm{R}|^5}({\bm f}\cdot\be_{\rho})\be_\theta + \frac{\be_\theta(\bm{R}\cdot\be_\rho)(\be_{\theta}\cdot\bm{f})}{|\bm{R}|^5} \\
&\hspace{5cm} -\frac{5(\bm{R}\cdot\be_\rho)(\bm{R}\cdot\bm{f})(\bm{R}_0\cdot\be_{\theta})\be_\theta}{|\bm{R}|^7} \bigg] d\bars \ts \epsilon^2\wh\kappa \ts d\theta.
\end{aligned}
\end{equation}

We estimate each of these integrals using the same procedure as each of the previous force term estimates. By Lemma \ref{Rintest0}, we have
\begin{equation}\label{FD33_est}
\abs{\bm{F}_{\mc{D},33}} \le \epsilon^{-1}c_\kappa\norm{\bm{f}}_{C(\T)},
\end{equation}
while Lemmas \ref{Rintest0} and \ref{Rintest2} give
\begin{equation}\label{FD32_est}
\abs{\bm{F}_{\mc{D},32} - \epsilon^{-2}\frac{4}{3}\bm{h}_b(s)} \le \epsilon^{-1}c_\kappa \norm{\bm{f}}_{C^1(\T)}
\end{equation}
for $\bm{h}_b(s)$ as in \eqref{FS32_est}. Finally, by Lemma \ref{Rintest0}, we have
\begin{equation}\label{FD31_est}
\abs{\bm{F}_{\mc{D},31}} \le \epsilon^{-1} c_\kappa \norm{\bm{f}}_{C(\T)}.
\end{equation}

Altogether, the estimates \eqref{FD33_est}, \eqref{FD32_est}, and \eqref{FD31_est} yield
\begin{equation}\label{FD3_est}
\abs{\bm{F}_{\mc{D},3} - \epsilon^{-2}4\bm{h}_b(s)} \le \epsilon^{-1}c_\kappa \norm{\bm{f}}_{C^1(\T)},
\end{equation}
where $\bm{h}_b(s)$ is as in \eqref{FS32_est}. \\

Using the estimates \eqref{FS3_est} and \eqref{FD3_est}, together with the expression \eqref{force_3} for $\bm{f}^{\SB}_3$, we obtain the bound
\begin{equation}\label{f3sb_est0}
\abs{\bm{f}^{\SB}_3} \le \frac{1}{8\pi}\bigg( \abs{\bm{F}_{\mc{S},3} + 2\bm{h}_b(s)} + \frac{\epsilon^2}{2}\abs{\bm{F}_{\mc{D},3} -\epsilon^{-2}4\bm{h}_b(s)}\bigg) \le \epsilon c_\kappa \norm{\bm{f}}_{C^1(\T)}.
\end{equation}

\end{proof}

It remains to estimate the final term $\bm{f}^{\SB}_4(s)$ of the slender body force expression \eqref{force_components}. We show that $\bm{f}^{\SB}_4(s)$ satisfies the following proposition.

\begin{proposition}\label{fSB4_est}
Let the slender body $\Sigma_\epsilon$ be as in Section \ref{geometric_constraints}. Given $\bm{f}\in C^1(\T)$, let $\bm{f}^{\SB}_4(s)$ be as defined in \eqref{force_components}. We have that $\bm{f}^{\SB}_4$ satisfies the estimate
\begin{equation}
\abs{\bm{f}^{\SB}_4 } \le \epsilon c_\kappa \norm{\bm{f}}_{C^1(\T)},
\end{equation}
where the constant $c_\kappa$ depends only on $c_{\Gamma}$ and $\kappa_{\max}$.
\end{proposition}

\begin{proof}
This estimate follows quickly from Proposition \ref{prop:uSBstheta}. Noting that $\be_\rho(s,\theta) = \cos\theta\be_{n_1}(s)+\sin\theta\be_{n_2}(s) = -\p/\p\theta \be_\theta(s,\theta)$, we have that, using the expression for $\bm{f}^{\SB}_4(s)$ in \eqref{force_components} and integrating by parts in $\theta$, $\bm{f}^{\SB}_4(s)$ can be written 
\begin{equation}\label{force_4}
\begin{aligned}
\bm{f}^{\SB}_4 & =-\frac{1}{8\pi} \int_0^{2\pi}\int_{-1/2}^{1/2} \be_t(s) \bigg(\frac{\p\bu^{\SB}}{\p s} - \kappa_3 \frac{\p \bu^{\SB}}{\p\theta}\bigg)\cdot\be_\rho(s,\theta) \ts d\bars \ts \epsilon \ts d\theta \\
&= -\frac{1}{8\pi} \int_0^{2\pi}\int_{-1/2}^{1/2} \be_t(s) \frac{\p}{\p\theta}\bigg(\frac{\p\bu^{\SB}}{\p s} - \kappa_3 \frac{\p \bu^{\SB}}{\p\theta}\bigg)\cdot\be_\theta(s,\theta) \ts d\bars \ts \epsilon \ts d\theta.
\end{aligned}
\end{equation} 
Then, by Proposition \ref{prop:uSBstheta}, we have
\begin{align*}
\abs{\bm{f}^{\SB}_4} &\le \frac{\epsilon}{8\pi}\int_0^{2\pi}\int_{-1/2}^{1/2} \bigg|\frac{\p}{\p\theta}\bigg(\frac{\p\bu^{\SB}}{\p s} - \kappa_3 \frac{\p \bu^{\SB}}{\p\theta}\bigg)\bigg| \ts d\bars d\theta \le \epsilon c_\kappa \norm{\bm{f}}_{C^1(\T)}.
\end{align*}

\end{proof}

Finally, we sum the estimates for the five force components defined in \eqref{force_components}, resulting the in the following estimate for the total slender body force $\bm{f}^{\SB}(s)$. 
\begin{proposition}\label{fSB_est}
Let the slender body $\Sigma_\epsilon$ be as in Section \ref{geometric_constraints}. Given $\bm{f}\in C^1(\T)$, let $\bm{f}^{\SB}(s)$ be the corresponding slender body approximation, given by \eqref{fSB_expr}. Then $\bm{f}^{\SB}$ satisfies 
\begin{equation}
\abs{\bm{f}^{\SB}(s) - \bm{f}(s)} \le \epsilon c_\kappa \norm{\bm{f}}_{C^1(\T)},
\end{equation}
where the constant $c_\kappa$ depends only on $c_{\Gamma}$ and $\kappa_{\max}$.
\end{proposition}

\begin{proof}
First, we introduce some notation. Let $f_t(s) :=(\bm{f}(s)\cdot\be_t(s))$, $f_{n_1}(s) :=(\bm{f}(s)\cdot\be_{n_1}(s))$, and $f_{n_2}(s) :=(\bm{f}(s)\cdot\be_{n_2}(s))$. Using the expression \eqref{force_components} for $\bm{f}^{\SB}$, together with Propositions \ref{fSBp_est}, \ref{fSB1_est}, \ref{fSB2_est}, \ref{fSB3_est}, and \ref{fSB4_est}, we have
\begin{align*}
\abs{\bm{f}^{\SB}(s) - \bm{f}(s)} &= \abs{\bm{f}^{\SB}(s) - \frac{1}{2}\bm{f}(s) - \frac{1}{2}\big(f_t(s)\be_t(s) + f_{n_1}(s)\be_{n_1}(s) + f_{n_2}(s)\be_{n_2}(s)\big)} \\
&\le \abs{\bm{f}^{\SB}_p(s)- \frac{1}{2} \big(f_{n_1}(s)\be_{n_1}(s) + f_{n_2}(s)\be_{n_2}(s) \big)} \\
&\quad  +\abs{\bm{f}^{\SB}_1(s) - \frac{1}{2}\bigg(\bm{f}(s) + f_t(s)\be_t(s)\bigg) } +\abs{\bm{f}^{\SB}_2(s)} +\abs{\bm{f}^{\SB}_3(s)}+ \abs{\bm{f}^{\SB}_4(s)} \\
&\le \epsilon c_\kappa \norm{\bm{f}}_{C^1(\T)}.
\end{align*}

\end{proof}


\section{Error estimate}\label{error_est_section}
Using the residual calculations for the surface velocity $\bu^{\SB}\big|_{\Gamma_{\epsilon}}$ and the total surface force ${\bm f}^{\SB}$, we proceed to prove the error estimate \eqref{err_stokes_thm} in Theorem \ref{stokes_err_theorem}. \\

Let $\bu^{\text{e}}=\bu^{\SB}-\bu$, $p^{\text{e}}=p^{\SB}-p$, and $\bm{\sigma}_e= -p^{\text{e}}{\bf I}+2\E(\bu^{\text{e}})=\bm{\sigma}^{\SB}-\bm{\sigma}$, where $\bu$, $p$, and $\bm{\sigma}=-p{\bf I}+\nabla \bu +(\nabla \bu)^{\rm T}$ correspond to the true solution to \eqref{exterior_stokes}. Then the difference $\bu^{\text{e}}$ satisfies, in the weak sense,
\begin{equation}\label{err_PDE_stokes}
\begin{aligned}
-\Delta \bu^{\text{e}} + \nabla p^{\text{e}} &= 0 \\
\dive\ts \bu^{\text{e}} &= 0 \qquad \text{ in }\Omega_{\epsilon} \\
\int_0^{2\pi} (\bm{\sigma}_e {\bm n}) \ts \mathcal{J}_{\epsilon}(s,\theta)\ts d\theta &= {\bm f}^{\text{e}}(s) \quad \text{on } \Gamma_{\epsilon} \\
\bu^{\text{e}}|_{\Gamma_{\epsilon}} &= \bar \bu^{\text{e}}(s)+\bu^{\rm r}(s,\theta) \\
\bu^{\text{e}} \to 0 &\text{ as }|\bx| \to \infty
\end{aligned}
\end{equation}

where the boundary value $\bar \bu^{\text{e}}(s)=\big(\bu^{\SB}-\bu^{\rm r}\big)\big|_{\Gamma_\epsilon}(s)-\bu\big|_{\Gamma_\epsilon}(s)$ is unknown (since $\bu(s)$ is unknown) but independent of $\theta$. Note that ${\bm f}^{\text{e}}(s)={\bm f}^{\SB}-{\bm f}$ and $\bu^{\rm r}(s,\theta)= \bu^{\rm SB}(\epsilon,\theta,s) - \frac{1}{2\pi}\int_0^{2\pi} \bu^{\SB}(\epsilon,\varphi,s) \ts d\varphi$ \eqref{ur} are both completely known functions along $\Gamma_{\epsilon}$. \\


More precisely, for arbitrary $\bw\in D^{1,2}(\Omega_{\epsilon})$, the error $\bu^{\text{e}}$ satisfies
\begin{equation}\label{variational_err_eqn}
\int_{\Omega_{\epsilon}} \bigg(2 \ts \E(\bu^{\text{e}}):\E(\bw) - p^{\text{e}} \ts\dive\ts\bw\bigg) \ts d\bx = \int_{\Gamma_{\epsilon}} (\bm{\sigma}_e {\bm n})\cdot \bw \ts dS.
\end{equation}
Now, unless $\bw\in \A_{\epsilon} = \big\{ \bw\in D^{1,2}(\Omega_{\epsilon}) \ts: \ts \bw|_{\Gamma_{\epsilon}}=\bw(s) \big\}$, i.e. $\bw$ additionally satisfies the $\theta$-independence condition on the slender body surface $\Gamma_{\epsilon}$, we cannot make use of the known expression ${\bm f}^{\text{e}}(s)$ for the error in the total force. Note in particular that the function $\bu^{\text{e}}$ itself does not belong to the set $\A_{\epsilon}$. \\

However, since $(\bu^{\text{e}},p^{\text{e}})$ satisfies \eqref{variational_err_eqn}, we can exactly follow the proof of the pressure estimate \eqref{press_est} to show that the pressure error $p^{\text{e}}$ satisfies 
\begin{equation}\label{press_err_est}
\|p^{\text{e}}\|_{L^2(\Omega_{\epsilon})} \le c_P\|\E(\bu^{\text{e}})\|_{L^2(\Omega_{\epsilon})}
\end{equation}
where $c_P$ is independent of the slender body radius $\epsilon$.\\

To derive a $D^{1,2}(\Omega_{\epsilon})$ bound for $\bu^{\text{e}}$, we use \eqref{variational_err_eqn} with a very specific choice of $\bw$. In particular, we take 
\begin{equation}
\bw= \widetilde\bu^{\text{e}}:= \bu^{\text{e}}- \widetilde\bv,
\end{equation}
where $\widetilde\bv\in D^{1,2}(\Omega_{\epsilon})$ with $\widetilde\bv\big|_{\Gamma_{\epsilon}} = \bu^{\rm r}(s,\theta)$. We explicitly construct such a $\widetilde\bv$ in Section \ref{v_construct} that we can bound in terms of ${\bm f}(s)$, the true prescribed force. \\

We then have that $\widetilde\bu^{\text{e}}\big|_{\Gamma_{\epsilon}}=\bar\bu^{\text{e}}(s)$, where $\bar \bu^{\text{e}}(s)$ is unknown but independent of $\theta$, so $\widetilde\bu^{\text{e}}\in \A_{\epsilon}$. Thus, using $\widetilde\bu^{\text{e}}$ in place of $\bw$ in \eqref{variational_err_eqn}, we obtain
\begin{equation}\label{variational_err_ee}
\int_{\Omega_{\epsilon}} \bigg(2 \ts \E(\bu^{\text{e}}):\E(\widetilde\bu^{\text{e}}) - p^{\text{e}} \ts\dive\ts\widetilde\bu^{\text{e}} \bigg)\ts d\bx = \int_{\T} {\bm f}^{\text{e}}(s)\cdot \bar\bu^{\text{e}}(s) \ts ds.
\end{equation}

From \eqref{variational_err_ee} we will derive a $D^{1,2}(\Omega_{\epsilon})$ estimate for $\bu^{\text{e}}$ in terms of the prescribed force ${\bm f}(s)$. 

\subsection{Construction of {\boldmath $\widetilde v$}}\label{v_construct}
In order to use \eqref{variational_err_ee} to obtain an estimate for $\bu^{\text{e}}$ in terms of ${\bm f}(s)$, we must construct the function $\widetilde\bv\in D^{1,2}(\Omega_{\epsilon})$ with $\widetilde\bv\big|_{\Gamma_{\epsilon}} = \bu^{\rm r}(s,\theta)$. We first define 
\[ \bu^{\SB}_{\text{ext}}(\rho,\theta,s) = \begin{cases}
\bu^{\rm r}(\theta,s) & \qquad \text{if } \rho < 4\epsilon, \\
0 & \qquad \text{otherwise}.
\end{cases} \]

Since $\bu^{\rm r}$ is at least $C^1$, $\nabla \bu^{\SB}_{\text{ext}}$ is continuous within $\rho < 4\epsilon$. \\

Let $\phi(\rho)$ be a smooth cutoff function equal to 1 for $\rho < 2$ and equal to 0 for $\rho > 4$ with smooth decay between. We require this decay to satisfy 
\begin{equation}\label{phi_estimate}
\abs{\frac{\p\phi}{\p\rho}}\le c_\phi
\end{equation}
for some constant $c_{\phi}>0$. Let $\phi_\epsilon(\rho) := \phi(\rho/\epsilon)$. \\

We define
\begin{equation}\label{corr_stokes}
\widetilde\bv(\rho,\theta,s) = \phi_\epsilon(\rho) \bu_{\text{ext}}^{\SB}(\rho,\theta,s). 
\end{equation}
Note that $\widetilde\bv(\rho,\theta,s)$ is supported in the region 
\begin{equation}\label{v_support}
\mathcal{O}_{\epsilon} := \big\{ s\be_t(s)+ \rho\be_{\rho}(s,\theta)+ \theta\be_{\theta}(s,\theta) \ts : \ts s\in\T, \ts \epsilon\le \rho\le 4\epsilon, \ts 0\le\theta<2\pi \big\}
\end{equation}
with $|\mathcal{O}_{\epsilon}| = c_{\mc{O}}^2\epsilon^2$. \\

Now, obtaining a $D^{1,2}(\Omega_{\epsilon})$ estimate for $\bu^{\text{e}}$ from \eqref{variational_err_ee} will require an $L^2(\Omega_{\epsilon})$ bound for $\nabla \widetilde\bv$, so we consider
\begin{equation}\label{grad_eSB_stokes}
\nabla \widetilde\bv = \phi_\epsilon \nabla \bu_{\text{ext}}^{\SB}+ (\nabla\phi_\epsilon) (\bu_{\text{ext}}^{\SB})^{\rm T}.
\end{equation}

We have 
\begin{align*}
\phi \nabla \bu_{\text{ext}}^{\SB} &= \frac{\phi_\epsilon}{\epsilon} \frac{\p \bu^{\rm r}}{\p\theta} \be_{\theta}^{\rm T} + \frac{\phi_\epsilon}{1-\rho\wh\kappa}\bigg(\frac{\p\bu^{\rm r}}{\p s}-\kappa_3 \frac{\p\bu^{\rm r}}{\p\theta}\bigg) \be_t^{\rm T} 
\end{align*}
and
\begin{align*}
 (\nabla\phi_\epsilon) (\bu_{\text{ext}}^{\SB})^{\rm T} &= \frac{\p \phi_\epsilon}{\p \rho} \be_{\rho} (\bu^{\rm r})^{\rm T}.
\end{align*}

Finally, notice that
\begin{equation}\label{denom_bound}
\begin{aligned}
|1-\rho\wh\kappa| & \ge 1- \rho |\kappa|(\cos\theta+\sin\theta) &\ge 1- \frac{1}{2 \kappa_{\max}}|\kappa|\sqrt{2} \ge 1- \frac{\sqrt{2}}{2} \ge \frac{1}{4}.
\end{aligned}
\end{equation}

Then, using Proposition \ref{ur_and_derivs} along with \eqref{phi_estimate}, \eqref{denom_bound}, and Lemma \ref{lemmaorthonormal}, we have 
\begin{equation}\label{grad_vSB_est}
\begin{aligned}
\|\nabla \widetilde\bv\|_{L^2(\Omega_{\epsilon})} &\le \|\nabla \widetilde\bv\|_{C(\mathcal{O}_{\epsilon})} \sqrt{|\mathcal{O}_{\epsilon}|} \\
&\le \epsilon c_\kappa\bigg(\bigg\|\frac{1}{\epsilon} \frac{\p \bu^{\rm r}}{\p\theta}\bigg\|_{C(\Gamma_{\epsilon})} + \bigg\|\frac{1}{1-\rho\wh\kappa}\bigg(\frac{\p\bu^{\rm r}}{\p s}-\kappa_3 \frac{\p\bu^{\rm r}}{\p\theta}\bigg)\bigg\|_{C(\Gamma_{\epsilon})}+ \frac{1}{\epsilon}\|\bu^{\rm r}\|_{C(\Gamma_{\epsilon})}\bigg) \\
&\le  c_\kappa \epsilon\abs{\log\epsilon}\|{\bm f}\|_{C^1(\T)}.
\end{aligned}
\end{equation}

\subsection{Estimating the error}
We now use \eqref{variational_err_ee} to obtain a $D^{1,2}(\Omega_{\epsilon})$ bound for the error $\bu^{\text{e}}$.
Recalling that $\widetilde\bu^{\text{e}}=\bu^{\text{e}}-\widetilde\bv$ and thus $\dive\ts \widetilde\bu^{\text{e}}=-\dive\ts \widetilde\bv$, we rewrite \eqref{variational_err_ee} as
\begin{equation}\label{variational_err_2}
\begin{aligned}
\int_{\Omega_{\epsilon}} 2 \ts |\E(\bu^{\text{e}})|^2 \ts d\bx &= \int_{\Omega_{\epsilon}} \bigg(2 \ts \E(\bu^{\text{e}}):\E(\widetilde\bv) - p^{\text{e}} \ts\dive\ts\widetilde\bv \bigg)\ts d\bx + \int_{\T} {\bm f}^{\text{e}}(s) \cdot\bar\bu^{\text{e}}(s) \ts ds \\
&\le \bigg|\int_{\Omega_{\epsilon}} 2 \ts \E(\bu^{\text{e}}):\E(\widetilde\bv)\ts d\bx\bigg| + \bigg|\int_{\Omega_{\epsilon}} p^{\text{e}} \ts\dive\ts\widetilde\bv \ts d\bx \bigg| + \bigg|\int_{\T} {\bm f}^{\text{e}}(s) \cdot\bar\bu^{\text{e}}(s) \ts ds\bigg|.
\end{aligned}
\end{equation}

Using Cauchy Schwarz, the first term on the right hand side of \eqref{variational_err_2} satisfies
\begin{align*}
 \bigg|\int_{\Omega_{\epsilon}} 2 \ts \E(\bu^{\text{e}}):\E(\widetilde\bv) \ts d\bx \bigg| &\le 2 \|\E(\bu^{\text{e}})\|_{L^2(\Omega_{\epsilon})}\|\E(\widetilde\bv)\|_{L^2(\Omega_{\epsilon})} \le \eta \|\E(\bu^{\text{e}}) \|_{L^2(\Omega_{\epsilon})}^2 + \frac{1}{\eta}\|\E(\widetilde\bv)\|_{L^2(\Omega_{\epsilon})}^2 \\
&\le \eta \|\E(\bu^{\text{e}}) \|_{L^2(\Omega_{\epsilon})}^2 + \frac{1}{\eta}\|\nabla \widetilde\bv\|_{L^2(\Omega_{\epsilon})}^2 
\end{align*}
for any $\eta\in \R_+$. \\

By \eqref{press_err_est} and Cauchy Schwarz, the second term on the right hand side of \eqref{variational_err_2} satisfies
\begin{align*}
\bigg| \int_{\Omega_{\epsilon}} p^{\text{e}} \ts\dive\ts\widetilde\bv \ts d\bx \bigg| &\le \|p^{\text{e}}\|_{L^2(\Omega_{\epsilon})}\|\nabla \widetilde\bv\|_{L^2(\Omega_{\epsilon})} \le c_P\|\E(\bu^{\text{e}})\|_{L^2(\Omega_{\epsilon})}\|\nabla \widetilde\bv\|_{L^2(\Omega_{\epsilon})} \\
 &\le \eta \|\E(\bu^{\text{e}}) \|_{L^2(\Omega_{\epsilon})}^2 + \frac{c_P^2}{4\eta}\|\nabla \widetilde\bv\|_{L^2(\Omega_{\epsilon})}^2.
\end{align*}

Finally, the third term on the right hand side of \eqref{variational_err_2} can be estimated using the trace inequality \eqref{Trace_ineq} on the admissible set $\A_{\epsilon}$, the Korn inequality \eqref{korn_ineq}, and Cauchy Schwarz. We have 
\begin{align*}
\bigg| \int_{\T} {\bm f}^{\text{e}}(s)\cdot\bar\bu^{\text{e}}(s) \ts ds \bigg| &\le \|{\bm f}^{\text{e}}\|_{L^2(\T)}\| \bar\bu^{\text{e}}\|_{L^2(\T)} \le c_T\|\nabla \widetilde\bu^{\text{e}}\|_{L^2(\Omega_{\epsilon})}\|{\bm f}^{\text{e}}\|_{L^2(\T)}\\
&\le c_Tc_K\|\E(\widetilde\bu^{\text{e}})\|_{L^2(\Omega_{\epsilon})}\|{\bm f}^{\text{e}}\|_{L^2(\T)} \le \eta\|\E(\widetilde\bu^{\text{e}})\|_{L^2(\Omega_{\epsilon})}^2 + \frac{c_T^2c_K^2}{4\eta}\|{\bm f}^{\text{e}}\|_{L^2(\T)}^2\\
&\le \eta\|\E(\bu^{\text{e}})\|_{L^2(\Omega_{\epsilon})}^2 + \eta\|\nabla \widetilde\bv\|_{L^2(\Omega_{\epsilon})}^2 + \frac{c_T^2c_K^2}{4\eta}\|{\bm f}^{\text{e}}\|_{L^2(\T)}^2,
\end{align*}
again for any $\eta\in \R_+$.  \\

Taking $\eta=\frac{1}{3}$, we obtain the following estimate from \eqref{variational_err_2}:
\begin{equation}\label{weak_err_3}
\|\mathcal{E}(\bu^{\text{e}})\|_{L^2(\Omega_{\epsilon})}^2 \le \frac{3c_T^2c_K^2}{4}\|{\bm f}^{\text{e}}\|_{L^2(\T)}^2 + \bigg(6+ \frac{3c_P^2}{4}\bigg) \|\nabla \widetilde\bv\|_{L^2(\Omega_{\epsilon})}^2.
\end{equation}

Then using the Korn inequality \eqref{korn_ineq}, we have
\begin{equation}\label{weak_err_4}
\|\nabla\bu^{\text{e}} \|_{L^2(\Omega_{\epsilon})}^2 \le \frac{3c_T^2c_K^4}{4}\|{\bm f}^{\text{e}}\|_{L^2(\T)}^2 + c_K^2\bigg(6+ \frac{3c_P^2}{4}\bigg) \|\nabla \widetilde\bv\|_{L^2(\Omega_{\epsilon})}^2.
\end{equation}

Recall that the Korn constant $c_K$ \eqref{korn_ineq} and the pressure constant $c_P$ \eqref{press_est} are both independent of $\epsilon$, while the trace constant $c_T$ \eqref{Trace_ineq} satisfies $c_T=c_{\kappa}|\log\epsilon|^{1/2}$. Also, from \eqref{grad_vSB_est} and Proposition \ref{fSB_est}, we have
\begin{align*}
\|\nabla\widetilde\bv\|_{L^2(\Omega_{\epsilon})} &\le \epsilon|\log\epsilon| c_\kappa  \|{\bm f}\|_{C^1(\T)} \\ 
\|{\bm f}^{\text{e}}\|_{L^2(\T)} &\le \epsilon c_\kappa  \|{\bm f}\|_{C^1(\T)}.
\end{align*}

Therefore we have 
\begin{equation}\label{err_stokes}
\begin{aligned}
\|\bu^{\text{e}}\|_{D^{1,2}(\Omega_{\epsilon})} &\le \epsilon(|\log\epsilon|^{1/2}+|\log\epsilon|)c_{\kappa} \|{\bm f}\|_{C^1(\T)} \\ 
&\le \epsilon |\log\epsilon| c_{\kappa}  \|{\bm f}\|_{C^1(\T)}
\end{aligned}
\end{equation}
where the constant $c_{\kappa}$ depends only on the shape of the fiber centerline through $\kappa_{\max}$ and $c_\Gamma$. Since the pressure error $p^{\text{e}}$ satisfies \eqref{press_err_est}, we also obtain
\begin{equation}\label{err_stokes}
\|\bu^{\text{e}}\|_{D^{1,2}(\Omega_{\epsilon})}+ \|p^{\text{e}}\|_{L^2(\Omega_{\epsilon})} \le \epsilon|\log\epsilon| c_{\kappa} \|{\bm f}\|_{C^1(\T)},
\end{equation}
where again, by Lemma \ref{divv_p_lem}, $c_{\kappa}$ depends only on $\kappa_{\max}$ and $c_\Gamma$.\\

Furthermore, using the $D^{1,2}(\Omega_{\epsilon})$ bound on the error $\bu^{\text{e}}=\bu^{\SB}-\bu$ throughout the fluid domain $\Omega_{\epsilon}$. We first write
\begin{align*}
\| {\rm Tr} (\bu^{\text{e}})\|_{L^2(\Gamma_{\epsilon})} \le \|\bar\bu^{\text{e}}(s)\|_{L^2(\Gamma_{\epsilon})}+ \|\bu^{\rm r}\|_{L^2(\Gamma_{\epsilon})}.
\end{align*}

Then, using the estimate \eqref{urest} for $\bu^{\rm r}$, we have 
\begin{align*}
\|\bu^{\rm r}\|_{L^2(\Gamma_{\epsilon})} &= \bigg(\int_{\T}\int_0^{2\pi} |\bu^{\rm r}(s,\theta)|^2 \ts \mathcal{J}_{\epsilon}(s,\theta) \ts d\theta ds\bigg)^{1/2}\\
& \le \sqrt{2}\abs{\Gamma_\epsilon}^{1/2} \|\bu^{\rm r}\|_{C(\Gamma_{\epsilon})} \le \epsilon|\log\epsilon|\ts c_\kappa \abs{\Gamma_\epsilon}^{1/2} \|{\bm f}\|_{C^1(\T)},
\end{align*}
where $|\Gamma_{\epsilon}|$ denotes the fiber surface area. \\ 

Moreover, using the trace inequality \eqref{Trace_ineq} and \eqref{err_stokes}, we have 
\begin{align*}
\|\bar\bu^{\text{e}}(s)\|_{L^2(\Gamma_{\epsilon})} &\le \abs{\Gamma_\epsilon}^{1/2}\|\bar\bu^{\text{e}}\|_{L^2(\T)} \le \abs{\Gamma_\epsilon}^{1/2} \ts c_T\|\nabla \widetilde\bu^{\text{e}}\|_{L^2(\Omega_{\epsilon})} \\
&\le \abs{\Gamma_\epsilon}^{1/2} \ts c_T \big(\|\nabla \bu^{\text{e}}\|_{L^2(\Omega_{\epsilon})} + \|\nabla \widetilde\bv\|_{L^2(\Omega_{\epsilon})} \big)  \\
&= c_{\kappa} \abs{\Gamma_\epsilon}^{1/2}|\log\epsilon|^{1/2}\big(\|\bu^{\text{e}}\|_{D^{1,2}(\Omega_{\epsilon})} + \|\nabla \widetilde\bv\|_{L^2(\Omega_{\epsilon})} \big) \\
&\le \epsilon|\log\epsilon|^{3/2} \ts \abs{\Gamma_\epsilon}^{1/2}c_{\kappa} \ts \|{\bm f}\|_{C^1(\T)},
\end{align*}
where the constant $c_\kappa$ depends only on $\kappa_{\max}$ and $c_{\Gamma}$. \\

In total, scaling by $|\Gamma_{\epsilon}|^{-1/2}$, we obtain
\begin{equation}\label{trace_err_bound}
\frac{1}{|\Gamma_{\epsilon}|^{1/2}}\| {\rm Tr} (\bu^{\text{e}})\|_{L^2(\Gamma_{\epsilon})} \le \epsilon |\log\epsilon|^{3/2} \ts c_{\kappa} \ts \|{\bm f}\|_{C^1(\T)}.
\end{equation}

Using \eqref{trace_err_bound}, we may finally estimate the error in the slender body centerline velocity approximation \eqref{SBT_asymp}, allowing us to obtain the estimate \eqref{center_err_thm} in Theorem \ref{stokes_err_theorem}. We first note that, by Proposition \ref{centerline_prop}, the difference between the surface velocity approximation $\bu^{\SB}(s,\theta)$ and the centerline velocity approximation $\bu^{\SB}_C(s)$ satisfies
\begin{align*}
\bigg(\int_{\T}\int_0^{2\pi}\abs{\bu^{\SB}(s,\theta)-\bu^{\SB}_C(s)}^2 \mc{J}_\epsilon(s,\theta) d\theta ds \bigg)^{1/2} &\le c_\kappa\epsilon\abs{\log\epsilon}\norm{\bm{f}}_{C^1(\T)} \bigg(\int_{\T}\int_0^{2\pi} \mc{J}_\epsilon(s,\theta) d\theta ds \bigg)^{1/2} \\
&= c_\kappa\epsilon\abs{\log\epsilon} \norm{\bm{f}}_{C^1(\T)} \abs{\Gamma_\epsilon}^{1/2}.
\end{align*}

Using the above estimate along with \eqref{trace_err_bound}, we then have that the difference between the true fiber velocity ${\rm Tr}(\bu)(s)$ and the centerline approximation $\bu^{\SB}_C(s)$ satisfies
\begin{align*}
\norm{{\rm Tr}(\bu) - \bu^{\SB}_C}_{L^2(\T)} &= \frac{1}{\abs{\Gamma_\epsilon}^{1/2}}\norm{{\rm Tr}(\bu) - \bu^{\SB}_C}_{L^2(\Gamma_\epsilon)} \\
& \le \frac{1}{\abs{\Gamma_\epsilon}^{1/2}} \bigg(\norm{{\rm Tr}(\bu) - {\rm Tr}(\bu^{\SB})}_{L^2(\Gamma_\epsilon)} + \norm{{\rm Tr}(\bu^{\SB}) - \bu^{\SB}_C}_{L^2(\Gamma_\epsilon)} \bigg)  \\
&\le (\epsilon\abs{\log\epsilon}^{3/2}+\epsilon\abs{\log\epsilon}) c_\kappa \norm{\bm{f}}_{C^1(\T)}.
\end{align*}

\appendix
\section{Appendix}\label{appendix}
\subsection{Proof of Lemma \ref{lemmaorthonormal}}\label{moving_frame_pf}
Here we show the bound \eqref{kappa3} on the moving frame coefficient $\kappa_3$.  
\begin{proof}
Let $\be_t(s)$, $\widetilde\be_1(s)$, and $\widetilde\be_2(s)$ define a $C^1$ orthonormal frame satisfying
%

\begin{equation}\label{kappat12wh}
\frac{d}{d s}\begin{pmatrix} \be_t \\ \widetilde\be_1 \\ \widetilde\be_2 \end{pmatrix}
=\begin{pmatrix}
0 & \widetilde{\kappa}_1 & \widetilde{\kappa}_2 \\
-\widetilde{\kappa}_1 & 0 & \widetilde{\kappa}_3 \\
-\widetilde{\kappa}_2 & -\widetilde{\kappa}_3 & 0
\end{pmatrix}
\begin{pmatrix} \be_t \\ \widetilde\be_1 \\ \widetilde\be_2 \end{pmatrix}.
\end{equation}
Take
\[ \overline{\widetilde{\kappa}_3}=\int_0^1 \widetilde{\kappa}_3(s) \ts ds \]
and let $k$ be the closest integer to $\overline{\widetilde{\kappa}_3}/{2\pi}$. Define 
\begin{equation}\label{kappa3_def}
\kappa_3=\overline{\widetilde{\kappa}_3}-2\pi k; \qquad \varphi(s)=\int_0^s (\widetilde{\kappa}_3(\tau)-\kappa_3) \ts d\tau.
\end{equation}
Note that, by construction,
\[ |\kappa_3| \le \pi. \]
Define
\[\begin{pmatrix}
\be_{n_1}(s) \\ \be_{n_2}(s)
\end{pmatrix}
=\begin{pmatrix}
\cos\varphi(s) & -\sin\varphi(s)\\ \sin\varphi(s) &\cos\varphi(s)
\end{pmatrix}
\begin{pmatrix}
\widetilde{\be}_1(s) \\ \widetilde{\be}_2(s)
\end{pmatrix}. \]
Since $\varphi(1)=2\pi k$, $\be_{n_1}(s)$ and $\be_{n_2}(s)$ are both in $C^1(\T)$. It is also clear that $\be_t(s), \be_{n_1}(s)$ and $\be_{n_2}(s)$ define an orthonormal basis. A straightforward calculation shows that $\be_t(s), \be_{n_1}(s)$ and $\be_{n_2}(s)$ satisfy 
\eqref{moving_ODE} with $\kappa_3$ as in \eqref{kappa3_def} and
\[ \begin{pmatrix}
\kappa_1(s) \\ \kappa_2(s)
\end{pmatrix}
=\begin{pmatrix}
\cos\varphi(s) & -\sin\varphi(s)\\ \sin\varphi(s) &\cos\varphi(s)
\end{pmatrix}
\begin{pmatrix}
\widetilde{\kappa}_1(s) \\ \widetilde{\kappa}_2(s)
\end{pmatrix}. \]

\end{proof}

\subsection{Proof of $\epsilon$-dependence in well-posedness constants}\label{constants}
In this appendix, we prove the $\epsilon$-dependence claims for each of the inequalities stated in Section \ref{constants0}.

\subsubsection{Trace inequality}\label{trace_sec}
We begin by proving the trace inequality for $\A_\epsilon$ functions stated in Lemma \ref{Trace_inequality}. We show that the trace constant $c_T$ is proportional to $\abs{\log\epsilon}^{1/2}$. 

\begin{proof}[Proof of Lemma \ref{Trace_inequality}]
Since the fiber centerline is $C^2$ and the fiber does not self-intersect \eqref{non_intersecting}, we can cover the slender body by finitely many open neighborhoods $W_j$ where 
\[ W_j = \{ \X(s)+\rho\be_\rho(s,\theta) \ts : \ts 0\le \theta < 2\pi, \ts 0 \le \rho < r_{\max}/2, \ts a_j < s < b_j\}, \quad j=1,\dots,N <\infty.\]
Here $a_j$ and $b_j$ are chosen such that over each $W_j$, the fiber centerline can be considered as the graph of a $C^2$ function. Note that this choice of $a_j$ and $b_j$ depends only on the shape of the fiber centerline -- in particular, $\kappa_{\max}$ and $c_{\Gamma}$ -- and not on the fiber radius. \\

Then, using a partition of unity $\{\phi_j\}_{j=1}^N$ subordinate to the cover $\{W_j\}$, there exist $\epsilon$-independent $C^2$ diffeomorphisms $\psi_j$, $j=1,\dots,N$ taking the curvature $\kappa$ of the fiber centerline to zero on the set $W_j$ while leaving the radius $\epsilon$ intact. \\ 

\begin{figure}[!h]
\centering
\includegraphics[scale=0.5]{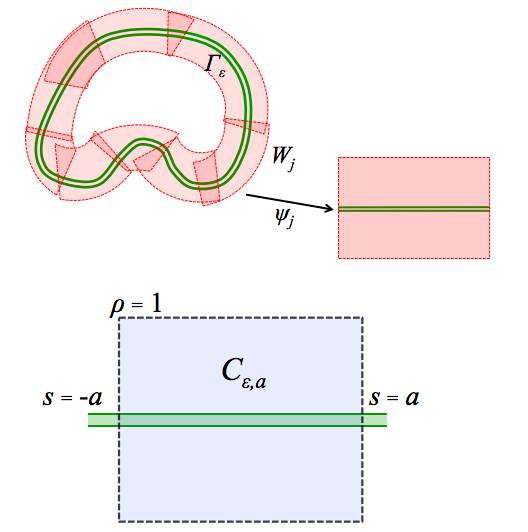}\\
\caption{ The slender body centerline can be straightened via $\epsilon$-independent diffeomorphisms $\psi_j$; thus it suffices to consider functions $\bu$ around a straight cylinder supported within the truncated cylindrical shell $C_{\epsilon,a}$.}
\label{fig:diffeo_SB}
\end{figure}

Let $D_{\rho}\subset \R^2$ denote the open disk of radius $\rho$ in $\R^2$ centered at the origin. Define the straight cylindrical surface $\Gamma_{\epsilon,a}:= \p D_{\epsilon}\times [-a,a]$ and the cylindrical shell $C_{\epsilon,a}:= (D_1\backslash \overline{D_{\epsilon}})\times [-a,a]$ for some $a<\infty$, parameterized in cylindrical coordinates $(\rho,\theta,s)$. We define the function space 
\[\A_S:= \big\{ \bv\in D^{1,2}(C_{\epsilon,a}) \ts : \ts \bv|_{\Gamma_{\epsilon,a}} = \bv(s); \ts \bv|_{\p C_{\epsilon,a}\backslash\Gamma_{\epsilon,a}} = 0 \big\}. \]

Then $\psi_j^*(\phi_j\bu)\circ \psi_j\in \A_S$, and to show Lemma \ref{Trace_inequality} it suffices to prove the $\abs{\log\epsilon}^{1/2}$ dependence in the trace constant about a straight cylinder.
\begin{lemma}\label{trace_straight_cylinder}
Let $\bu\in \A_S$. Then the $\theta$-independent trace of $\bu$ on the straight cylinder $\Gamma_{\epsilon,a}$ satisfies
\begin{equation}\label{cylinder_trace}
\|{\rm Tr}(\bu)\|_{L^2(-a,a)} \le \frac{1}{2\pi}\abs{\log\epsilon}^{1/2}\|\nabla \bu\|_{L^2(C_{\epsilon,a})}.
\end{equation}
\end{lemma}

\begin{proof}
We show the inequality \eqref{trace_straight_cylinder} for $\bu\in C^1(C_{\epsilon,a})\cap C^0(\overline{C_{\epsilon,a}})\cap \A_S$; the proof for $\bu\in \A_S$ then follows by density. \\

First note that for any $\bu\in C^1(C_{\epsilon,a})\cap C^0(\overline{C_{\epsilon,a}})$ and any $\bx = s\be_t +\epsilon\be_{\rho} + \theta\be_\theta \in \Gamma_{\epsilon,a}$, we may use the fundamental theorem of calculus to write
\[ \bu(s,\theta,\epsilon) = - \int_{\epsilon}^1 \frac{\p \bu}{\p \rho} \ts d\rho. \]
Then
\begin{align*}
|\bu(s,\theta,\epsilon)| &\le \int_{\epsilon}^{1} \bigg|\frac{\p \bu}{\p\rho} \bigg|\ts d\rho = \int_{\epsilon}^{1} \frac{1}{\sqrt{\rho}}\sqrt{\rho} \bigg|\frac{\p \bu}{\p\rho} \bigg|\ts d\rho \\
&\le \left(\int_{\epsilon}^{1} \frac{1}{\rho} \ts d\rho\right)^{\frac{1}{2}} \left(\int_{\epsilon}^{1} \bigg|\frac{\p \bu}{\p\rho} \bigg|^2 \ts \rho \ts d\rho\right)^{\frac{1}{2}} = \sqrt{|\log\epsilon|}\left(\int_{\epsilon}^{1} \bigg|\frac{\p \bu}{\p\rho} \bigg|^2 \ts \rho \ts d\rho\right)^{\frac{1}{2}}.
\end{align*}

Therefore ${\rm Tr}(\bu)$ obeys
\begin{equation}\label{surface_ineq}
\big|{\rm Tr}(\bu)\big|^2 \le |\log \epsilon| \int_{\epsilon}^1\bigg|\frac{\p \bu}{\p\rho} \bigg|^2 \ts \rho \ts d\rho.
\end{equation}

This holds for arbitrary $\bu \in C^1(C_{\epsilon,a})\cap C^0(\overline{C_{\epsilon,a}})$, but if $\bu$ also belongs to $\A_S$, by the $\theta$-independence of ${\rm Tr}(\bu)$, we have
\[\|{\rm Tr}(\bu)\|_{L^2(-a,a)}^2 = \frac{1}{2\pi}\int_{-a}^a \int_0^{2\pi} |{\rm Tr}(\bu)|^2 \ts d\theta \ts ds.\]

Then, using \eqref{surface_ineq}, we have that this $\bu$ satisfies
\begin{align*}
\|{\rm Tr}(\bu)\|_{L^2(-a,a)}^2 = \frac{1}{2\pi}\int_{-a}^a\int_0^{2\pi} |{\rm Tr}(\bu)|^2 \ts d\theta ds &\le \frac{1}{2\pi}|\log \epsilon| \int_{-a}^a\int_0^{2\pi}\int_{\epsilon}^1 \bigg| \frac{\p \bu}{\p \rho} \bigg|^2 \ts \rho \ts d\rho \ts d\theta\ts ds \\
& \le \frac{1}{2\pi}|\log \epsilon| \|\nabla \bu\|_{L^2(C_{\epsilon,a})}^2.
\end{align*}
\end{proof}

This estimate holds for $\bu$ defined around a straight cylinder; to return to a curved centerline, the diffeomorphisms $\psi_j^{-1}$ result in an additional constant on each set $W_j$  depending on $\psi_j$ but not $\epsilon$. Note that any constants arising from the use of cutoffs $\phi_j$ also depend on the Sobolev constant $c_S$ in $\Omega_\epsilon$, but by Lemma \ref{sobo_ineq}, $c_S$ is independent of $\epsilon$. \\

Summing over the neighborhoods $W_j$, we obtain the following trace inequality for any slender body $\Sigma_{\epsilon}$ satisfying the geometric constraints in Section \ref{geometric_constraints}:
\begin{equation}\label{trace_const0} 
\|{\rm Tr}(\bu)\|_{L^2(\T)} \le c_{\kappa}|\log\epsilon|^{1/2}\norm{\nabla\bu}_{L^2(\cup W_j)} \le c_{\kappa}|\log\epsilon|^{1/2}\norm{\nabla\bu}_{L^2(\Omega_\epsilon)},
\end{equation}
where $c_{\kappa}$ depends on the shape of the fiber centerline -- in particular, on the constants $\kappa_{\max}$ and $c_\Gamma$ -- but not on $\epsilon$. 
\end{proof}

\subsubsection{Extension operator}\label{extension}
The proof of the Korn inequality (Lemmas \ref{korn_eps}) essentially relies on the existence of a linear operator $T_{\epsilon}$ extending $\bu$ to the interior of the slender body such that $\E(T_{\epsilon}\bu)$ is bounded independent of $\epsilon$ as $\epsilon\to 0$. In this section we prove the existence of such an extension. In particular, we show the following lemma:
\begin{lemma}\emph{(Extension operator)}\label{extension_eps}
Let $\Omega_{\epsilon}=\R^3 \backslash \overline{\Sigma_{\epsilon}}$ be as in Section \ref{geometric_constraints}. For $\bu\in D^{1,2}(\Omega_{\epsilon})$, there exists a bounded linear operator $T_{\epsilon}: D^{1,2}(\Omega_{\epsilon})\to D^{1,2}(\R^3)$ extending $\bu$ to the interior of the slender body and satisfying 
\begin{enumerate}
\item $T_{\epsilon}\bu|_{\Omega_{\epsilon}} = \bu$ 
\item $\|\E(T_{\epsilon}\bu) \|_{L^2(\R^3)} \le c_E \| \E(\bu)\|_{L^2(\Omega_{\epsilon})}$, where the constant $c_E$ is independent of the slender body radius $\epsilon$ as $\epsilon\to 0$.
\end{enumerate}
\end{lemma}

Note that property 2 implies $\|T_{\epsilon}\bu\|_{D^{1,2}(\R^3)} \le \sqrt{2} c_E \|\bu\|_{D^{1,2}(\Omega_{\epsilon})}$, since
\begin{align*}
\|T_{\epsilon}\bu\|_{D^{1,2}(\R^3)} &= \|\nabla(T_{\epsilon}\bu)\|_{L^2(\R^3)} \le \sqrt{2} \|\E(T_{\epsilon}\bu) \|_{L^2(\R^3)} \le \sqrt{2} c_E \| \E(\bu)\|_{L^2(\Omega_{\epsilon})} \\
& \le 2\sqrt{2} c_E \|\nabla \bu \|_{L^2(\Omega_{\epsilon})} = 2\sqrt{2} c_E \| \bu \|_{D^{1,2}(\Omega_{\epsilon})} .
\end{align*}

In order to prove Lemma \ref{extension_eps}, we will need a few additional lemmas. The first is an important result from elasticity theory concerning the symmetric gradient. The proof can be found in \cite{duvaut1976inequalities}.  
\begin{lemma}\emph{(Rigid motion)}\label{rigid_motion}
Let $\Omega \subset \R^3$ be any domain. If $\bu: \Omega \to \R^3$ with $\nabla \bu \in L^2(\Omega)$ satisfies $\nabla \bu +(\nabla \bu)^{\rm T}=0$, then $\bu$ is a rigid body motion: $\bu(\bx) = \bm{A}\bx+{\bm b}$ for some constant, antisymmetric $\bm{A}\in \R^{3\times3}$ and constant ${\bm b}\in \R^3$. 
\end{lemma}

The fact that the symmetric gradient $\E(\cdot)$ exactly vanishes for rigid motions will be used repeatedly throughout the following construction. \\

Again, let $\D$ be a bounded, $C^2$ domain in $\R^3$. Let $H^1(\D)$ denote the Sobolev space $\{\bv\in L^2(\D) \ts :\ts \nabla\bv\in L^2(\D)\}$ with norm $\norm{\bv}_{H^1(\D)}^2:= \norm{\bv}_{L^2(\D)}^2+\norm{\nabla\bv}_{L^2(\D)}^2$. On $\D$, we define the space of rigid motions 
\[ \mathcal{R} = \{ \bv\in H^1(\D) \ts : \ts \bv = \bm{A}\bx + \bm{b} \text{ for some } \bm{A}= -\bm{A}^{\rm T} \in \R^{3\times 3} \text{ and } \bm{b}\in \R^3\}.\]
Note that $\sR$ is a closed subspace of $L^2(\D)$. For $\bu\in H^1(\D)$, let $P_{\mathcal{R}}\bu$ be the $L^2$ projection of $\bu$ onto the space of rigid motions, i.e. 
\[ P_{\mathcal{R}}\bu = \bv\in \mathcal{R} \text{ such that } \|\bu-\bv\|_{L^2(\D)} \le \|\bu - \bw\|_{L^2(\D)}\quad  \text{for all } \bw\in \mathcal{R}. \]

\begin{lemma}\emph{(Korn inequality for pure strain)}\label{korn_nonrigid}
Let $\D$ be a bounded Lipschitz domain and let $\sR$ be the space of rigid motions on $\D$. For any $\bw\in H^1(\D)$ with $\bw\perp \sR$ in $L^2$, the Korn inequality holds:
\[ \|\nabla\bw\|_{L^2(\D)} \le c\|\E(\bw)\|_{L^2(\D)}.\] 
\end{lemma}

\begin{proof}
The proof of Lemma \ref{korn_nonrigid} relies on the following Korn-type inequality for the bounded domain $\D$: 
\begin{equation}\label{korn_bdd}
\|\bu\|_{H^1(\D)} \le c (\|\mathcal{E}(\bu)\|_{L^2(\D)}+\|\bu\|_{L^2(\D)}).
\end{equation}
Since the domain dependence of the constant $c$ does not need to be specified in Lemma \ref{korn_nonrigid}, we refer to \cite{duvaut1976inequalities} for a proof of \eqref{korn_bdd}. \\

Now, assume Lemma \ref{korn_nonrigid} does not hold. Then there exists a sequence of functions $\{\bw_k\}\subset H^1(\D)$, $k=1,2,3,\dots$, such that $\bw_k\perp \sR$ and 
\[ \|\nabla\bw_k\|_{L^2(\D)} > k \|\E(\bw_k)\|_{L^2(\D)}.\] 
Without loss of generality, $\|\bw_k\|_{L^2(\D)}=1$, so by \eqref{korn_bdd},
\[\|\E(\bw_k)\|_{L^2(\D)} < \frac{1}{k}\|\nabla\bw_k\|_{L^2(\D)} \le \frac{1}{k}\|\bw_k\|_{H^1(\D)} \le \frac{c}{k}(\|\E(\bw_k)\|_{L^2(\D)}+1).\]
Taking $k$ sufficiently large (in particular, $k>c$), we have 
\[ \bigg(1-\frac{c}{k}\bigg)\|\E(\bw_k)\|_{L^2(\D)} < \frac{c}{k},\]
and thus $\|\E(\bw_k)\|_{L^2(\D)}\to 0$ as $k\to\infty$. Again by the inequality \eqref{korn_bdd}, 
\[ \|\bw_k\|_{H^1(\D)} \le c\bigg(\frac{c}{k-c}+1\bigg), \]
so there exists a subsequence $\{\bw_{k_j}\}$ such that $\bw_{k_j}\rightharpoonup \bw$ in $H^1$ for some $\bw\in H^1(\D)$. By Rellich compactness, $\bw_{k_j} \to \bw$ in $L^2$. Furthermore, $\liminf_k \norm{\E(\bw_{k_j})}_{L^2(\D)}\ge \norm{\E(\bw)}_{L^2(\D)}$, so $\E(\bw)=0$. Thus $\bw\in \sR$, but $\bw_k\perp\sR$ for all $k$, and $\bw_{k_j} \to \bw$ in $L^2$, so $\bw\equiv 0$. Thus $\bw_{k_j} \to 0$ in $L^2$, which contradicts $\|\bw_k\|_{L^2(\D)}=1$ for all $k$.
\end{proof}

\begin{remark}\label{korn_rmk}
Note that Lemma \ref{korn_nonrigid} remains true if we replace the orthogonality condition $\bw\perp \sR$ in $L^2(\D)$ with the condition that $\bw$ vanishes on an open set of $\p\D$ containing four points not in a plane. The proof is exactly as above, except that now the sequence $\bw_k\not\in\sR$ due to the vanishing condition on $\p\D$. In the last two lines we then conclude that the limit $\bw\in \sR$ but each $\bw_k\not\in\sR$, so $\bw\equiv 0$, yielding the same contradiction. Note that under the domain rescaling $\D\to \epsilon\D$, the constant in Lemma \ref{korn_nonrigid} remains unchanged.
\end{remark}

Using Lemma \ref{korn_nonrigid}, we can show the following inequality.
\begin{lemma}\emph{(Korn-Poincar\'e inequality)}\label{korn_poincare}
Let $\D$ be a bounded, Lipschitz domain in $\R^3$. For any $\bu\in H^1(\D)$, we have 
\begin{equation}\label{KP_ineq}
\| \bu - P_{\mathcal{R}}\bu\|_{L^2(\D)} \le c\| \E(\bu)\|_{L^2(\D)}
\end{equation}
for some constant $c>0$.
\end{lemma}

\begin{proof}
Assume that inequality \eqref{KP_ineq} does not hold. Then for each $k=1,2,3,\dots$ there exists a sequence $\{\bu_k\}\subset H^1(\D)$ such that
\[ \|\bu_k - P_{\sR}\bu_k \|_{L^2(\D)} > k\|\E(\bu_k)\|_{L^2(\D)}. \]
Define $\bw_k=\bu_k - P_{\sR}\bu_k$, so $\bw_k \perp \sR$ for each $k=1,2,3,\dots$ and $\E(\bw_k)=\E(\bu_k)$. Without loss of generality $\|\bw_k\|_{L^2(\D)} = 1$. Then
\[ 1= \|\bw_k \|_{L^2(\D)} > k\|\E(\bw_k)\|_{L^2(\D)},\] 
so $\|\E(\bw_k)\|_{L^2(\D)} <\frac{1}{k} \to 0$ as $k\to \infty$. Furthermore, since $\bw_k \perp \sR$ for each $k$, by the Korn inequality for pure strain (Lemma \ref{korn_nonrigid}) we have $\|\nabla \bw_k\|_{L^2(\D)} < \frac{c}{k}$. Thus $\bw_k$ is uniformly bounded in $H^1$ and there exists a subsequence $\{\bw_{k_l}\}$ such that $\bw_{k_l}\rightharpoonup \bw$ in $H^1$ for some $\bw\in H^1(\D)$. By compactness, $\bw_{k_l}\to \bw$ in $L^2$. Then, since $\liminf_k \|\E(\bw_k)\|_{L^2(\D)} \ge \|\E(\bw)\|_{L^2(\D)}$, we have that the limit $\bw$ satisfies $\E(\bw)=0$, so $\bw\in \sR$. But $\bw_{k_l}\to \bw$ in $L^2$ and $\bw_{k} \perp \sR$ for each $k$, so we must have $\bw\perp \sR$ as well. Thus $\bw\equiv 0$, so $\bw_{k_l} \to 0$ in $L^2$, which contradicts $\|\bw_{k_l} \|_{L^2(\D)}=1$. 
\end{proof}

Finally, we show an analogue of Lemma 3.1.2(1) in \cite{mazya1997differentiable}, adapted to use the symmetric gradient rather than the full gradient. 
\begin{lemma}\emph{(Extension-by-reflection scaling)}\label{extension_ineq}
Let $\D_1$, $\D_2$ be bounded $C^2$ domains in $\R^2$ with $\overline \D_2\subset \D_1$, and let $\D=\D_1\times[-1,1] \subset\R^3$ and $\D_H=(\D_1\backslash \overline \D_2)\times [-1,1]\subset\R^3$. For the rescaled domains $\D_{H,\epsilon}= \epsilon \D_H$, $\D_{\epsilon}=\epsilon\D$ ($\epsilon>0$), there exists a linear extension operator $T: H^1(\D_{H,\epsilon}) \to H^1(\D_{\epsilon})$ satisfying 
\begin{equation}\label{extension_symm}
\| T \bu \|_{L^2(\D_{\epsilon})} \le c\|\bu\|_{L^2(\D_{H,\epsilon})}
\end{equation}
as well as the estimate
\begin{equation}\label{extension_symm}
\| \E(T \bu) \|_{L^2(\D_{\epsilon})} \le c\bigg( \epsilon^{-1}\|\bu\|_{L^2(\D_{H,\epsilon})} + \|\E(\bu)\|_{L^2(\D_{H,\epsilon})} \bigg).
\end{equation}
\end{lemma}

\begin{proof}
For a function $\bv$ defined in the upper half-space $\R^3_+$, we recall the standard extension-by-reflection $E:\R^3_+\to \R^3$ across the boundary $x_3=0$ (see \cite{mazya1997differentiable} or \cite{evans2010pde}):
\[ E\bv(\bx) = \begin{cases}
\bv(\bx), & \bv\in \R^3_+ \\
\bv(x_1,x_2,-x_3) & \bv \not\in \R^3_+.
\end{cases} \]
For the domain-with-hole $\D_H\subset \R^3$, we cover a neighborhood of the inner boundary $\p\D_2\times [-1,1]$ with finitely many balls $B^H_i$, $i=1, \dots,N$, centered at points on $\p\D_2$, choosing the cover such that $\D_H\cap B^H_i$ can be mapped via $C^2$ diffeomorphism, denoted by $\Phi^{-1}_i$, to the half-ball $B\cap\R^3_+$, where $B$ is a ball in $\R^3$. We then choose open sets $U_j\subset\D_H$, $j=1,\dots,M$, such that $\{ B^H_i\} \cup \{U_j\}$ cover $\D_H$.  We define a partition of unity $\{\varphi_i\}\cup \{\varphi_j\}$ subordinate to this cover, and define the usual extension operator $T:\D_H\to \D$ by
\[T\bu = \sum_i \bigg(E\big((\varphi_i \bu)\circ \Phi_i\big)\bigg)\circ \Phi_i^{-1} + \sum_j \varphi_j \bu.\]

From this extension operator $T$, we can directly estimate $\|\E(T\bu)\|_{L^2(\D)}$. First, note that $\varphi_i \bu$ vanishes on $\p B_i^H \cap \D_H\subset \p(B_i^H\cap\D_H)$. Since $\p B_i^H$ is curved, we may use Remark \ref{korn_rmk} to estimate: 
\begin{align*}
\|\E(T\bu)\|_{L^2(\D)} &\le c\sum_i\norm{\nabla\bigg(\big(E\big((\varphi_i \bu)\circ \Phi_i\big)\big)\circ \Phi_i^{-1}\bigg) }_{L^2(\D)} \\
&\hspace{2cm} +\sum_j \|\varphi_j\E(\bu) \|_{L^2(\D_H)} + \sum_j \|\nabla \varphi_j \bu^{\rm T} \|_{L^2(\D_H)}\\
&\le c\sum_i\big\|\nabla(\varphi_i \bu)\big\|_{L^2(\D_H)} + \|\E(\bu)\|_{L^2(\D_H)} + c_{\phi}\|\bu\|_{L^2(\D_H)} \\
&\le c\sum_i\big\|\E(\varphi_i \bu) \big\|_{L^2(\D_H)} + \|\E(\bu)\|_{L^2(\D_H)} + c_{\phi}\|\bu\|_{L^2(\D_H)} \\
&\le c\big( \|\bu\|_{L^2(\D_H)} + \|\E(\bu)\|_{L^2(\D_H)} \big).
\end{align*}

The above inequality, coupled with a scaling argument ($\bx\to \epsilon\bx$) results in the desired $\epsilon$-dependent inequality \eqref{extension_symm}.
 \end{proof}
 
With Lemmas \ref{korn_poincare} and \ref{extension_ineq}, we are equipped to prove Lemma \ref{extension_eps}.

\begin{proof}[Proof of Lemma \ref{extension_eps}]
Let $D_{r}$ denote the disk in $\R^2$ of radius $r$. Using the diffeomorphisms $\psi_j$ defined in Section \ref{trace_sec}, it suffices to consider $\bu\in D^{1,2}((\R^2\backslash D_{\epsilon})\times \R)$ with supp$(\bu)\subset (\R^2\backslash D_{\epsilon})\times [-a,a]$ for $a<\infty$ and show that there exists an extension operator into the interior of the infinite cylinder $D_{\epsilon}\times\R \subset\R^3$ with symmetric gradient that is bounded independently of $\epsilon$ as $\epsilon\to 0$. \\

First we define 
\[ S_{\epsilon} = D_{2\epsilon}\times \R \quad \text{and} \quad G_{\epsilon} = (D_{2\epsilon}\backslash \overline{D_{\epsilon}}) \times \R \subset \R^3. \]
Since $\bu\in D^{1,2}((\R^2\backslash D_{\epsilon})\times \R)$ with supp$(\bu)\subset (\R^2\backslash D_{\epsilon})\times [-a,a]$, we have $\bu\in H^1(G_{\epsilon})$. We show that we can in fact construct a linear extension operator extending $\bu\in H^1(G_{\epsilon})$ to $H^1(S_{\epsilon})$ whose symmetric gradient is bounded independent of $\epsilon$. \\

Following \cite{mazya1997differentiable}, we begin by defining a cover $\{Q_j\}$ of $\R$:
\[ Q_j = \{s \in \R \ts:\ts |s-j| <1 \}, \quad j\in \Z. \] 
Let $\{\eta_j\}$ denote a smooth partition of unity subordinate to $Q_j$, where $\eta_j$ can be written as $\eta_j= \phi(s-j)$, translates of the same smooth cutoff function, such that $|\nabla \eta_j|\le c$ for each $j$. We define a sequence of cylinders and cylindrical layers 
\[ S_{2}^{(j)} = D_{2}\times Q_j \quad \text{and} \quad G^{(j)}_{2} = (D_{2}\backslash \overline{D_1}) \times Q_j \subset \R^3. \]
and set $S^{(j)}_{\epsilon} = \epsilon S^{(j)}_{2}$ and $G^{(j)}_{\epsilon} = \epsilon G^{(j)}_{2}$. Then by Lemma \ref{extension_ineq}, there exists a linear extension operator $T_{\epsilon}^{(j)}: H^1(G^{(j)}_{\epsilon}) \to H^1(S^{(j)}_{\epsilon})$ satisfying
 \begin{equation}\label{ext_est_seq1}
 \|\E (T_{\epsilon}^{(j)}\bu) \|_{L^2(S^{(j)}_{\epsilon})} \le c\left(\epsilon^{-1}\|\bu\|_{L^2(G^{(j)}_{\epsilon})} + \|\E(\bu)\|_{L^2(G^{(j)}_{\epsilon})} \right)
 \end{equation}
and
 \begin{equation}\label{ext_est_seq2}
  \| T_{\epsilon}^{(j)}\bu \|_{L^2(S^{(j)}_{\epsilon})} \le c\|\bu\|_{L^2(G^{(j)}_{\epsilon})}.
  \end{equation}
 
Let $P_{\sR}^{(j)}\bu$ denote the projection of $\bu\big|_{G^{(j)}_{\epsilon}}\in H^1(G^{(j)}_{\epsilon})$ onto $\sR$, the space of rigid motions on each $G^{(j)}_{\epsilon}$. Then, since $\E(\bw)= 0$ for any $\bw\in \sR$, we have 
 \[ \|\E(\bu - P_{\sR}^{(j)}\bu)\|_{L^2(G^{(j)}_{\epsilon})} =\|\E(\bu)\|_{L^2(G^{(j)}_{\epsilon})}. \]
By the Korn-Poincar\'e inequality (Lemma \ref{korn_poincare}) and a scaling argument we also have 
\begin{equation}\label{poincare_est}
  \|\bu - P_{\sR}^{(j)}\bu\|_{L^2(G^{(j)}_{\epsilon})} \le c\epsilon\|\E(\bu)\|_{L^2(G^{(j)}_{\epsilon})}.
\end{equation}

Since $P^{(j)}_{\sR}\bu\in \sR$ on each cylindrical shell $G_{\epsilon}^{(j)}$, we can write $P^{(j)}_{\sR}\bu= \bm{A}_j\bx+\bm{b}_j$ for $\bx\in G_{\epsilon}^{(j)}$. We then define the extension to each of the cylinders $S_{\epsilon}^{(j)}$ by 
\begin{equation}\label{PR_on_S}
\overline P_{\sR}^{(j)}\bu =\bm{A}_j\bx+ \bm{b}_j, \qquad \bx\in S_{\epsilon}^{(j)}.
\end{equation}
  
With these tools in mind, we now define an extension operator from the cylindrical shell $G_{\epsilon}$ to the cylinder $S_{\epsilon}$. We take 
 \begin{equation}\label{ext_operator}
  T_{\epsilon}\bu(\bx) = \bv(\bx)+ \bw(\bx) 
  \end{equation}
 where, for $\bx=\bx(\rho,\theta,s)\in S_{\epsilon}$ and $\bu_j = \bu|_{G^{(j)}_{\epsilon}}$, 
 \begin{align*}
 \bv(\rho,\theta,s) &= \sum_{j\in \Z} \eta_j(s/\epsilon)\bigg(\overline P_{\sR}^{(j)}\bu\bigg)(\bx) \\
 \bw(\rho,\theta,s) &= \sum_{j\in \Z} \eta_j(s/\epsilon)\left(T_{\epsilon}^{(j)}\big(\bu_j -P_{\sR}^{(j)}\bu\big)\right)(\bx).
 \end{align*}
 
Note that $T_{\epsilon}\bu \big|_{G_{\epsilon}}=\bu$. Furthermore, we show
\begin{equation}\label{ext_est_1}
 \|\E(T_{\epsilon}\bu)\|_{L^2(S_{\epsilon})} \le c\| \E(\bu)\|_{L^2(G_{\epsilon})} 
 \end{equation}
where the constant $c$ does not depend on $\epsilon$ as $\epsilon \to 0$. \\

We begin by estimating $\bv$. Let
\[\tilde Q_j = \{s\in\R \ts:\ts 0 <s-j<1\}, \quad j\in \Z.\]
Note that for each $j$ we have $\tilde Q_j\subset Q_j$ and $\tilde Q_j\subset Q_{j+1}$; in particular, $\eta_j(s)+\eta_{j+1}(s)=1$ on $\tilde Q_j$. Define
\[ \tilde S_{\epsilon}^{(j)} = \epsilon\left(D_{2}\times \tilde Q_j \right) \quad \text{and}\quad\tilde G_{\epsilon}^{(j)} = \epsilon\left((D_{2}\backslash \overline{D_1})\times \tilde Q_j \right).\]

On each $\tilde S_{\epsilon}^{(j)}$, $\bv$ can be rewritten as
\[ \bv(\rho,\theta,s) = \overline P_{\sR}^{(j)}\bu +\eta_{j+1}(s/\epsilon) (\overline P_{\sR}^{(j+1)}\bu - \overline P_{\sR}^{(j)}\bu). \]

Then, by the definition \eqref{PR_on_S}, we can bound the norm of $\overline P_{\sR}^{(j)}\bu$ on each cylinder $\tilde S_{\epsilon}^{(j)}$ by its norm over the shell $\tilde G_{\epsilon}^{(j)}$: $\|\overline P_{\sR}^{(j)}\bu\|_{L^2(\tilde S_{\epsilon}^{(j)})} \le c\|P_{\sR}^{(j)}\bu\|_{L^2(\tilde G_{\epsilon}^{(j)})}$. Using this, we bound the symmetric gradient of $\bv$: 
\begin{align*}
\|\E(\bv)\|_{L^2(\tilde S_{\epsilon}^{(j)})} &= \|\nabla \eta_{j+1}(s/\epsilon)(\overline P_{\sR}^{(j+1)}\bu - \overline P_{\sR}^{(j)}\bu)^{\rm T} +  (\overline P_{\sR}^{(j+1)}\bu - \overline P_{\sR}^{(j)}\bu)\nabla \eta_{j+1}(s/\epsilon)^{\rm T} \|_{L^2(\tilde S_{\epsilon}^{(j)})} \\
&\le c\epsilon^{-1}\|P_{\sR}^{(j+1)}\bu - P_{\sR}^{(j)}\bu \|_{L^2(\tilde G_{\epsilon}^{(j)})} \\
&\le c\epsilon^{-1}\left(\|\bu-P_{\sR}^{(j+1)}\bu\|_{L^2(G_{\epsilon}^{(j+1)})}+\|\bu-P_{\sR}^{(j)}\bu\|_{L^2(G_{\epsilon}^{(j)})} \right),
\end{align*}
 where in the last step we have used that $\tilde G_{\epsilon}^{(j)} \subset G_{\epsilon}^{(j+1)}$ and $\tilde G_{\epsilon}^{(j)} \subset G_{\epsilon}^{(j)}$. Finally, using \eqref{poincare_est}, we have
\[\|\E( \bv)\|_{L^2(\tilde S_{\epsilon}^{(j)})} \le c\left(\|\E(\bu)\|_{L^2(G_{\epsilon}^{(j)})}+\|\E(\bu)\|_{L^2(G_{\epsilon}^{(j+1)})}\right).\]

Summing over $j$, we then have
\[ \|\E(\bv)\|_{L^2(S_{\epsilon})} \le c \|\E(\bu)\|_{L^2(G_{\epsilon})} \]
where $c$ is bounded independent of $\epsilon$ as $\epsilon\to 0$. \\

We now bound the symmetric gradient of $\bw$. On each $\tilde S_{\epsilon}^{(j)}$ we have
\begin{align*}
 \|\E(\bw)\|_{L^2(\tilde S_{\epsilon}^{(j)})} &\le  \|\E\big(T_{\epsilon}^{(j)}(\bu_j -P_{\sR}^{(j)}\bu)\big) \|_{L^2(\tilde S_{\epsilon}^{(j)})}+  \|\E\big(T_{\epsilon}^{(j+1)}(\bu_{j+1} -P_{\sR}^{(j+1)}\bu)\big) \|_{L^2(\tilde S_{\epsilon}^{(j)})} \\
 &\quad + 2c\epsilon^{-1} \|T_{\epsilon}^{(j)}(\bu_j-P_{\sR}^{(j)}\bu)\|_{L^2(\tilde S_{\epsilon}^{(j)})} + 2c\epsilon^{-1} \|T_{\epsilon}^{(j+1)}(\bu_{j+1}-P_{\sR}^{(j+1)}\bu)\|_{L^2(\tilde S_{\epsilon}^{(j)})}. 
 \end{align*}

Using the inequalities \eqref{ext_est_seq1}, \eqref{ext_est_seq2}, and \eqref{poincare_est}, we have
\begin{align*}
\|T_{\epsilon}^{(j)}(\bu_j -P_{\sR}^{(j)}\bu)\|_{L^2(\tilde S_{\epsilon}^{(j)})} &\le c \|\bu_j -P_{\sR}^{(j)}\bu\|_{L^2(\tilde G_{\epsilon}^{(j)})} \le c\epsilon\|\E(\bu)\|_{L^2(\tilde G_{\epsilon}^{(j)})}
\end{align*}
and
\begin{align*}
\|\E\big(T_{\epsilon}^{(j)}(\bu_j-P_{\sR}^{(j)}\bu)\big)\|_{L^2(\tilde S_{\epsilon}^{(j)})} &\le c\left(\epsilon^{-1}\|\bu_j-P_{\sR}^{(j)}\bu\|_{L^2(\tilde G_{\epsilon}^{(j)})}+\|\E\big(\bu_j -P_{\sR}^{(j)}\bu)\|_{L^2(\tilde G_{\epsilon}^{(j)}\big)} \right) \\
&\le c\| \E(\bu) \|_{L^2(\tilde G_{\epsilon}^{(j)})} .
\end{align*}

Summing over $j$, we have
\[ \|\E(\bw)\|_{L^2(S_{\epsilon})} \le c\| \E(\bu) \|_{L^2(G_{\epsilon})}. \]

Therefore the extension operator $T_{\epsilon}: G_{\epsilon}\to S_{\epsilon}$ \eqref{ext_operator} is bounded independent of $\epsilon$ as $\epsilon\to 0$. Defining $T_{\epsilon}\bu=\bu$ in $\R^3\backslash S_{\epsilon}$ then gives the desired extension on all of $\R^3$. 
\end{proof}

\subsubsection{Korn inequality}\label{korn_proof}
Using the extension operator $T_\epsilon$ defined in Section \ref{extension}, we can now prove $\epsilon$-independence of the Korn constant (Lemma \ref{korn_eps}). \\

We first note that the proof of the Korn inequality for function in $D^{1,2}(\R^3)$ is very simple. We first consider $\bv\in C_0^{\infty}(\R^3)$, then take the closure to show the result for $D^{1,2}(\R^3)$. We have that $\bv\in C_0^{\infty}(\R^3)$ satisfies 
\begin{align*}
\int_{\R^3} |\mathcal{E}(\bv)|^2 \ts d\bx &= \int_{\R^3}\bigg( \frac{1}{2}|\nabla \bv|^2 + \frac{1}{2}\nabla\bv:(\nabla\bv)^{\rm T}\bigg) \ts d\bx = \int_{\R^3} \frac{1}{2}|\nabla \bv|^2 \ts d\bx - \frac{1}{2}\int_{\R^3} \bv \cdot \nabla( \dive \ts\bv) \ts d\bx \\
&= \int_{\R^3} \frac{1}{2}|\nabla \bv|^2 \ts d\bx + \frac{1}{2}\int_{\R^3} |\dive \ts \bv|^2 \ts d\bx \ge \int_{\R^3} \frac{1}{2}|\nabla \bv|^2 \ts d\bx,
\end{align*}
where we have used integration by parts twice, as well as the fact that $\bv$ vanishes at $\infty$. \\

Now, using the extension operator $T_{\epsilon}$ established in Lemma \ref{extension_eps} to extend $\bu\in D^{1,2}(\Omega_{\epsilon})$ to all of $\R^3$, the proof of the Korn inequality (Lemma \ref{korn_eps}) is immediate.

\begin{proof}[Proof of Lemma \ref{korn_eps}]
Let $\bu \in D^{1,2}(\Omega_{\epsilon})$ and let $T_\epsilon\bu$ be the extension of $\bu$ to $\R^3$ defined in Lemma \ref{extension_eps}. Using properties of the extension operator $T_{\epsilon}$ and the Korn inequality on $\R^3$, we then have
\begin{align*}
\|\nabla \bu\|_{L^2(\Omega_{\epsilon})} &\le \|\nabla (T_{\epsilon}\bu)\|_{L^2(\R^3)} \le \sqrt{2}\|T_{\epsilon}\mathcal{E}(\bu)\|_{L^2(\R^3)} \le \sqrt{2}c_E \| \mathcal{E}(\bu) \|_{L^2(\Omega_{\epsilon})}.
\end{align*}
 Taking $c_K=\sqrt{2}c_E$, we obtain \eqref{korn_ineq}. 
\end{proof}


\subsubsection{Sobolev inequality}\label{Sob_ineq}
Using the extension operator defined in Section \ref{extension}, we also immediately obtain the $\epsilon$-independence of the Sobolev inequality stated in Lemma \ref{sobo_ineq}.
\begin{proof}[Proof of Lemma \ref{sobo_ineq}]
We have
\begin{align*}
\| \bu\|_{L^6(\Omega_{\epsilon})} &\le \| T_{\epsilon}\bu\|_{L^6(\R^3)} \le  c_R\| \nabla (T_{\epsilon}\bu)\|_{L^2(\R^3)} \\
&\le  c_R c_{E}\| \nabla \bu\|_{L^2(\Omega_{\epsilon})}, \quad \text{by Lemma \ref{extension_eps},}
\end{align*}
where $c_R$ is the constant in the Sobolev inequality on $\R^3$. Taking $c_S=c_Rc_E$, we obtain the desired result. 
\end{proof}


\subsubsection{Pressure estimate}\label{pressure_const}
Finally, we prove the $\epsilon$-independence claim for the problem $\dive\ts\bv=p$ stated in Lemma \ref{divv_p_lem}. The proof closely follows \cite{galdi2011introduction}, Chapter III.3, with additional attention paid to the domain dependence.

\begin{proof}[Proof of Lemma \ref{divv_p_lem}]
We begin by taking a sequence $\{p_m\}\subset C_0^{\infty}(\Omega_{\epsilon})$ approximating $p$ in $L^2(\Omega_{\epsilon})$. For each $m\in \N$, let $\psi_m$ be the solution to the Poisson problem $\Delta\psi_m = \overline{p_m}$ in $\R^3$, where $\overline{p_m}$ denotes the extension by zero of $p_m$ to the interior of $\Sigma_{\epsilon}$; i.e. to all of $\R^3$. Then by standard solution theory for the Poisson problem (\cite{galdi2011introduction}, Chapter II.11), we have the estimate
\begin{equation}\label{poisson_est}
 \|\nabla^2\psi_m\|_{L^2(\Omega_{\epsilon})} \le \|\nabla^2\psi_m\|_{L^2(\R^3)} \le c_q\|\overline{p_m}\|_{L^2(\R^3)} = c_q\|p_m\|_{L^2(\Omega_{\epsilon})} 
\end{equation}
where $\nabla^2$ denotes the matrix of second partial derivatives and the constant $c_q$ is independent of $\epsilon$. \\

We define
\[ \bv_m :=\nabla \psi_m+\bw_m \]
where $\bw_m\in D^{1,2}(\Omega_{\epsilon})$ is supported only within the neighborhood $\mathcal{O}$ \eqref{region_O} of $\Gamma_{\epsilon}$, and serves to correct for $\nabla \psi_m\neq 0$ on $\Gamma_{\epsilon}$. To this end, $\bw_m$ can be considered as a function in $H^1(\mathcal{O})$ satisfying
\begin{equation}\label{w_equation}
\begin{aligned}
\dive\ts\bw_m &= 0 \quad \text{in }\mathcal{O} \\
\bw_m &= - \nabla \psi_m \quad \text{on }\Gamma_{\epsilon} \\
\bw_m &=0 \quad \text{on } \p \mathcal{O} \backslash\Gamma_{\epsilon}, 
\end{aligned}
\end{equation}
which is then extended by zero to all of $\Omega_{\epsilon}$. For each $m\in \N$, such a function $\bw_m$ exists since $\Delta \psi_m=0$ within $\Sigma_{\epsilon}$ and therefore
\[ \int_{\Gamma_{\epsilon}} \nabla \psi_m\cdot{\bm n}=0. \]
A solution to \eqref{w_equation} can be constructed by considering the function ${\bm \Psi}_m = - \phi\nabla \psi_m$ where $\phi\in C^{\infty}(\Omega_{\epsilon})$ is a cutoff function satisfying 
\[ \phi(\rho)=\begin{cases}
1, & \rho \le r_{\max}/2 \\
0 & \rho > r_{\max}.
\end{cases} \] 
Then by \cite{galdi2011introduction}, Theorem III.3.1, there exists a solution $\bw_m-{\bm \Psi}_m\in H^1_0(\mathcal{O})$ satisfying
\begin{equation}\label{new_w_equation}
\begin{aligned}
\dive(\bw_m-{\bm \Psi}_m) &= -\dive \ts {\bm \Psi}_m \quad \text{in }\mathcal{O}; \\
\|\nabla(\bw_m-{\bm \Psi}_m)\|_{L^2(\mathcal{O})} &\le c_B\|\dive \ts {\bm \Psi}_m\|_{L^2(\mathcal{O})}.
\end{aligned}
\end{equation}
Since the slender body surface $\Gamma_{\epsilon}$ satisfies the geometric constraints in Section \ref{geometric_constraints}, the region $\mathcal{O}$ satisfies an interior sphere condition with uniform radius $r_{\max}/2$. Then $\mathcal{O}$ can be considered as the infinite union of balls of radius $r_{\max}/2$. Following the construction in the proof of Lemma 2, Chapter 1.1.9 of \cite{maz2013sobolev}, there exist a finite number of domains $\mathcal{O}_k$, star-shaped with respect to balls of radius $r_{\max}/4$, such that
\[\mathcal{O} = \bigcup_{k=1}^N \mathcal{O}_k. \]
Here $N$ depends only on $\kappa_{\max}$ and $c_\Gamma$. Then the domain dependence of the constant $c_B$ in \eqref{new_w_equation} has an explicit formula (\cite{galdi2011introduction}, equation III.3.27): 
\[ c_B \le c_0 \bigg(\frac{\delta(\mathcal{O})}{r_{\max}} \bigg)^3\bigg(1+ \frac{\delta(\mathcal{O})}{r_{\max}} \bigg) \]
where $\delta(\mathcal{O})$ is the diameter of the region $\mathcal{O}$ and $c_0$ depends on the diameter of the domains $\mathcal{O}_k$, each of which are bounded independent of $\epsilon$ as $\epsilon\to 0$. \\

Then, from \eqref{new_w_equation}, we have
\begin{equation}\label{w_est1}
\begin{aligned}
\|\nabla\bw_m\|_{L^2(\Omega_{\epsilon})} &\le c_B\|\dive\ts {\bm \Psi}_m\|_{L^2(\Omega_{\epsilon})} + \|\nabla{\bm \Psi}_m\|_{L^2(\Omega_{\epsilon})} \\
&= c_B\|\dive(\phi\nabla \psi_m) \|_{L^2(\Omega_{\epsilon})} + \|\nabla(\phi\nabla \psi_m)\|_{L^2(\Omega_{\epsilon})} .
\end{aligned}
\end{equation}

Therefore, using \eqref{poisson_est} and \eqref{w_est1}, we have
\begin{align*}
\|\nabla \bw_m\|_{L^2(\Omega_{\epsilon})} &\le (c_B+1)(c_q\|p_m\|_{L^2(\Omega_{\epsilon})}+ c_{\phi}\|\nabla \psi_m\|_{L^2(\mathcal{O})}), 
\end{align*}
where $c_{\phi}$ depends on $\nabla\phi$ but is independent of $\epsilon$. We then use the Sobolev inequality on $\R^3$ to obtain 
\begin{align*}
 \|\nabla \psi_m\|_{L^2(\mathcal{O})} &\le  |\mathcal{O}|^{1/3} \|\nabla\psi_m\|_{L^6(\mathcal{O})} \le |\mathcal{O}|^{1/3} \|\nabla\psi_m\|_{L^6(\Omega_{\epsilon})} \\
 &\le |\mathcal{O}|^{1/3} c_S\|\nabla^2\psi_m\|_{L^2(\Omega_{\epsilon})}  \le |\mathcal{O}|^{1/3} c_S c_q \|p_m\|_{L^2(\Omega_{\epsilon})}, \quad \text{using }\eqref{poisson_est}.
 \end{align*}
Now, $|\mathcal{O}|\le c_{\kappa}r_{\max}^2$ is bounded independent of $\epsilon$, and by Lemma \ref{sobo_ineq} the Sobolev constant $c_S$ is independent of $\epsilon$. Thus
\[ \|\nabla \bw_m\|_{L^2(\Omega_{\epsilon})} \le c_W\|p_m\|_{L^2(\Omega_{\epsilon})} \]
for $c_W$ independent of $\epsilon$, and 
\[ \|\nabla \bv_m\|_{L^2(\Omega_{\epsilon})} \le \|\nabla^2\psi_m\|_{L^2(\Omega_{\epsilon})}+ \|\nabla \bw_m\|_{L^2(\Omega_{\epsilon})} \le (c_q+c_W)\|p_m\|_{L^2(\Omega_{\epsilon})}. \]

 Passing to the limit we obtain the desired solution to the $\dive \ts\bv=p$ problem of Lemma \eqref{divv_p_lem}, as the constant $c_P=c_q+c_W$ is independent of $\epsilon$.
 \end{proof}

\bibliography{SBT_bib.bib}{}

\begin{thebibliography}{10}

\bibitem{antman2005nonlinear}
S.~Antman.
\newblock {\em Nonlinear problems of elasticity, 2nd edition}, volume 107.
\newblock Springer, 2005.

\bibitem{avron2008geometric}
J.~Avron and O.~Raz.
\newblock A geometric theory of swimming: {P}urcell's swimmer and its
  symmetrized cousin.
\newblock {\em New J. Phys.}, 10(6):063016, 2008.

\bibitem{batchelor1970slender}
G.~Batchelor.
\newblock Slender-body theory for particles of arbitrary cross-section in
  {S}tokes flow.
\newblock {\em J. Fluid Mech.}, 44(3):419--440, 1970.

\bibitem{becker2003self}
L.~E. Becker, S.~A. Koehler, and H.~A. Stone.
\newblock On self-propulsion of micro-machines at low {R}eynolds number:
  {P}urcell's three-link swimmer.
\newblock {\em J. Fluid Mech.}, 490:15--35, 2003.

\bibitem{bishop1975there}
R.~L. Bishop.
\newblock There is more than one way to frame a curve.
\newblock {\em The American Mathematical Monthly}, 82(3):246--251, 1975.

\bibitem{bogovskii1980solutions}
M.~Bogovski{\i}.
\newblock Solutions of some problems of vector analysis, associated with the
  operators div and grad.
\newblock {\em Theory of cubature formulas and the application of functional
  analysis to problems of mathematical physics}, 1980:5--40, 1980.

\bibitem{bouzarth2011modeling}
E.~L. Bouzarth and M.~L. Minion.
\newblock Modeling slender bodies with the method of regularized {S}tokeslets.
\newblock {\em J. Comput. Phys.}, 230(10):3929--3947, 2011.

\bibitem{bringley2008validation}
T.~T. Bringley and C.~S. Peskin.
\newblock Validation of a simple method for representing spheres and slender
  bodies in an immersed boundary method for {S}tokes flow on an unbounded
  domain.
\newblock {\em J. Comput. Phys.}, 227(11):5397--5425, 2008.

\bibitem{buchmann2015flow}
A.~L. Buchmann, L.~J. Fauci, K.~Leiderman, E.~M. Strawbridge, and L.~Zhao.
\newblock Flow induced by bacterial carpets and transport of microscale loads.
\newblock In {\em Applications of Dynamical Systems in Biology and Medicine},
  pages 35--53. Springer, 2015.

\bibitem{childress1981mechanics}
S.~Childress.
\newblock {\em Mechanics of swimming and flying}, volume~2.
\newblock Cambridge University Press, 1981.

\bibitem{cortez2005method}
R.~Cortez, L.~Fauci, and A.~Medovikov.
\newblock The method of regularized {S}tokeslets in three dimensions: analysis,
  validation, and application to helical swimming.
\newblock {\em Phys. Fluids (1994-present)}, 17(3):031504, 2005.

\bibitem{cortez2012slender}
R.~Cortez and M.~Nicholas.
\newblock Slender body theory for {S}tokes flows with regularized forces.
\newblock {\em Commun. Appl. Math. Comput. Sci.}, 7(1):33--62, 2012.

\bibitem{cox1970motion}
R.~Cox.
\newblock The motion of long slender bodies in a viscous fluid part 1. general
  theory.
\newblock {\em J. Fluid Mech.}, 44(4):791--810, 1970.

\bibitem{dreyfus2005microscopic}
R.~Dreyfus, J.~Baudry, M.~L. Roper, M.~Fermigier, H.~A. Stone, and J.~Bibette.
\newblock Microscopic artificial swimmers.
\newblock {\em Nature}, 437(7060):862--865, 2005.

\bibitem{duvaut1976inequalities}
G.~Duvaut and J.~Lions.
\newblock Inequalities in mechanics and physics, a series of comprehensive
  studies in mathematics, no. 219, 1976.

\bibitem{evans2010pde}
L.~Evans.
\newblock {\em Partial differential equations, second edition}.
\newblock Graduate Studies in Mathematics, 19. American Mathematical Society,
  Providence, RI, 2010.

\bibitem{fan1998direct}
X.~Fan, N.~Phan-Thien, and R.~Zheng.
\newblock A direct simulation of fibre suspensions.
\newblock {\em J. Non-Newton. Fluid Mech.}, 74(1):113--135, 1998.

\bibitem{galdi2011introduction}
G.~P. Galdi.
\newblock {\em An introduction to the mathematical theory of the
  {N}avier-{S}tokes equations: Steady-state problems}.
\newblock Springer Science \& Business Media, 2011.

\bibitem{goriely1997nonlinear}
A.~Goriely and M.~Tabor.
\newblock Nonlinear dynamics of filaments. {III}. {I}nstabilities of helical
  rods.
\newblock In {\em Proc. R. Soc. Lond. A}, volume 453, pages 2583--2601. The
  Royal Society, 1997.

\bibitem{gotz2000interactions}
T.~G{\"o}tz.
\newblock {\em Interactions of fibers and flow: asymptotics, theory and
  numerics}.
\newblock Doctoral dissertation, TU Kaiserslautern, 2000.

\bibitem{gueron1997cilia}
S.~Gueron, K.~Levit-Gurevich, N.~Liron, and J.~J. Blum.
\newblock Cilia internal mechanism and metachronal coordination as the result
  of hydrodynamical coupling.
\newblock {\em Proc. Natl. Acad. Sci.}, 94(12):6001--6006, 1997.

\bibitem{hamalainen2011papermaking}
J.~H{\"a}m{\"a}l{\"a}inen, S.~B. Lindstr{\"o}m, T.~H{\"a}m{\"a}l{\"a}inen, and
  H.~Niskanen.
\newblock Papermaking fibre-suspension flow simulations at multiple scales.
\newblock {\em J. Engrg. Math.}, 71(1):55--79, 2011.

\bibitem{hancock1953self}
G.~Hancock.
\newblock The self-propulsion of microscopic organisms through liquids.
\newblock {\em Proc. R. Soc. Lond. A}, 217(1128):96--121, 1953.

\bibitem{johnson1980improved}
R.~E. Johnson.
\newblock An improved slender-body theory for {S}tokes flow.
\newblock {\em J. Fluid Mech.}, 99(02):411--431, 1980.

\bibitem{keller1976slender}
J.~B. Keller and S.~I. Rubinow.
\newblock Slender-body theory for slow viscous flow.
\newblock {\em J. Fluid Mech}, 75(4):705--714, 1976.

\bibitem{koens2018boundary}
L.~Koens and E.~Lauga.
\newblock The boundary integral formulation of {S}tokes flows includes
  slender-body theory.
\newblock {\em J. Fluid Mech.}, (850):1--12.

\bibitem{lauga2009hydrodynamics}
E.~Lauga and T.~R. Powers.
\newblock The hydrodynamics of swimming microorganisms.
\newblock {\em Rep. Progr. Phys.}, 72(9):096601, 2009.

\bibitem{lighthill1975mathematical}
J.~Lighthill.
\newblock {\em Mathematical biofluid dynamics}, volume~17.
\newblock SIAM, 1975.

\bibitem{lim2008dynamics}
S.~Lim, A.~Ferent, X.~S. Wang, and C.~S. Peskin.
\newblock Dynamics of a closed rod with twist and bend in fluid.
\newblock {\em SIAM J. Sci. Comput.}, 31(1):273--302, 2008.

\bibitem{lim2004simulations}
S.~Lim and C.~S. Peskin.
\newblock Simulations of the whirling instability by the immersed boundary
  method.
\newblock {\em SIAM J. Sci. Comput.}, 25(6):2066--2083, 2004.

\bibitem{maz2013sobolev}
V.~Maz'ya.
\newblock {\em Sobolev spaces}.
\newblock Springer, 2013.

\bibitem{mazya1997differentiable}
V.~G. Maz'ya and S.~V. Poborchi.
\newblock {\em Differentiable functions on bad domains}.
\newblock World Scientific, 1997.

\bibitem{nguyen2011action}
H.~Nguyen, R.~Ortiz, R.~Cortez, and L.~Fauci.
\newblock The action of waving cylindrical rings in a viscous fluid.
\newblock {\em J. Fluid Mech.}, 671:574--586, 2011.

\bibitem{petrie1999rheology}
C.~J. Petrie.
\newblock The rheology of fibre suspensions.
\newblock {\em J. Non-Newton. Fluid Mech.}, 87(2):369--402, 1999.

\bibitem{pozrikidis1992boundary}
C.~Pozrikidis.
\newblock {\em Boundary integral and singularity methods for linearized viscous
  flow}.
\newblock Cambridge University Press, 1992.

\bibitem{rodenborn2013propulsion}
B.~Rodenborn, C.-H. Chen, H.~L. Swinney, B.~Liu, and H.~Zhang.
\newblock Propulsion of microorganisms by a helical flagellum.
\newblock {\em Proc. Natl. Acad. Sci.}, 110(5):E338--E347, 2013.

\bibitem{saintillan2011emergence}
D.~Saintillan and M.~J. Shelley.
\newblock Emergence of coherent structures and large-scale flows in motile
  suspensions.
\newblock {\em Journal of the Royal Society Interface}, page rsif20110355,
  2011.

\bibitem{sellier1999stokes}
A.~Sellier.
\newblock Stokes flow past a slender particle.
\newblock In {\em Proc. R. Soc. Lond. A}, volume 455, pages 2975--3002. The
  Royal Society, 1999.

\bibitem{shelley2016dynamics}
M.~J. Shelley.
\newblock The dynamics of microtubule/motor-protein assemblies in biology and
  physics.
\newblock {\em Annual Review of Fluid Mechanics}, 48:487--506, 2016.

\bibitem{shelley2000stokesian}
M.~J. Shelley and T.~Ueda.
\newblock The {S}tokesian hydrodynamics of flexing, stretching filaments.
\newblock {\em Physica D: Nonlinear Phenomena}, 146(1):221--245, 2000.

\bibitem{smith2007discrete}
D.~Smith, E.~Gaffney, and J.~Blake.
\newblock Discrete cilia modelling with singularity distributions: application
  to the embryonic node and the airway surface liquid.
\newblock {\em Bull. Math. Biol.}, 69(5):1477--1510, 2007.

\bibitem{smith2011mathematical}
D.~J. Smith, A.~A. Smith, and J.~R. Blake.
\newblock Mathematical embryology: the fluid mechanics of nodal cilia.
\newblock {\em J. Engrg. Math.}, 70(1-3):255--279, 2011.

\bibitem{spagnolie2011comparative}
S.~E. Spagnolie and E.~Lauga.
\newblock Comparative hydrodynamics of bacterial polymorphism.
\newblock {\em Phys. Rev. Lett.}, 106(5):058103, 2011.

\bibitem{tornberg2006numerical}
A.-K. Tornberg and K.~Gustavsson.
\newblock A numerical method for simulations of rigid fiber suspensions.
\newblock {\em J. Comput. Phys.}, 215(1):172--196, 2006.

\bibitem{tornberg2004simulating}
A.-K. Tornberg and M.~J. Shelley.
\newblock Simulating the dynamics and interactions of flexible fibers in
  {S}tokes flows.
\newblock {\em J. Comput. Phys.}, 196(1):8--40, 2004.

\end{thebibliography}
\bibliographystyle{abbrv}


\end{document}